\documentclass[11pt]{amsart}
\usepackage{amsmath, amsthm, amssymb}
\usepackage[margin=2cm]{geometry}
\usepackage{epsfig}
\usepackage{graphicx}
\usepackage{rawfonts}
\usepackage{enumerate}
\usepackage{amssymb}
\usepackage{float} 
\usepackage{graphicx}
\usepackage{pgf,tikz,pgfplots}
\usepackage{mathrsfs}
\usetikzlibrary{arrows}
\usepackage{amsmath}
\usepackage{amsfonts}
\usepackage{amssymb}
\usepackage{amsthm}
\usepackage{graphicx}
\usepackage{booktabs}
\usepackage{caption}
\usepackage{listings}
\usepackage{setspace}
\usepackage[mathscr]{eucal}
\usepackage{pgfplots}
\usepackage{hyperref}
\usepackage{wrapfig}
\usepackage{floatflt,epsfig}
\usepackage{ dsfont }
\usepackage{amscd}
\usepackage{tikz-cd}
\usepackage{fancyhdr}
\usepackage[all]{xy}
\usepackage{latexsym}
\usepackage{amscd}
\usepackage{pifont}
\usepackage{eufrak}
\usepackage{subfig}
\usepackage{easyReview}

\def\Implies{\ifmmode\Longrightarrow \else
	\unskip${}\Longrightarrow{}$\ignorespaces\fi}
\def\implies{\ifmmode\Rightarrow \else
	\unskip${}\Rightarrow{}$\ignorespaces\fi}
\def\iff{\ifmmode\Longleftrightarrow \else
	\unskip${}\Longleftrightarrow{}$\ignorespaces\fi}

\let\:=\colon

\newcommand{\intt}{\mathop{\rm int}\nolimits}

\newcommand{\rank}{\mathop{\rm rank}\nolimits}

\newcommand{\numberset}{\mathbb}

\newcommand{\Z}{\numberset{Z}}

\def\ZZ{{\mathbb Z}}

\newcommand{\lt}{\mathop{\rm in}\nolimits}

\newcommand{\cC}{\mathcal{C}}

\newcommand{\cP}{\mathcal{P}}
\newcommand{\cH}{\mathcal{H}}

\newcommand{\cB}{\mathcal{B}}

\newcommand{\cF}{\mathcal{F}}

\newcommand{\cR}{\mathcal{R}}
\newcommand{\cS}{\mathcal{S}}
\newcommand{\cT}{\mathcal{T}}

\theoremstyle{plain}

\theoremstyle{theorem}

\newtheorem{defn}{Definition}[section]
\newtheorem{prop}[defn]{Proposition}
\newtheorem{thm}[defn]{Theorem}
\newtheorem{lemma}[defn]{Lemma}

\newtheorem{coro}[defn]{Corollary}
\newtheorem{exa}[defn]{Example}

\newtheorem{qst}[defn]{Question}

\theoremstyle{remark}

\begin{document}
	\title[On the rook polynomial of grid polyominoes]{On the rook polynomial of grid polyominoes}
 \author {Rodica Dinu, FRANCESCO NAVARRA}

	\address{
		University of Konstanz, Fachbereich Mathematik und Statistik, Fach D 197 D-78457 Konstanz, Germany, and Institute of Mathematics Simion Stoilow of the Romanian Academy, Calea Grivitei 21, 010702, Bucharest, Romania}
	\email{rodica.dinu@uni-konstanz.de}
	
	\address{Sabanci University, Faculty of Engineering and Natural Sciences, Orta Mahalle, Tuzla 34956, Istanbul, Turkey}
	\email{francesco.navarra@sabanciuniv.edu}

	

 \begin{abstract}
 Grid polyominoes form a class of thin polyominoes with one or more holes arranged in a grid-like pattern in the plane. In this paper, we prove that the rook polynomial of grid polyominoes coincides with the $h$-polynomial of their corresponding coordinate ring. Our approach is based on the theory of simplicial complexes and extends previous results for frame polyominoes, which are special cases of polyominoes with exactly one hole.
 \end{abstract}

\maketitle

\section*{Introduction}

Polyominoes are finite collections of unit squares joined edge to edge. While their study in combinatorics dates back to ancient times, Golomb formally defined them in 1953. His later monograph~\cite{golomb} provides an excellent overview of polyomino combinatorics and tiling problems. A polyomino can essentially be viewed as a pruned chessboard on which chess pieces can be placed. The placement of non-attacking rooks on a skew diagram corresponds to the enumeration of permutations with certain properties; this idea was introduced by Kaplansky and Riordan~\cite{KR} and further developed by Riordan~\cite{Riordan}. For a comprehensive exposition of permutations with forbidden positions, we refer the reader to Stanley~\cite[Chapter~2]{StanleyEC}. 
 The \textit{rook problem} concerns counting the number of ways to place $k$ non-attacking rooks on a polyomino $\cP$. Generally speaking, two rooks are considered to be in \textit{non-attacking position} in $\cP$ if they cannot be connected by a path of edge-adjacent cells of $\cP$ along that line. The maximum number of non-attacking rooks that can be placed on a polyomino $\mathcal{P}$ is called the \textit{rook number} of $\mathcal{P}$ and is denoted by $r(\mathcal{P})$. The \textit{rook polynomial} of $\mathcal{P}$ is defined as $
r_\mathcal{P}(t) = \sum_{i=0}^{r(\mathcal{P})} r_i(\mathcal{P}) t^i,
$ where $r_i(\mathcal{P})$ denotes the number of ways to place $i$ non-attacking rooks on $\mathcal{P}$. The problem of determining the rook number and the rook polynomial for a pruned chessboard remains a challenging combinatorial question. Recently, a new approach has been developed to study this problem, by employing algebraic invariants, such as the Castelnuovo-Mumford regularity and the $h$-polynomial, of binomial ideals associated with pruned chessboards. Thanks to the work of Qureshi~\cite{Qureshi}, polyominoes have also gained prominence in commutative algebra. She associated to each polyomino a binomial ideal, called the \textit{polyomino ideal}. This ideal, denoted by $I_{\mathcal{P}}$, is generated by all inner 2-minors of a polyomino $\mathcal{P}$ in the polynomial ring over a field $K$ in the variables $x_v$, where $v$ is a vertex of $\mathcal{P}$. The $K$-algebra $K[\mathcal{P}]$ whose relations are given by $I_{\mathcal{P}}$ is called the \textit{coordinate ring} of $\mathcal{P}$. Since its introduction, significant progress has been made in understanding the algebraic properties of polyomino ideals. For further reading, we refer the reader to some of the most significant contributions in the literature \cite{Cisto_Navarra_closed_path,  Cisto_Navarra_weakly, Cisto_Navarra_Hilbert_series, Dinu_Navarra_Konig, Moradi, Simple equivalent balanced, Not simple with localization, KKV, KV2, Trento, Navarra,  Qureshi, Simple are prime, Shikama}. In~\cite{L-convessi}, Ene et al.\ were the first to establish a relationship between the Castelnuovo-Mumford regularity of $K[\mathcal{P}]$ and the rook number of $\mathcal{P}$, specifically for the class of $L$-convex polyominoes. Building upon this, Rinaldo and Romeo~\cite{Trento3} proved that if $\mathcal{P}$ is a thin simple polyomino, then the $h$-polynomial of $K[\mathcal{P}]$ coincides with the rook polynomial of $\mathcal{P}$. Roughly speaking, a polyomino without holes is called \textit{simple}, while a polyomino that does not contain the square tetromino is called \textit{thin}. Furthermore, they conjectured that this result holds for any thin polyomino, as stated in~\cite[Conjecture 4.5]{Trento3}. These results were later generalized to closed path polyominoes in~\cite{Cisto_Navarra_Rizwan, Cisto_Navarra_Hilbert_series, Navarra}. In~\cite{Kummini rook polynomial}, Kummini and Veer proved~\cite[Conjecture 4.5]{Trento3} for convex polyominoes whose vertex sets are sublattices of $\mathbb{N}^2$. It is important to emphasize that~\cite[Conjecture 4.5]{Trento3} has so far only been examined for polyominoes with at most one hole. In this paper, we extend the analysis to \textit{grid polyominoes}, a more intricate class of thin polyominoes with one or more holes, first formalized in~\cite{Trento}. More precisely, a grid polyomino is obtained by removing finitely many axis-aligned rectangular subregions from a rectangle, yielding a shape that resembles a perforated grid. This class naturally generalizes the classical rectangular chessboard to configurations with holes, for which the rook polynomial is well understood. Indeed, if $\mathcal{B}_{m,n}$ denotes a $m \times n$ chessboard, then $ r_{\mathcal{B}_{m,n}}(t) = \sum_{k=0}^{n} \binom{m}{k} P(n,k) t^k $, where $P(n,k) = n(n-1) \dots (n-k+1)$. This further motivates the study of rook polynomials for grid polyominoes, viewed as chessboards with holes, as a natural and meaningful generalization.

This article is structured as follows. In Section~\ref{polyominoideals}, we introduce the basics of simplicial complexes, polyominoes, and their $K$-algebras. We prove in Theorem~\ref{dim} that the coordinate ring of a grid polyomino $\cP$ is a normal Cohen-Macaulay domain with Krull dimension equal to the difference between the number of vertices and of cells of $\cP$. As a consequence, Corollary~\ref{height} shows that the height of the polyomino ideal of $\cP$ is equal to the number of cells of $\cP$. This result answers affirmatively to \cite[Conjecture 3.6]{Dinu_Navarra_Konig}; look at \cite[Theorem 1.1]{Moradi} for an interesting sufficient condition for this conjecture. The primary aim of this paper is to prove that the rook number and the rook polynomial of a grid polyomino $\cP$ coincide with the Castelnuovo-Mumford regularity and the $h$-polynomial of the coordinate ring of $\cP$, respectively. This result, combined with the use of algebraic software such as \texttt{Macaulay2} (\cite{Package_M2}, \cite{M2}), provides a practical and efficient method for computing the rook polynomial and the rook number of a grid polyomino (see Example \ref{Exa: rook polynomial}). To achieve our main goal, we explore the relationship of the $h$-polynomial of $K[\cP]$ with the facets of the simplicial complex $\Delta_{\cP}$ associated with $\cP$. We extend the technique from \cite[Section 3]{Navarra_Rizwan}, which was originally applied to polyominoes with a single hole to the more general case of grid polyominoes, which may have one or more holes. This is discussed in Section~\ref{Section: Shellability}. In Definition~\ref{Definition: Generalized Step of a facet}, we introduce the concept of a \textit{generalized step} in a face of $\Delta_{\cP}$, extending \cite[Definition 3.3]{Navarra_Rizwan}. This notion is further illustrated in Example~\ref{ex_gen_step}. Next, in Lemmas~\ref{Lemma: Structure generalized step}, \ref{Lemma: Structure generalized step - particular structure 1} and \ref{Lemma: Structure generalized step - particular structure 2}, we describe the possible structures of a generalized step in a grid polyomino. These structures are strongly determined by the combinatorics of the polyomino. Later on, continuing our study of the simplicial complex $\Delta_{\cP}$, we show in Theorem~\ref{Thm: The lexicographic order gives a shelling order} that the set of the facets of $\Delta_{\cP}$ ordered with respect to the descending lexicographical order forms a shelling for $\Delta_{\cP}$. Section~\ref{Sec: Rook} is dedicated to establish a well-defined one-to-one correspondence between the facets of $\Delta_{\cP}$ and the configurations of non-attacking rooks in $\cP$. This task, which constitutes the main innovation of this work, is both challenging and conceptually meaningful. It is not merely a technical argument, but rather a constructive framework that precisely establishes the correspondence between the combinatorial structure of polyominoes and the algebraic properties of their associated simplicial complexes. Specifically, for a general class of polyominoes, where the descending lexicographical order forms a shelling order for the associated simplicial complex, establishing a bijective correspondence between configurations of $k$ non-attacking rooks and facets with $k$ generalized steps is a highly non-trivial problem. In Definition~\ref{Defn: rook configuration to a facet}, we introduce a map $\cR(-)$ that assigns a configuration of $k$ non-attacking rooks to a facet of $\Delta_{\cP}$ with $k$ generalized steps. We then prove that $\cR(-)$ is bijective: the injectivity and surjectivity of this correspondence are established in Propositions~\ref{Prop: For injective} and~\ref{Prop: For surjective}, respectively. From Theorem~\ref{Thm: The lexicographic order gives a shelling order}, we know that the $k$-th coefficient of the $h$-polynomial of $K[\cP]$ equals the number of facets of $\Delta_{\cP}$ with $k$ generalized steps. This leads to the main result of this paper, Theorem~\ref{mainthm}, which shows that the $h$-polynomial of $K[\cP]$ for a grid polyomino $\cP$ equals the rook polynomial of $\cP$, thereby providing a definitive proof of \cite[Conjecture 4.5]{Trento3}. As a consequence, the regularity of the coordinate ring of $\cP$ coincides with the rook number of $\cP$, and the unique Gorenstein grid polyomino is finally determined.

\section{Simplicial complexes, polyominoes and their $K$-algebras}\label{polyominoideals}

 In this section, we recall the definitions and notation of simplicial complexes and, in particular, on polyominoes and their coordinate rings.
 
 Let us start by introducing some basic facts about simplicial complexes. A \textit{finite simplicial complex} $\Delta$ on $[n]:=\{1,\dots,n\}$ is a collection of subsets of $[n]$ satisfying the following two conditions:
	\begin{enumerate}
		\item if $F'\in \Delta$ and $F \subseteq F'$ then $F \in \Delta$;
		\item $\{i\}\in \Delta$ for all $i=1,\dots,n$.
	\end{enumerate}  
	The elements of $\Delta$ are called \textit{faces}; the dimension of a face $F$ is one less than its cardinality and it is denoted by $\dim(F)$. An \textit{edge} of $\Delta$ is a face of dimension 1, while a \textit{vertex} of $\Delta$ is a face of dimension 0. The maximal faces of $\Delta$ with respect to the set inclusion are called \textit{facets}. The dimension of $\Delta$ is given by $\sup\{\dim(F):F\in \Delta\}$. A simplicial complex $\Delta$ is \textit{pure} if all facets have the same dimension. For a pure simplicial complex $\Delta$, the dimension of $\Delta$ is given trivially by the dimension of a facet of $\Delta$. Given a collection $\mathscr{F}=\{F_1,\dots,F_m\}$ of subsets of $[n]$, we denote by $\langle F_1,\dots,F_m\rangle$ or briefly $\langle \mathscr{F}\rangle$ the simplicial complex consisting of all subsets of $[n]$ which are contained in $F_i$, for some $i\in[m]$. This simplicial complex is said to be generated by $F_1,\dots,F_m$; in particular, if $\cF$ is the set of the facets of a simplicial complex $\Delta$, then $\Delta$ is generated by $\cF$. A pure simplicial complex $\Delta$ is called \textit{shellable} if the facets of $\Delta$ can be ordered as $F_1,\dots,F_m$ in such a way that $\langle F_1,\dots,F_{i-1}\rangle \cap \langle F_i\rangle$ is generated by a non-empty set of maximal proper faces of $F_i$, for all $i\in \{2,\dots,m\}$. In this case, $F_1,\dots,F_m$ is called a \textit{shelling} of $\Delta$. \\ 
	Let $\Delta$ be a simplicial complex on $[n]$ and  $R=K[x_1,\dots,x_n]$ be the polynomial ring in $n$ variables over a field $K$. To every collection $F=\{i_1,\dots,i_r\}$ of $r$ distinct vertices of $\Delta$, there is an associated monomial $x_F$ in $R$ where $x_F=x_{i_1}\dots x_{i_r}.$ The monomial ideal generated by all monomials $x_F$ such that $F$ is not a face of $\Delta$ is called \textit{Stanley-Reisner ideal} and it denoted by $I_{\Delta}$. The \textit{face ring} of $\Delta$, denoted by $K[\Delta]$, is defined to be the quotient ring $R/I_{\Delta}$. From \cite[Corollary 6.3.5]{Villareal}, if $\Delta$ is a simplicial complex on $[n]$ of dimension $d$, then $\dim K[\Delta]=d+1=\max\{s: x_{i_1}\cdots x_{i_s}\notin I_{\Delta},i_1<\dots<i_s\}$. We recall here a nice combinatorial interpretation of the $h$-vector of a shellable simplicial complex that will be used later.
	
	\begin{prop}{\cite[Corollary 5.1.14]{Bruns_Herzog}}\label{Prop: shellability BH}
		Let $\Delta$ be a shellable simplicial complex of dimension $d$ with shelling $F_1,\dots,F_m$. For $j\in \{2,\dots,m\}$ we denote by $r_j$ the number of facets of $\langle F_1,\dots,F_{j-1}\rangle\cap \langle F_j\rangle$ and we set $r_1=0$. Let $(h_0,\dots,h_{d+1})$ be the $h$-vector of $K[\Delta]$. Then $h_i=|\{j:r_j=i\}|$ for all $i\in [d+1]$. In particular, up to their order, the numbers $r_j$ do not depend on the particular shelling. 
	\end{prop} 

 Now, we introduce the definition of a polyomino and the associated $K$-algebras. Let $(i,j),(k,l)\in \ZZ^2$. We say that $(i,j)\leq(k,l)$ if $i\leq k$ and $j\leq l$. Consider $a=(i,j)$ and $b=(k,l)$ in $\ZZ^2$ with $a\leq b$. The set $[a,b]=\{(m,n)\in \ZZ^2: i\leq m\leq k,\ j\leq n\leq l \}$ is called an \textit{interval} of $\ZZ^2$. If $i< k$ and $j<l$, then we say that $[a,b]$ is a \textit{proper} interval. In this case, we call $a$ and $b$ the \textit{diagonal corners} of $[a,b]$, and $c=(i,l)$ and $d=(k,j)$ the \textit{anti-diagonal corners} of $[a,b]$. If $j=l$ (or $i=k$), then $a$ and $b$ are in \textit{horizontal} (or \textit{vertical}) \textit{position}. We denote by $\mathrm{int}([a,b])$ the set $\{(m,n)\in \Z^2: i< m< k,\ j< n< l\}$. A proper interval $C=[a,b]$ with $b=a+(1,1)$ is called a \textit{cell} of $\ZZ^2$; moreover, the elements $a$, $b$, $c$ and $d$ are called respectively the \textit{lower left}, \textit{upper right}, \textit{upper left} and \textit{lower right} \textit{corners} of $C$. The set of vertices of $C$ is $V(C)=\{a,b,c,d\}$ and the set of edges of $C$ is $E(C)=\{\{a,c\},\{c,b\},\{b,d\},\{a,d\}\}$. Let $\cS$ be a non-empty collection of cells in $\Z^2$. Then $V(\cS)=\bigcup_{C\in \cS}V(C)$ and $E(\cS)=\bigcup_{C\in \cS}E(C)$. The \textit{rank} of $\cS$ is the number of cells that belong to $\cS$ and is denoted by $\rank(\cS)$. An interval $[a,b]$ with $a=(i,j)$, $b=(k,j)$ and $i<k$ is called a \textit{horizontal edge interval} of $\cS$ if the sets $\{(\ell,j),(\ell+1,j)\}$ are edges of cells of $\cS$, for all $\ell=i,\dots,k-1$. In addition, if $\{(i-1,j),(i,j)\}$ and $\{(k,j),(k+1,j)\}$ do not belong to $E(\cS)$, then $[a,b]$ is called a \textit{maximal} horizontal edge interval of $\cS$. We define similarly a \textit{vertical edge interval} and a \textit{maximal} vertical edge interval. If $C$ and $D$ are two distinct cells of $\cS$, then a \textit{walk} from $C$ to $D$ in $\cS$ is a sequence $\cC:C=C_1,\dots,C_m=D$ of cells of $\ZZ^2$ such that $C_i \cap C_{i+1}$ is an edge of $C_i$ and $C_{i+1}$, for $i=1,\dots,m-1$. Moreover, if $C_i \neq C_j$ for all $i\neq j$, then $\cC$ is called a \textit{path} from $C$ to $D$. Moreover, if we denote by $(a_i,b_i)$ the lower left corner of $C_i$ for all $i=1,\dots,m$, then $\cC$ has a \textit{change of direction} at $C_k$, for some $2\leq k \leq m-1$, if $a_{k-1} \neq a_{k+1}$ and $b_{k-1} \neq b_{k+1}$; in this case, $\{C_{k-1},C_k,C_{k+1}\}$ is said to be the set of the \textit{cells of a change of direction}. We say that $C$ and $D$ are \textit{connected} in $\cS$ if there exists a path of cells in $\cS$ from $C$ to $D$. A \textit{polyomino} $\cP$ is a non-empty, finite collection of cells in $\Z^2$ where any two cells of $\cP$ are connected in $\cP$. For instance, see Figure \ref{Figure: Polyomino introduction}.

\begin{figure}[h]
\centering
\includegraphics[scale=0.5]{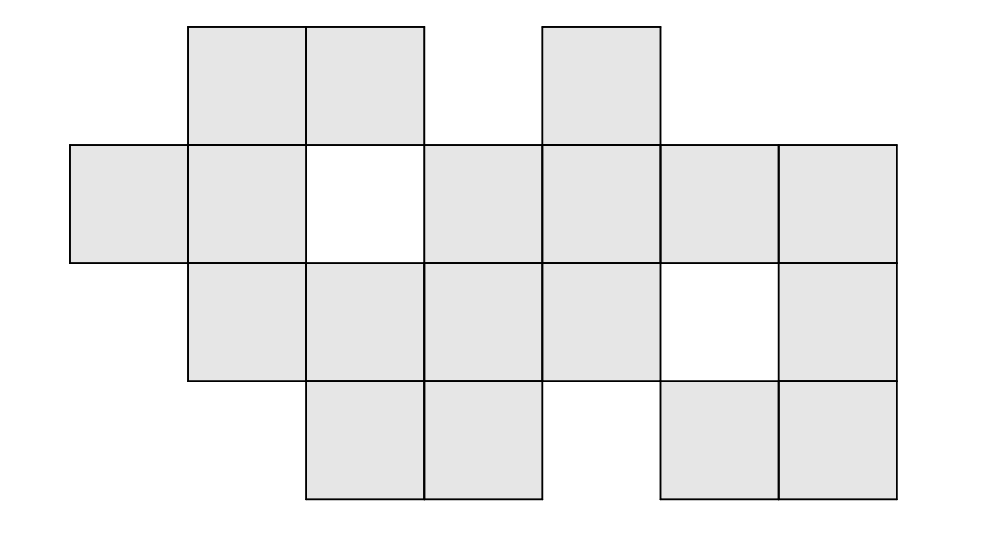}
\caption{A polyomino.}
\label{Figure: Polyomino introduction}
\end{figure}
	
 \noindent A \textit{sub-polyomino} of $\cP$ is a polyomino whose cells belong to $\cP$. We say that $\cP$ is \textit{simple} if for any two cells $C$ and $D$ not in $\cP$ there exists a path of cells not in $\cP$ from $C$ to $D$. A finite collection of cells $\cH$ not in $\cP$ is a \textit{hole} of $\cP$ if any two cells of $\cH$ are connected in $\cH$ and $\cH$ is maximal with respect to set inclusion. For example, the polyomino in Figure \ref{Figure: Polyomino introduction} is not simple. Obviously, each hole of $\cP$ is a simple polyomino, and $\cP$ is simple if and only if it has no hole. A polyomino is said to be \textit{thin} if it does not contain the square tetromino, which is a polyomino consisting of four unit squares arranged in a $2 \times 2$ square block. Consider two cells $A$ and $B$ of $\Z^2$, with $a = (i,j)$ and $b = (k,l)$ representing the lower left corners of $A$ and $B$, and $a < b$. A \textit{cell interval} $[A,B]$ is the set of the cells of $\Z^2$ with lower left corner $(r,s)$ such that $i\leqslant r\leqslant k$ and $j\leqslant s\leqslant l$. If $(i,j)$ and $(k,l)$ are in horizontal (or vertical) position, we say that the cells $A$ and $B$ are in \textit{horizontal} (or \textit{vertical}) \textit{position}. Let $\cP$ be a polyomino. Consider  two cells $A$ and $B$ of $\cP$ in vertical or horizontal position. The cell interval $[A,B]$, containing $n>1$ cells, is called a \textit{block of $\cP$ of rank n} if all cells of $[A,B]$ belong to $\cP$. The cell interval $[A,B]$, containing $n>1$ cells, is called a \textit{column} (resp. \textit{row}) of $\cP$ of rank $n$ if all the cells in $[A,B]$ belong to $\cP$. The cells $A$ and $B$ are called \textit{extremal} cells of $[A,B]$. We also define $]A,B[ = [A,B] \setminus {A,B}$, $]A,B] = [A,B] \setminus {A}$, and $[A,B[ = [A,B] \setminus {B}$. Moreover, a block $\cB$ of $\cP$ is \textit{maximal} if there does not exist any block of $\cP$ which properly contains $\cB$. It is clear that an interval of $\ZZ^2$ identifies a cell interval of $\ZZ^2$ and vice versa: hence we can associate to an interval $I$ of $\ZZ^2$ the corresponding cell interval denoted by $\cP_{I}$. A proper interval $[a,b]$ is called an \textit{inner interval} of $\cP$ if all cells of $\cP_{[a,b]}$ belong to $\cP$. \\
      Let $\cP$ be a polyomino and $S_\cP=K[x_v| v\in V(\cP)]$ be a polynomial ring over a field $K$. If $[a,b]$ is an inner interval of $\cP$, with $a$,$b$ and $c$,$d$, respectively, diagonal and anti-diagonal corners, then the binomial $x_ax_b-x_cx_d$ is called an \textit{inner 2-minor} of $\cP$. $I_{\cP}$ is called \textit{polyomino ideal} of $\cP$ and is defined as the ideal in $S_\cP$ generated by all the inner 2-minors of $\cP$. We set also $K[\cP] = S_\cP/I_{\cP}$, which is the \textit{coordinate ring} of $\cP$. Along the paper, if we do not specify differently, we denote by $<$ the reverse lexicographical order on $S_{\cP}$ induced by the ordering of the variables defined by $x_{ij}>x_{kl}$ if $j > l$, or, $j = l$ and $i > k$. In particular, if $f=x_ax_b-x_cx_d$ is a generator of $I_{\cP}$, with $a$,$b$ and $c$,$d$ respectively diagonal and anti-diagonal corners, then the \textit{initial monomial} $\lt_<(f)$ of $f$ with respect to $<$ is $x_cx_d$. As proved in \cite[Remark 4.2]{Qureshi}, if $\cP$ is a collection of cells, then the set of inner 2-minors of $\cP$ forms the reduced (quadratic) Gr\"obner basis with respect to $<$ if and only if for any two inner intervals $[a, b]$ and $[e, f]$ of $\cP$, with $d$ anti-diagonal corner of both the inner intervals, either $a, e$ or $b, f$ are anti-diagonal corners of an inner interval of $\cP$.  Moreover, if $\cP$ satisfies such a condition, which means that $\lt_{<}(I_\cP)$ is squarefree and it is generated in degree two, then the simplicial complex on $V(\cP)$ having $\lt_<(I_{\cP})$ as Stanley-Reisner ideal is called the \textit{simplicial complex attached to $\cP$} and is denoted by $\Delta_{\cP}$. Finally, the set of the facets of $\Delta_{\cP}$ will be denoted by $\cF_\cP$.
      
     In this article, we will focus on a specific class of polyominoes, called \textit{grid polyominoes}. This class was formalized by Mascia, Rinaldo, and Romeo, in \cite{Trento}. Let us recall the definition.

\begin{defn}\label{grid}\rm 
A polyomino $\cP \subseteq \cP_I$, where $I= [(1, 1),(m, n)]$ is called \textit{grid} if:
\begin{enumerate}
    \item $\cP=\cP_{I} \setminus \bigcup_{i\in[r]\atop  j\in[s]}\cH_{ij}$, where $\cH_{ij} =\cP_{[a_{ij} , b_{ij} ]}$, with $a_{ij} = ((a_{ij})_1,(a_{ij})_2), b_{ij} = ((b_{ij} )_1,(b_{ij} )_2)$, $1 < (a_{ij})_1 <(b_{ij} )_1 < m, 1 < (a_{ij} )_2 < (b_{ij} )_2 < n$;
    \item for any $i \in [r]$ and $l, k \in[s]$ we have $(a_{il})_1 = (a_{ik})_1$ and $(b_{il})_1 = (b_{ik})_1$;
    \item for any $j \in [s]$ and $l, k \in [r]$ we have $(a_{lj})_2 = (a_{kj} )_2$ and $(b_{lj} )_2 = (b_{kj} )_2$;
    \item for any $i \in [r - 1]$ and $j \in [s-1]$, we have $(a_{i+1,j} )_1 = (b_{ij} )_1 + 1$ and $(a_{i,j+1})_2 = (b_{ij} )_2 + 1$.
\end{enumerate}
We set the anti-diagonal corners of $\cH_{ij}$ by $c_{ij}:=((a_{ij})_1, (b_{ij})_2)$ and $d_{ij}:=((b_{ij})_1, (a_{ij})_2)$.
\end{defn}

\noindent Examples of grid polyominoes are presented in Figure~\ref{Figure: grid polyominoes}. Moreover, for simplicity, for a grid polyomino, we will assume the notation in Definition~\ref{grid} throughout the paper.

    \begin{figure}[H]
		\centering
		\subfloat[]{\includegraphics[scale=0.6]{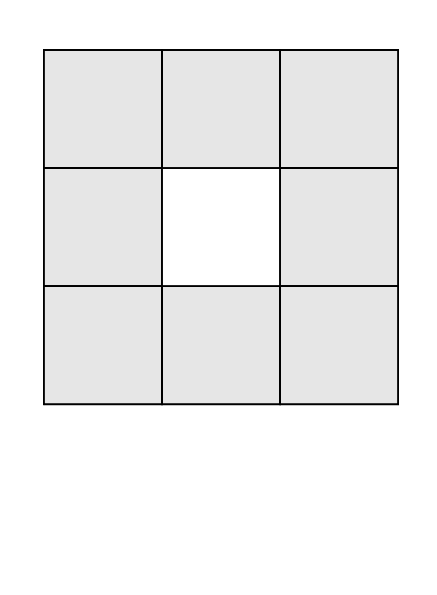}}
		\qquad\qquad
		\subfloat[]{\includegraphics[scale=0.45]{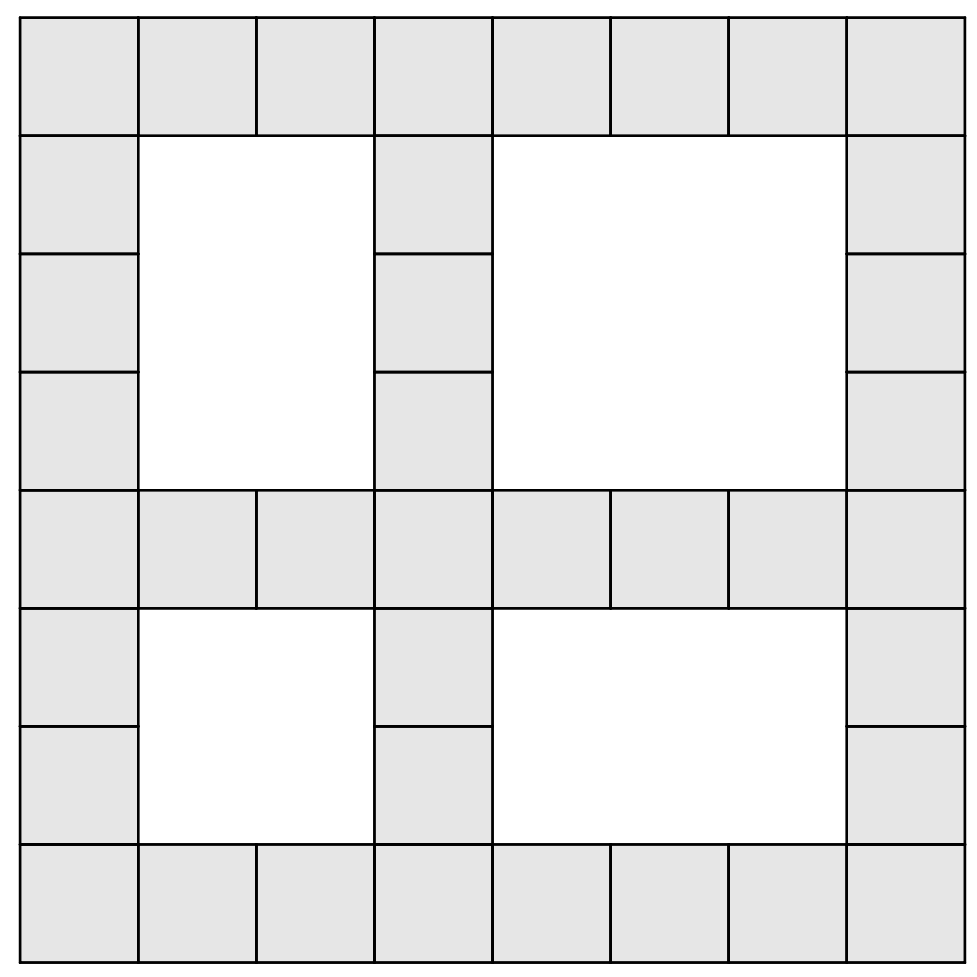}}
		\caption{Examples of grid polyominoes.}
		\label{Figure: grid polyominoes}
    \end{figure}

   \begin{thm}\label{dim}
Let $\cP$ be a grid polyomino. Then $K[\cP]$ is a normal Cohen-Macaulay domain of Krull dimension $\vert V(\cP)\vert -\rank(\cP)$.
    \end{thm}
\begin{proof}
 By \cite[Theorem 4.4, Corollary 4.5]{Trento}, the ideal $I_{\cP}$ is prime, as the authors provided an appropriate toric representation for it. Moreover, it is straightforward to verify that $\cP$ satisfies the condition in \cite[Remark 4.2]{Qureshi}. Consequently, the set of generators of $I_{\cP}$ constitutes the reduced Gr\"obner basis of $I_{\cP}$ with respect to $<$. Therefore, $K[\cP]$ is a normal Cohen-Macaulay domain (see \cite[Corollary 4.26]{binomial ideals} and \cite[Theorem 6.3.5]{Bruns_Herzog}).

We establish the Krull dimension formula by analyzing the simplicial complex associated with $\cP$. According to \cite[Theorem 9.6.1]{Villareal}, the simplicial complex $\Delta_{\cP}$ is pure and shellable. Hence, to determine the dimension of $\Delta_{\cP}$, it suffices to compute the cardinality of a facet of $\Delta_{\cP}$. Using the notation in Definition \ref{grid}, for $i\in[r]$ and $j\in[s]$, $\cH_{ij}$ can be expressed as the disjoint union $\intt(\cH_{ij}) \sqcup U_{ij} \sqcup B_{ij},$
where $U_{ij}=[a_{ij}, c_{ij}] \cup [c_{ij}, b_{ij}]$ and $B_{ij}=[a_{ij}+(1,0), d_{ij}] \cup [d_{ij}, b_{ij}-(0,1)]$.
Thus, $|V(\cH_{ij})| = |\intt(\cH_{ij})| + |U_{ij}| + |B_{ij}|$, and therefore,
$|\intt(\cH_{ij})| = |V(\cH_{ij})| - |U_{ij}| - |B_{ij}|$. Since $\cP = \cP_{I} \setminus \bigcup_{i\in[r]\atop j\in[s]}\cH_{ij}$, we have 
\[
\vert V(\cP)\vert = \vert V(\cP_{I})\vert - \sum_{i\in[r]\atop j\in[s]} |\intt(\cH_{ij})| \quad \text{and} \quad \rank(\cP) = \rank(\cP_{I}) - \sum_{i\in[r]\atop j\in[s]} \rank(\cH_{ij}).
\]
Therefore, we obtain 
\begin{equation}
|V(\cP)| - \rank(\cP) = |V(\cP_{I})| - \rank(\cP_{I}) - \sum_{i\in[r]\atop j\in[s]} |V(\cH_{ij})| + \sum_{i\in[r]\atop j\in[s]} (|U_{ij}| + |B_{ij}|) + \sum_{i\in[r]\atop j\in[s]} \rank(\cH_{ij}).
\end{equation}

For any $i \in [r]$ and $j \in [s]$, both $\cP_{I}$ and $\cH_{ij}$ satisfy \cite[Remark 4.2]{Qureshi}. Therefore, consider the simplicial complexes $\Delta_{\cP_{I}}$ and $\Delta_{\cH_{ij}}$ associated with $\cP_{I}$ and $\cH_{ij}$, respectively. It follows that $\dim K[\cP_{I}] = \dim K[\Delta_{\cP_{I}}]$ and $\dim K[\cH_{ij}] = \dim K[\Delta_{\cH_{ij}}]$. Moreover, as $\cP_{I}$ and $\cH_{ij}$ are simple polyominoes, by \cite[Theorem 2.1]{Simple equivalent balanced} and \cite[Corollary 3.3]{def balanced}, we have $\dim K[\cP_{I}] = \vert V(\cP_{I})\vert - \rank(\cP_{I})$ and
$\dim K[\cH_{ij}] = \vert V(\cH_{ij})\vert - \rank(\cH_{ij})$. Since $\Delta_{\cP_{I}}$ and $\Delta_{\cH_{ij}}$ are pure, we can write $\dim (\Delta_{\cP_{I}}) = |F_{I}|$, where $F_I = [(1,1), (1,n)] \cup [(1,n), (m,n)]$, and $\dim (\Delta_{\cH_{ij}}) = |U_{ij}|$. After straightforward computations, we deduce from the equation (1) that $|V(\cP)| - \rank(\cP) = \left| F_{I} \cup \left(\bigcup_{i\in[r]\atop j\in[s]} B_{ij}\right)\right|.$

Observe that for any $v, w \in F_{I} \cup \left(\bigcup_{i\in[r]\atop j\in[s]} B_{ij}\right)$, there is no inner interval of $\cP$ such that $v, w$ are its anti-diagonal corners. Additionally, for any $p \in V(\cP) \setminus \left(F_{I} \cup \left(\bigcup_{i\in[r]\atop j\in[s]} B_{ij}\right)\right)$, there exists a vertex $q \in F_{I} \cup \left(\bigcup_{i\in[r]\atop j\in[s]} B_{ij}\right)$ and an inner interval of $\cP$ whose anti-diagonal corners are $p$ and $q$. Therefore, $F_{I} \cup \left(\bigcup_{i\in[r]\atop j\in[s]} B_{ij}\right)$ is a facet of the simplicial complex $\Delta_{\cP}$, which establishes the desired conclusion.

\end{proof}

\begin{coro}\label{height}
Let $\cP$ be a grid polyomino. Then the height of $I_{\cP}$ is equal to $\rank(\cP).$
\end{coro}

\begin{proof}
    It follows from Theorem~\ref{dim} and \cite[Corollary 3.1.7]{Villareal}.
\end{proof}

\section{Shellability of the simplicial complex attached to a grid polyomino}\label{Section: Shellability}
    
Inspired by the work \cite{Navarra_Rizwan}, particularly by Remark 3.11, we are interested in investigating the shellability of the simplicial complexes attached to grid polyominoes. In this section, we will first provide a generalization of the concept of a step of a face, a notion originally introduced in \cite[Definition 3.3]{Navarra_Rizwan}. This generalization will serve as the foundation for understanding the shellability of these complexes. Based on the analysis of the structure of the steps and the specific shape of grid polyominoes, we will prove that the simplicial complexes associated with these polyominoes are shellable.

	\begin{defn}\rm\label{Definition: Generalized Step of a facet}
	 Let $\cP$ be a polyomino and $\Delta_{\cP}$ be the simplicial complex on $V(\cP)$ having $\lt_<(I_{\cP})$ as Stanley-Reisner ideal. A subset $F'$ of a face $F$ of $\Delta_{\cP}$ is a \textit{generalized step} in $F$ if:
\begin{enumerate}
\item $F'=\{(a,b),(c,b),(c,d)\}$ with $a<c$ and $b<d$;
\item for all $i\in \{a+1, \dots, c-1\}$ there is no $(i, b)$ in $F$; 
\item for all $j\in \{b+1, \dots, d-1\}$ there is no $(c, j)$ in $F$;
\item $[(a,b),(c,d)]$ is an inner interval of $\cP$.
\end{enumerate}
	We say that the vertex $(c,b)$ is the \textit{lower right corner} of $F'$.
	\end{defn}

Observe that if $F'$ is a generalized step in $F$ then $F'$ is a \textit{step} in $F$, according to \cite[Definition 3.3]{Navarra_Rizwan}.

	\begin{exa}\rm \label{ex_gen_step}
	Consider the polyomino displayed in Figure~\ref{Figure: Examples steps}. Observe that $\cP$ satisfies the condition from \cite[Remark 4.2]{Qureshi}, so the reduced Gröbner basis of $I_{\cP}$ with respect to $<$ is the set of the generators of $I_{\cP}$; in particular, the set $\{a_i:i\in [21]\}$ is a facet of the simplicial complex $\Delta_{\cP}$ associated to $\cP$. There are only three generalized steps in $F$: $\{a_{11},a_{10},a_7\}, \{a_{15},a_{14},a_{8}\}, \{a_{18},a_{17},a_{16}\}$. Note that $\{a_{19},a_{18},a_{2}\}$ is a step, but is not a generalized step since $[a_{19},a_{2}]$ is not an inner interval of $\cP$.
	
	\begin{figure}[h!]
		\centering
		\includegraphics[scale=0.8]{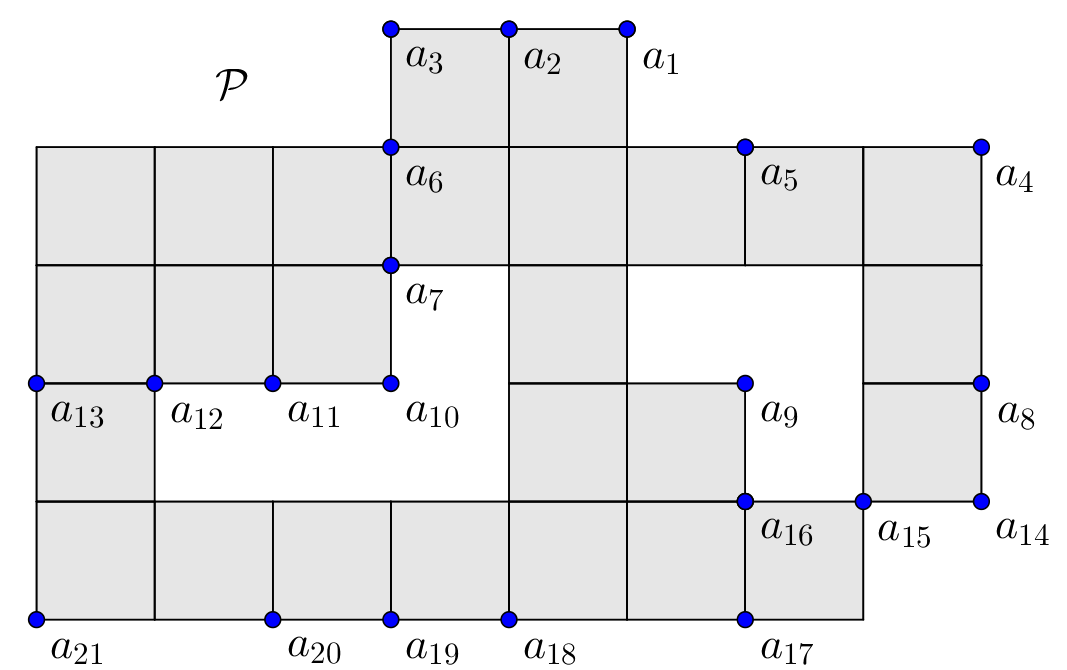}
		\caption{Example of (generalized) steps in a facet of $\Delta_{\cP}$.}
		\label{Figure: Examples steps}
	\end{figure}	

	\noindent Observe that $a_1, a_4$, $a_9$, $a_{15}$ and $a_{21}$ belong to any facet of $\Delta_{\cP}$ because they are never the anti-diagonal corner of an inner interval of $\cP$. For the same reason, if $\cP$ is a grid polyomino with $V(\cP) \subseteq [(1,1),(m,n)]$, then $(1,1)$ and $(m,n)$ belong to any facet of the simplicial complex attached to $\cP$.

\end{exa}

In the following three lemmas, we provide a complete description of the structure of a generalized step in a grid polyomino.
	 
	\begin{lemma}\label{Lemma: Structure generalized step}
	Let $\cP$ be a grid polyomino and $\Delta_{\cP}$ the attached simplicial complex. Let $F'=\{(a,j),(i,j),(i,b)\}$ be a generalized step of a facet $F$ of $\Delta_{\cP}$, where $(i,j)$ is the lower right corner of $F'$. Then the following are the only possibilities for the values of $a$ and $b$:
    \begin{itemize}
        \item $a=i-1$ and $b=j+1$;
        \item $a \in \{i-2, i-3\}$ and $b = j+1$;
        \item $a = i-1$ and $b \in \{j+2, j+3\}$.
    \end{itemize}     
	\end{lemma}
    
	\begin{proof}
First, we observe that the assumption $a \leq i - 2$ and $b \geq j + 2$ simultaneously leads to a contradiction. Indeed, if $a \leq i - 2$ and $b \geq j + 2$, then the interval $[(a,j), (i,b)]$ contains a square tetromino, due to (2) of Definition \ref{Definition: Generalized Step of a facet}. However, this is impossible since $\cP$ is a grid polyomino, and therefore thin. Moreover, if either $a\geq i$ or $b\leq j$, then we have a contradiction with (1) of Definition \ref{Definition: Generalized Step of a facet}. Therefore, we can have either $a=i-1$ and $b=j+1$ or $a\leq i-2$ and $b=j+1$ or $a=i-1$ and $b\geq j+2$. \\
It follows directly from the definition of a generalized step that if $a = i - 1$ and $b = j + 1$, this does not lead to any contradiction, thus satisfying the conditions of the claim. Next, assume that $a \neq i - 1$ and $b = j + 1$. We aim to prove that, in this case, it must hold that $a \in\{ i - 2, i - 3\}$.  Now, suppose by contradiction that $a<i-3$. Keep in mind that no vertex in $\{(p, j) \in V(\cP): a < p < i\}$ is in $F$, due to be $F'$ a step of $F$. Referring to Figure \ref{Figure: step in grid polyominoes} (A), we define the following. Let $A$ and $B$ denote the cells with $(a, j)$ and $(i - 1, j)$ as their respective lower left corners. By (2) of Definition \ref{Definition: Generalized Step of a facet}, we have $[A, B] \subset \cP$. Furthermore, we define $\rank([A, B]) = k$, implying that $i - a = k$, and we observe that $k \geq 4$. Let $A_h$ be the cell in $\mathbb{Z}^2$ whose lower left corner is $(a - 1 + h, j + 1)$, and let $B_h$ be the cell whose lower left corner is $(a - 1 + h, j - 1)$, for all $h \in [k]$. Define $H_1$ and $H_2$ as the maximal edge intervals of $\cP$ containing, respectively, $(i, j + 1)$ and $(i, j)$ and set $H_1^{<i}=\{(p,j+1)\in H_1:p<i\}$ and $H_2^{<i}=\{(p,j)\in H_2:p<i\}$.
\begin{figure}[h!]
		\centering
		\subfloat[]{\includegraphics[scale=0.55]{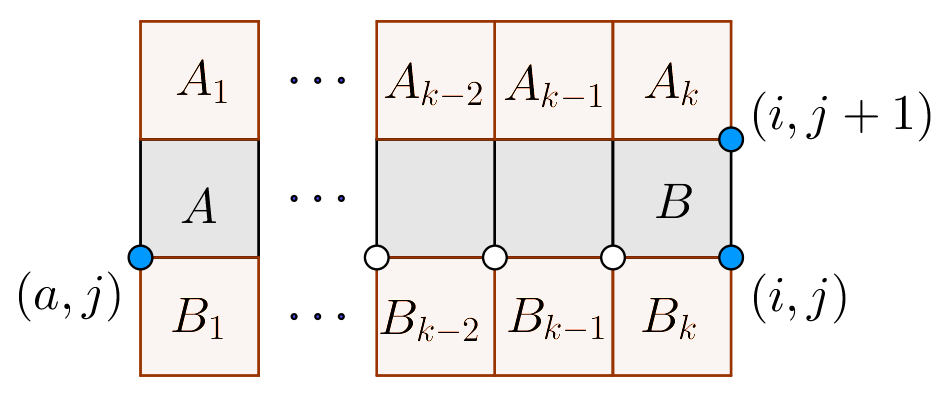}}
		\subfloat[]{\includegraphics[scale=0.55]{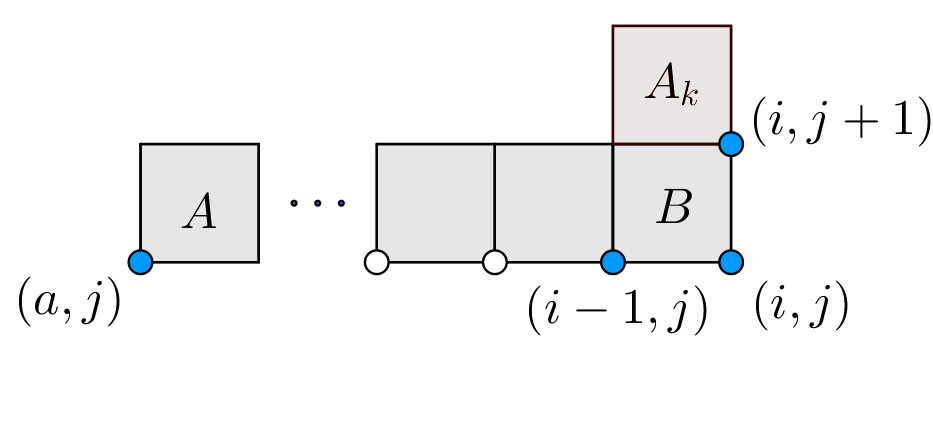}}
		\subfloat[]{\includegraphics[scale=0.55]{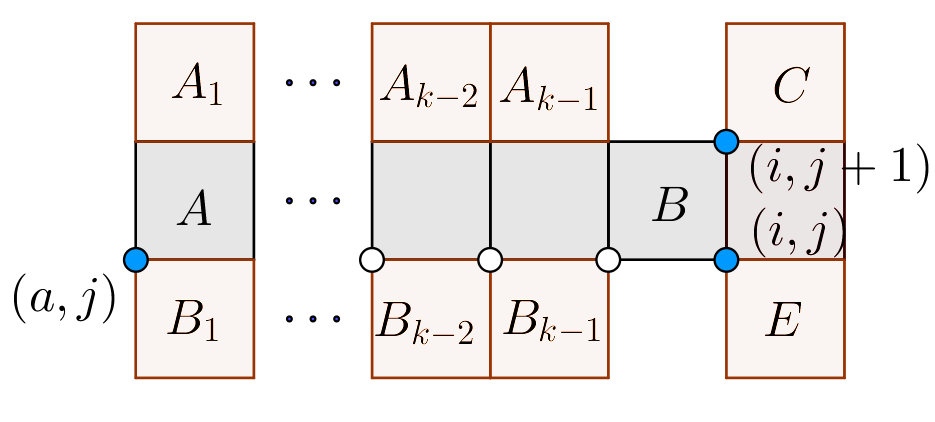}}
		\caption{Several arrangements of the cells of $\cP$.}
		\label{Figure: step in grid polyominoes}
	\end{figure}
    
Now, we examine all the possible cases depending on the cells of $[A,A_k]$ and $[B,B_k]$ which belong to $\cP$ and we demonstrate that each one leads to a contradiction.
    
    \begin{enumerate}
\item Assume that $A_k \in \cP$ and $B_k \notin \cP$. In this case, all cells of $\ZZ^2$ with lower left corners $(p, q)$, where $p \in \ZZ$ and $q < j$, do not belong to $\cP$; otherwise, $B_k \in \cP$, as $\cP$ is a grid polyomino. Moreover, $A_{k-1} \notin \cP$ trivially follows since $\cP$ is thin. Look now at Figure \ref{Figure: step in grid polyominoes} (B)). Since $(i, j) \in F$, every vertex in $H_1^{<i}$ does not belong to $F$. Hence, no inner interval of $\cP$ exists with $(i-1, j)$ as an anti-diagonal corner and another anti-diagonal corner in $F$. From the maximality of $F$, it follows that $(i-1, j) \in F$. This leads to a contradiction because no vertex in $\{(p, j) \in V(\cP): a < p < i\}$ is in $F$, as $F'$ is a step of $F$.
\item Suppose that $A_{k-1} \in \cP$ and $B_{k-1} \notin \cP$. In this case, $(i-2, j) \in F$ follows from arguments analogous to the previous case. Similarly, all other cases where $A_h \in \cP$ and $B_h \notin \cP$ for some $h \in [k]$ can be handled using the same reasoning, always leading to a contradiction.
\item Assume that $A_k \notin \cP$ and $B_k \in \cP$. In this case, we encounter a contradiction because $(i-2, j) \in F$. Similarly, if $A_{k-1} \notin \cP$ and $B_{k-1} \in \cP$, then $(i-1, j) \in F$, again resulting in a contradiction. By analogous reasoning, the same conclusion can be obtained for all other cases where $A_h \notin \cP$ and $B_h \in \cP$ for some $h \in [k]$.
\item If $A_k, B_{k-3} \in \cP$, then $(i-2, j) \in F$, which again leads to a contradiction.
	\item Suppose that $A_{k-1}, B_{k-3} \in \cP$. Then $B_{k-1}, A_{k-3} \in \cP$, since $\cP$ is a grid polyomino. Let $V_1$ and $V_2$ denote the maximal vertical edge intervals of $\cP$ containing, respectively, $(i-1, j)$ and $(i-2, j)$. Define $V_1^{<j} = \{(i-1, q) \in V_1 : q < j\}$ and $V_2^{>j} = \{(i-2, q) \in V_2 : q > j\}$. Here, we distinguish two sub-cases:
\begin{enumerate} \item If there exists an element of $F$ belonging to $V_1^{<j}$, then $(i-2, j) \notin F$. However, in this situation, it is evident that $(i-1, j) \in F$, as no element in $V_2^{>j} \cup H_1^{<i}$ belongs to $F$. This leads to a contradiction.
\item If no vertex in $V_1^{<j}$ belongs to $F$, then $(i-2, j) \in F$, which is a contradiction.
\end{enumerate}
\item Assume that $A_{k}, B_{k-2} \in \cP$, so $B_{k}, A_{k-2} \in \cP$. Let $V$ be the maximal edge interval of $\cP$ containing $(i, j)$ and $V^{<j} = \{(i, q) \in V : q < j\}$. We consider two sub-cases.
\begin{enumerate} \item If no element of $V^{<j}$ belongs to $F$, then obviously $(i-1, j) \in F$, a contradiction.
\item Suppose that there exists an element of $V^{<j}$ in $F$. Let $\tilde{V}$ be the maximal edge interval of $\cP$ containing $(i-2, j)$ and $\tilde{V}^{<j} = {(i, q) \in \tilde{V} : q < j}$. We need to examine two sub-cases related to $\tilde{V}^{<j} \cap F$. If there exists an element of $F$ in $\tilde{V}^{<j}$, then $(i-2, j) \in F$; instead, if $\tilde{V}^{<j} \cap F = \emptyset$, $(i-3, j) \in F$, a contradiction. Hence the case $A_{k}, B_{k-2} \in \cP$ cannot happen.
\end{enumerate}
\item The case $A_h, B_h \notin \cP$ for all $h \in [k]$ implies that $(i-1, j) \in F$, which is a contradiction.
\item The reader can easily verify that all other remained cases give contradictions similarly.
\end{enumerate}
Hence, all the cases depending on the cells of $[A, A_k]$ and $[B, B_k]$ belonging to $\cP$ provide a contradiction.
Let $C$ and $E$ be the cells in Figure \ref{Figure: step in grid polyominoes} (C). The reader can easily verify that the cases when either $(C \in \cP, E \notin \cP)$ or $(C \notin \cP, E \in \cP)$ or $(C \in \cP, E \in \cP)$ or $(C \notin \cP, E \notin \cP)$ imply that either $(i-1, j) \in F$ or $(i-2, j) \in F$, which is a contradiction. \\
Therefore, all the possible configurations of cells of $\cP$ lead to a contradiction, so necessarily $a \in \{i-1, i-2\}$. The last claim, that is, $a = i-1$ and $b \in \{j+2, j+3\}$, can be proved similarly. This concludes the proof of our lemma.
	\end{proof}
	
	\begin{lemma}\label{Lemma: Structure generalized step - particular structure 1}
		Let $\cP$ be a grid polyomino and $\Delta_{\cP}$ be the attached simplicial complex. Let $F'=\{(a,j),(i,j),(i,b)\}$ be a generalized step of a facet $F$ of $\Delta_{\cP}$, where $(i,j)$ is the lower right corner of $F'$. Then $a=i-3$ and $b=j+1$ if and only if the following hold:
		\begin{enumerate}
			\item the vertices of $F'$ are arranged in a sub-polyomino of $\cP$ as shown in Figure in Figure \ref{Figure: step in grid polyominoes if and only if - uno};
			\item if $V_1=\{(i,q)\in V(\cP):q<j\}$ and $V_2=\{(i-3,q)\in V(\cP):q>j+1\}$ then $V_1\cap F\neq \emptyset$ and $V_2\cap F\neq \emptyset$.
		\end{enumerate}
	\begin{figure}[h!]
		\centering
		\includegraphics[scale=0.6]{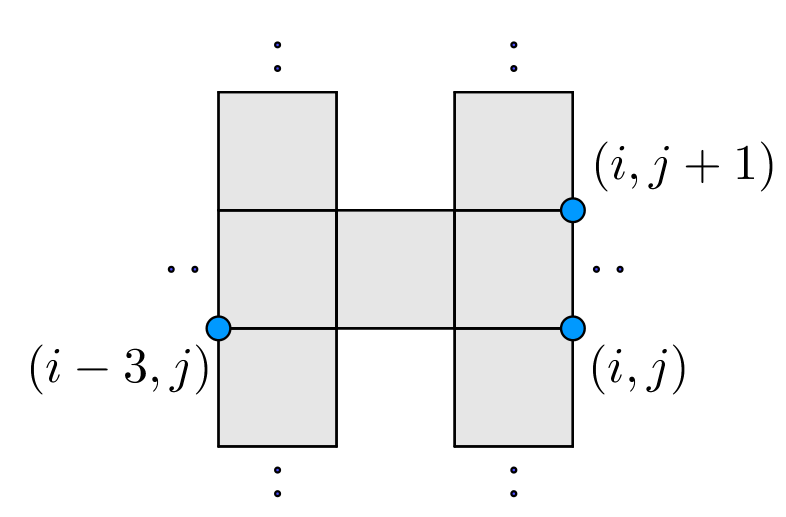}
		\caption{Sub-polyomino when $a=i-3$ and $b=j+1$.}
		\label{Figure: step in grid polyominoes if and only if - uno}
	\end{figure}
	\end{lemma}

	\begin{proof}
	$\Leftarrow)$ It is trivial.\\
	$\Rightarrow)$  We prove that if $a = i - 3$ and $b = j + 1$, then the conditions (1) and (2) hold. Let $A$ and $B$ be the cells whose lower left corners are $(i-3,j)$ and $(i-1,j)$, respectively. Obviously, $\rank([A,B]) = 3$. The cell of $\ZZ^2$ to the North of $A$ belongs to $\cP$, otherwise $(i-2,j) \in F$, which cannot be possible, since $F'$ is a step of $F$. Similarly, we observe that the cell to the South of $B$ belongs to $\cP$, otherwise $(i-1,j) \in F$. Consequently, since $\cP$ is a grid polyomino, the cell to the South of $A$ and to the North of $B$ both belong to $\cP$. Hence, we get claim (1). Now suppose by contradiction that no vertex in $V_1$ is in $F$. From this assumption, it follows that the lower right corner of any inner interval of $\cP$ having $(i-1,j)$ as the upper left corner is not in $F$. Moreover, since $(i,j) \in F$, the upper left corner of any inner interval of $\cP$ having $(i-1,j)$ as the lower right corner does not belong to $F$. Therefore, $(i-1,j) \in F$, which cannot be possible. Hence, $V_1 \cap F \neq \emptyset$. By similar arguments, we can prove that $V_2 \cap F \neq \emptyset$. In conclusion, we deduce claim (2).       
	\end{proof}

	\begin{lemma}\label{Lemma: Structure generalized step - particular structure 2}
		Let $\cP$ be a grid polyomino and $\Delta_{\cP}$ be the attached simplicial complex. Let $F'=\{(a,j),(i,j),(i,b)\}$ be a generalized step of a facet $F$ of $\Delta_{\cP}$, where $(i,j)$ is the lower right corner of $F'$. Then $a=i-1$ and $b=j+3$ if and only if the following hold:
		\begin{enumerate}
			\item the vertices of $F'$ are arranged in a sub-polyomino of $\cP$ like in Figure \ref{Figure: step in grid polyominoes if and only if - due};
			\item if $H_1=\{(p,j+3)\in V(\cP):p<i-1\}$ and $H_2=\{(p,j)\in V(\cP):p>i\}$, then $H_1\cap F\neq \emptyset$ and $H_2\cap F\neq \emptyset$.
		\end{enumerate}
		\begin{figure}[h!]
			\centering
			\includegraphics[scale=0.6]{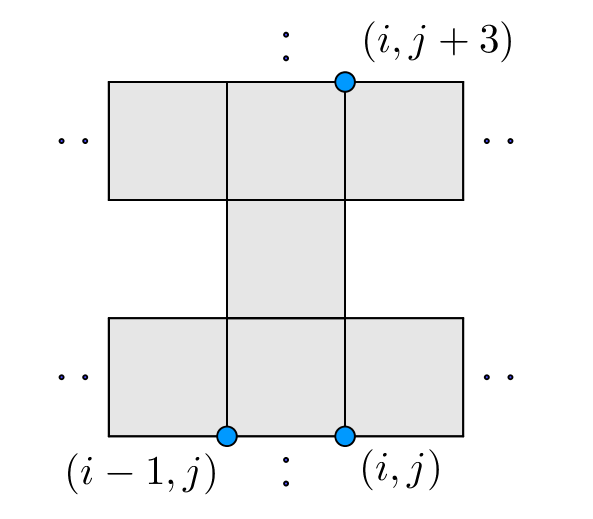}
			\caption{Sub-polyomino when $a=i-1$ and $b=j+3$.}
			\label{Figure: step in grid polyominoes if and only if - due}
		\end{figure}
	\end{lemma}
	
	\begin{proof}
	It can be proved following the same arguments as for Lemma~\ref{Lemma: Structure generalized step - particular structure 1}. 
	\end{proof}
	
    Let $\cP$ be a grid polyomino and we recall the lexicographic total order on the set of the facets of $\Delta_{\cP}$, given in \cite[Definition 3.10]{Navarra_Rizwan}. Let $a,b\in V(\cP)$ with $a=(i,j)$ and $b=(k,l)$, we say $b<_{V(\cP)}a$ if $j > l$, or, $j = l$ and $i > k$. Set the dimension of $\Delta_{\cP}$ by $d-1$. We define a total order on $\cF_{\cP}$. Let $F=\{a_1,\dots, a_d\}$ and $G=\{b_1,\dots, b_d\}$ be two facets of $\Delta$, where $a_{i+1}<_{V(\cP)}a_i$ and $b_{i+1}<_{V(\cP)}b_{i}$ for all $i=1,\dots,d-1$, and $k$ be the smallest integer in $[d]$ such that $b_k\neq a_k$. Then we define $F<_{\mathrm{lex}} G$ if the $a_k<_{V(\cP)} b_k$.  

 \begin{exa}\rm 
Let $\cP$ be the grid polyomino in Figure \ref{Figure: exa for order facet}. Consider the following two facets $F$ and $G$ of the simplicial complex attached to $\cP$:
\begin{itemize}
    \item[-] $F=\{(7,4),(5, 4),(3, 4),(5, 3),(3, 3),(2, 3),(1, 3),(7, 2),(6, 2),(5, 2),(4,2),(3,2),(3,1),(1,1)\}$;
    \item[-] $G=\{(7,4),(5, 4),(5, 3),(4,3), (3, 3),(2, 3),(7, 2),(4,2),(3,2),(7,1),(6,1),(3,1),(2,1),(1,1)\}$.
\end{itemize}
 Then $G<_{\mathrm{lex}} F$, because in $F$ and $G$ the first different vertices from left to right are in the 3-rd position and $(5,3)<_{V(\cP)}(3,4)$.
 
 \begin{figure}[h]
		\centering
		\subfloat[Facet $F$]{\includegraphics[scale=0.7]{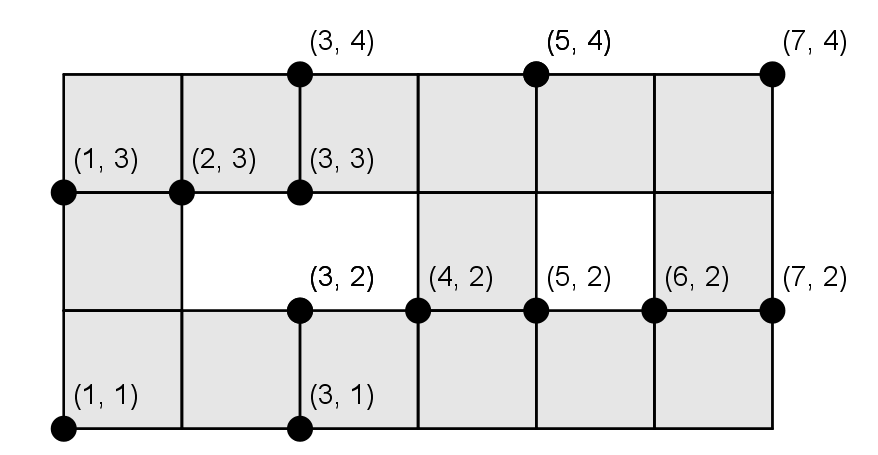}}
		\quad
		\subfloat[Facet $G$]{\includegraphics[scale=0.7]{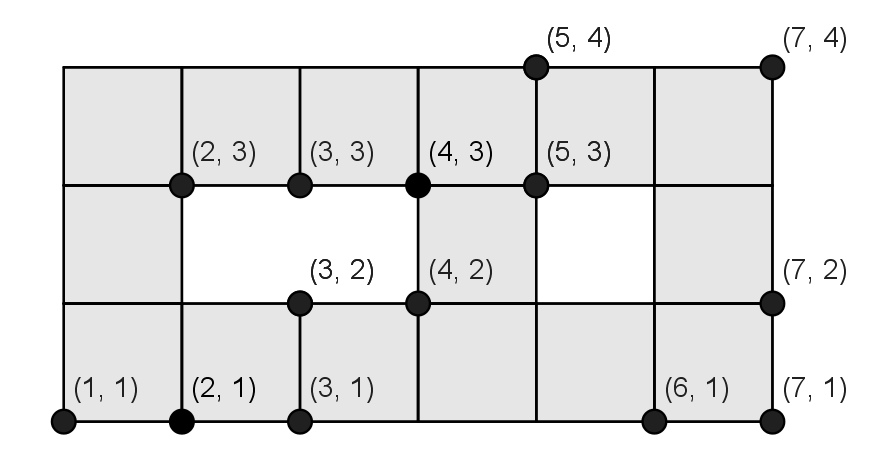}}
		\caption{}
		\label{Figure: exa for order facet}
	\end{figure} 
 \end{exa}
	
	 In the following result we prove that, if $\cP$ is a grid polyomino, then the set of the facets ordered in descending with respect to $<_{\mathrm{lex}}$ forms a shelling order for $\Delta_{\cP}$, as proven for frame polyominoes in \cite{Navarra_Rizwan}.
		
	\begin{thm}\label{Thm: The lexicographic order gives a shelling order}
	Let $\cP$ be a grid polyomino. Let $\Delta_{\cP}$ be the simplicial complex attached to $\cP$ and $\cF_{\cP}=\{F_0, F_1,\dots,F_s\}$ be the set of facets of $\Delta_{\cP}$, where $F_0$ is the facet $ F_{I} \cup \left(\bigcup_{i\in[r]\atop j\in[s]} B_{ij}\right)$ defined in the proof of Theorem $\ref{dim}$ and $F_{i+1}<_{\mathrm{lex}} F_i$ for all $i=0,\dots, s-1$. Then for all $j\in [s]$: $$\langle F_0,\dots,F_{j-1} \rangle \cap \langle F_j\rangle =\langle \big\{F_j\setminus\{v\} : v\ \text{is the lower right corner of a generalized step of}\ F_j\big\}\rangle,$$
    In particular, $(\cF_{\cP}$, $<_{\mathrm{lex}})$ provides a shelling for $\Delta_{\cP}$. 
	\end{thm}

	\begin{proof}
	The proof follows the same strategy of \cite[(1) of Theorem 3.12]{Navarra_Rizwan}. Let $j\in [s]$. Denote by $\cC_j$ the simplicial complex generated by all $F_j\setminus\{v\}$, where $v$ is the lower right corner of a generalized step of $F_j$. We need to prove that $\langle F_0,\dots,F_{j-1} \rangle \cap \langle F_j\rangle= \cC_j$.\\
    
	$\supseteq$ Let $\{(a,k),(i,k),(i,b)\}$ be a generalized step of $F_j$ where $(i,k)$ is the lower right corner. From condition (2) of Definition~\ref{Definition: Generalized Step of a facet}, we know that $[(a,k),(i,b)]$ is an inner interval of $\cP$. Assume firstly that $[(a,k),(i,b)]$ is a cell $C$ of $\cP$. 
    Observe that no vertex of $\cP$ in $\mathcal{A}=\{(i,q):q<k\}$ and in $\cB=\{(p,k):p>i\}$ belongs to $F_j$, since $(i-1,k),(i,k+1)\in F_j$. Moreover, $(i-1,k+1)\notin F_j$ because $(i,k)\in F_j$. Consider $H=F_j\setminus\{(i,k)\}\cup \{(i-1,k+1)\}$. Recall that $\Delta_{\cP}$ is pure, that is all facets in $\Delta_{\cP}$ have the same cardinality.  Observe that $\vert H\vert=\vert F_j\vert$ and there does not exist any inner interval of $\cP$ having $(i-1,k+1)$ and a vertex $w\in F\setminus\{(i,k)\}$ as anti-diagonal corners, since no element of $\mathcal{A}\cup \cB$ is in $F_j$. Moreover, since we have defined $H$ by removing $(i,k)$ from $F_j$ and adding $(i-1,k+1)$ to it, we have that $F_j<_{\mathrm{lex}} H$, so there exists $h\in \{0,\dots,j-1 \}$ such that $H=F_h$. Observe that $F_j\setminus \{(i,k)\}\subset F_j$ trivially and $F_j\setminus \{(i,k)\} \subset H= F_h \in \langle F_0,\dots,F_{j-1} \rangle$. Hence, $F_j\setminus \{(i,k)\}\in \langle F_0,\dots,F_{j-1} \rangle\cap \langle F_j\rangle$. If $[(a,k),(i,b)]$ is not a cell, then due to be $\cP$ a grid polyomino, $[(a,k),(i,b)]$ is a vertical or a horizontal interval of $\cP$. Then, arguing as done in the previous case, by considering the facet $F_j\setminus \{(i,k)\}\cup \{(a,b)\}$ as $H$, we get that $F_j\setminus \{(i,k)\}\in \langle F_0,\dots,F_{j-1} \rangle\cap \langle F_j\rangle$. Therefore, we have that $\cC_j\subseteq \langle F_0,\dots,F_{j-1} \rangle\cap \langle F_j\rangle$. \\

	 $\subseteq$ For the other inclusion, let us consider $G \in \langle F_0, \dots, F_{j-1} \rangle \cap \langle F_j \rangle$. Since $G \in \langle F_j \rangle$, we have $G \subset F_j$. Assume that $G = F_j \setminus {(i, k)}$. We will prove that $(i, k)$ is the lower right corner of a generalized step of $F_j$. Once this case is understood, it becomes evident that when $|G| < |F| - 1$, i.e., $G = F \setminus \{v_1, v_2, \dots, v_t\}$, it is easy to show that there exists $q \in [t]$ such that $v_q$ is the lower right corner of a generalized step of $F_j$.\\
     Let us starting with the observation that $F_j\setminus \{(i,k)\}\in\langle F_0,\dots,F_{j-1} \rangle$ implies that there exists $h\in \{0,\dots, F_{j-1}\}$ such that $F_j\setminus \{(i,k)\} \subset F_h$. Hence, $F_h$ consists of all points of $F_j$ except $\{(i,k)\}$ and of another vertex $w=(p,q)\notin F_j$, with either $q>k$ or $q=k$ and $p>i$; that means we can obtain $F_h$ by replacing $(i,k)$ with $w$ in $F_j$, and vice versa. Imagine that $F_j\setminus \{(i,k)\}$ is figured out in $V(\cP)$, and we wonder where we can place $w$ to obtain $F_h$. Using the notation in Definition \ref{grid}, for $\alpha\in[r]$ and $\beta\in[s]$, $U_{\alpha\beta}=[a_{\alpha\beta}, c_{\alpha\beta}] \cup [c_{\alpha\beta}, b_{\alpha\beta}]$ and $B_{\alpha\beta}=[a_{\alpha\beta}+(1,0), d_{\alpha\beta}] \cup [d_{\alpha\beta}, b_{\alpha\beta}-(0,1)]$, let us observe that if $(i,k)\in\left(\bigcup_{\alpha\in [r]\atop \beta\in[s]} B_{\alpha\beta}\right)\cup [(1,1),(1,n)]\cup[(1,n),(m,n)]$, then $F_h$ cannot exist due to the maximality of $F_j$. Therefore, $(i,k)\in\left(\bigcup_{\alpha\in [r]\atop \beta\in[s]} U_{\alpha\beta}\right)\cup [(2,1),(m,1)]\cup[(m,1),(m,n-1)]$. For any position of $(i,k)$, the reader can easily verify that the operation to get $F_h$ from $F_j$ replacing $(i,k)$ by the vertex $w$ can be done just when $(i,k)$ is the lower right corner of a generalized step of $F_j$, so that we have $F_h=F_j\setminus \{(i,k)\}\cup \{(a,b)\}$. Hence, $G\in C_j$.\\

     Therefore we have that $\langle F_0,\dots,F_{j-1} \rangle \cap \langle F_j\rangle= \cC_j$. Finally, we can conclude that $(\cF_{\cP}$, $<_{\mathrm{lex}})$ provides a shelling for $\Delta_{\cP}$.  
	\end{proof}

\begin{qst}\rm \label{Question 1}
    As a first step toward proving \cite[Conjecture 4.5]{Trento3} or \cite[Conjecture 4.9]{Navarra_Rizwan}, it would be useful to characterize the class of polyominoes satisfying \cite[Remark 4.2]{Qureshi} for which the set of facets of the attached simplicial complex, ordered with respect to $<$, forms a shelling for $\Delta_{\cP}$.
\end{qst}

    \section{Hilbert series and rook polynomial of a grid polyomino}\label{Sec: Rook}
    
    This section is dedicated to studying the $h$-polynomial of the coordinate ring of a grid polyomino in relation to its rook polynomial. To begin, we first recall some fundamental concepts related to the Hilbert-Poincar\'{e} series of a graded $K$-algebra. \\
     Let $R$ be a graded $K$-algebra, and let $I$ be a homogeneous ideal of $R$. The quotient $R/I$ naturally inherits a graded $K$-algebra structure as $ \bigoplus_{k \in \mathbb{N}} (R/I)_k $. The formal series $ \mathrm{HP}_{R/I}(t) = \sum_{k \in \mathbb{N}} \dim_K (R/I)_k t^k $ is defined as the Hilbert-Poincar\'{e} series of $R/I$. By the Hilbert-Serre theorem, there exists a unique polynomial $ h(t) \in \mathbb{Z}[t] $, called the $h$-polynomial of $R/I$, such that $ h(1) \neq 0 $ and $ \mathrm{HP}_{R/I}(t) = \frac{h(t)}{(1-t)^d} $, where $d$ is the Krull dimension of $R/I$. Moreover, if $R/I$ is Cohen-Macaulay, then $ \mathrm{reg}(R/I) = \deg h(t) $. 
     
     Next, we introduce some definitions related to the rook polynomial of a polyomino $\mathcal{P}$. Two rooks in $\mathcal{P}$ are said to be in \textit{non-attacking position} in $\cP$ if they do not occupy the same row or column in $\mathcal{P}$. A $k$-rook configuration in $\mathcal{P}$ is an arrangement of $k$ rooks in non-attacking positions. For instance, see Figure \ref{Figura:esempio rook configuration}. The rook number, denoted by $ r(\mathcal{P})$, is the maximum number of rooks that can be placed in $\mathcal{P}$ in non-attacking positions. We denote by $ \mathcal{R}(\mathcal{P},k)$ the set of all $k$-rook configurations in $\mathcal{P}$ and set $ r_k = \vert \mathcal{R}(\mathcal{P},k) \vert $, for all $ k \in \{0, \dots, r(\mathcal{P})\} $ (with the convention $ r_0 = 1 $). The \textit{rook polynomial} of $\mathcal{P}$ is the polynomial in $\mathbb{Z}[t]$ defined as $ r_{\mathcal{P}}(t) = \sum_{k=0}^{r(\mathcal{P})} r_k t^k $. Finally, the reader may consult \cite{Navarra_Rizwan, Parallelogram Hilbert series} for more details about a nice generalization of the rook polynomial to the case of non-thin polyominoes, as well as the recent papers \cite{NQR1, NQR2}, which appeared during the review process.
    
    \begin{figure}[h]
		\centering
		\includegraphics[scale=0.6]{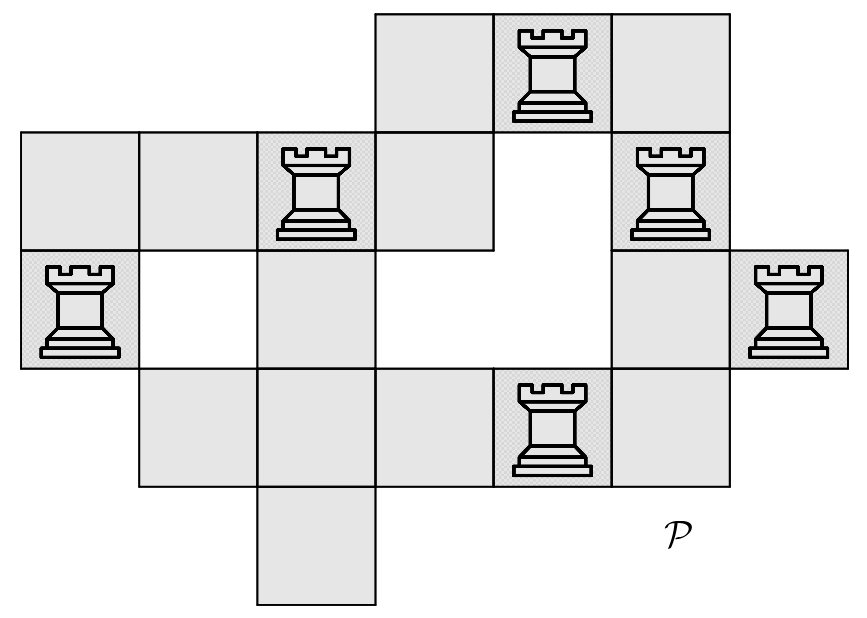}
		\caption{An example of a $6$-rook configuration in $\cP$.}
		\label{Figura:esempio rook configuration}
	\end{figure}

The key to proving that the $h$-polynomial of the coordinate ring of a grid polyomino equals its rook polynomial lies in relating the facets with $k$ generalized steps of $\Delta_{\cP}$ to the $k$-rook configurations of $\cP$. Even if Question \ref{Question 1} holds, it represents only an initial step; proving \cite[Conjecture 4.5]{Trento3}, or more in general \cite[Conjecture 4.9]{Navarra_Rizwan}, requires significantly more intricate arguments. One of the main challenges is to establish a unique correspondence between these facets and the $k$-rook configurations, which is the novel contribution of this work. For frame polyominoes, this was linked to the maximal chains of parallelogram polyominoes, as discussed in Remark 4.5 and the ``Surjectivity of $\psi$" part of Theorem 4.6 in \cite{Navarra_Rizwan}. In contrast, for grid polyominoes, we introduce a new method of description, which is the focus of this section.

Let us start pointing out the only configurations where a change of direction can occur in a grid polyomino.

	\begin{lemma}\label{Lemma: A is a cell of a change of direction}
		Let $\cP$ be a grid polyomino and $\Delta_{\cP}$ be the attached simplicial complex. Let $F'=\{(a,j),(i,j),(i,b)\}$ be a generalized step of a facet $F$ of $\Delta_{\cP}$, where $(i,j)$ is the lower right corner of $F'$. Denote by $A$ the cell with lower right corner $(i,j)$. If $a\neq i-1$ and $b\neq j+1$, then $A$ is a cell of a change of direction in $\cP$.
	\end{lemma}
	
	\begin{proof}
		By Lemma \ref{Lemma: Structure generalized step}, we have that either $a\in \{i-2,i-3\}$ and $b=j+1$ or $a=i-1$ and $b\in \{j+2,j+3\}$. We may assume that $a=i-2$ and $b=j+1$ because all other cases can be proved similarly. Suppose by contradiction that $A$ is not a cell of a change of direction in $\cP$. Then $A$ belongs to a cell interval of $\cP$ as shown in Figure \ref{Figure: A is a cell of a change direction}. In such a case, we must have that $(i-1,j)\in F$ due to the maximality of $F$, but this is contradiction, since $F'$ is a generalized step of $F$. Hence, we get the desired result.
		
		\begin{figure}[h!]
			\centering
			\includegraphics[scale=0.55]{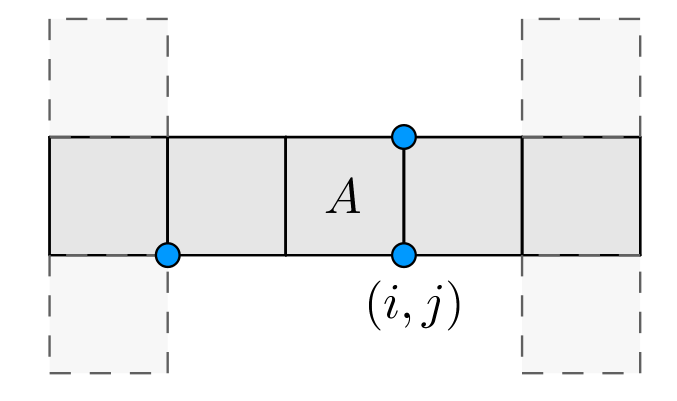}
			\caption{The dotted cells could belong to $\cP$.}
			\label{Figure: A is a cell of a change direction}
		\end{figure} 
	\end{proof}

	 In what follows we define a configuration of $k$ rooks in non-attacking position in $\cP$ to a facet of $\Delta_{\cP}$ having $k$ generalized steps.
	
	\begin{defn}\rm \label{Defn: rook configuration to a facet}
	Let $\Delta_{\cP}$ be the simplicial complex attached to $\cP$ and $F$ be a facet of $\Delta_{\cP}$ with $k$ generalized steps, where $k\geq 1$. We want to attach a $k$-rook configuration of $\cP$ to $F$. For that, we define a rook placement for every generalized step of $F$. Let $F'$ be a generalized step of $F$ whose lower right corner is $(i,j)$. We distinguish several cases depending on the structure of $F'$ by Lemmas \ref{Lemma: Structure generalized step}, \ref{Lemma: Structure generalized step - particular structure 1} and \ref{Lemma: Structure generalized step - particular structure 2}.
	\begin{enumerate}
		\item If $F'=\{(i-1,j),(i,j),(i,j+1)\}$, then we place a rook in the cell having $(i,j)$ as lower right corner. 
		\item Suppose that $F'\neq\{(i-1,j),(i,j),(i,j+1)\}$. Then the cell $C$ having $(i,j)$ as lower right corner is a cell of a change of direction in $\cP$ from Lemma \ref{Lemma: A is a cell of a change of direction}.
		\begin{enumerate}
			\item Assume $F'=\{(i-2,j),(i,j),(i,j+1)\}$. If $C$ is an extremal cell of the change of direction then we place a rook in $C$. Otherwise, if $C$ is the middle cell then a rook is placed in the cell with $(i-1,j)$ as lower right corner (see Figure \ref*{Figure: rook place i-2} (A) and (B)).
			\item If $F'=\{(i-3,j),(i,j),(i,j+1)\}$, then we have just the case described in Lemma \ref{Lemma: Structure generalized step - particular structure 1}, so we place a rook in the cell with $(i-1,j)$ as lower right corner (see Figure \ref*{Figure: rook place i-2} (C)).
			
				\begin{figure}[h!]
				\centering
				\subfloat[]{\includegraphics[scale=0.6]{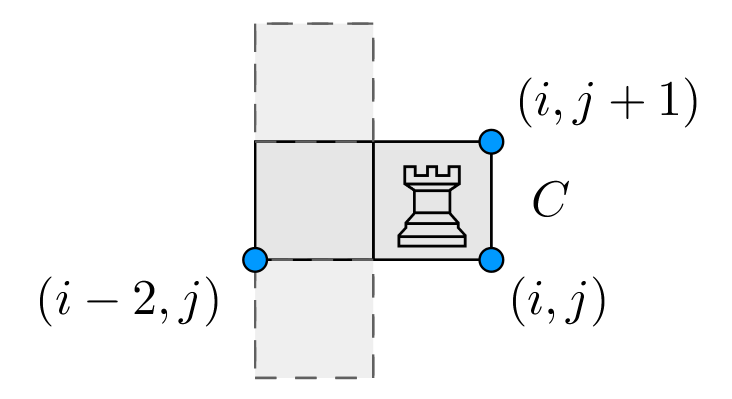}}\
				\subfloat[]{\includegraphics[scale=0.6]{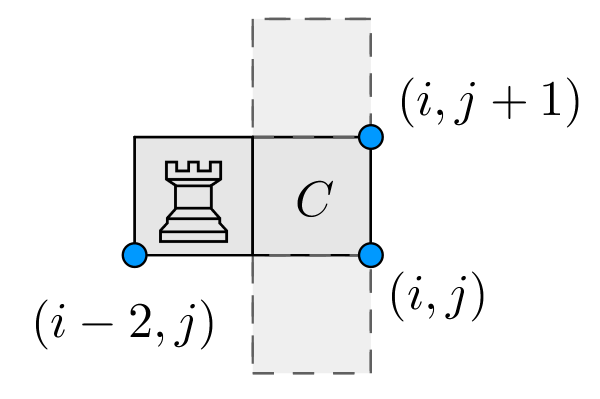}}\
				\subfloat[]{\includegraphics[scale=0.6]{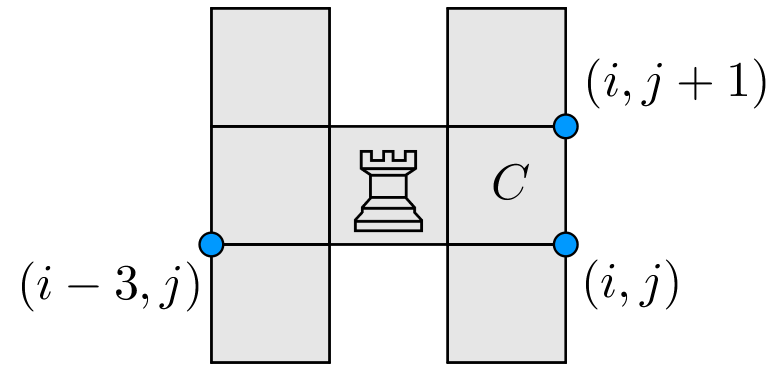}}
				\caption{Rook placements for the cases (a) and (b).}
				\label{Figure: rook place i-2}
			\end{figure} 
		
			\item Assume $F'=\{(i-1,j),(i,j),(i,j+2)\}$. As done in case (a), if $C$ is an extremal cell of the change of direction then we place a rook in $C$, otherwise in the cell with $(i,j+1)$ as lower right corner (see Figure \ref{Figure: rook place j+2} (A) and (B)). 
		\item If $F'=\{(i-1,j),(i,j),(i,j+3)\}$, then we have just the case described in Lemma \ref{Lemma: Structure generalized step - particular structure 2}, so we place a rook in the cell with $(i,j+1)$ as lower right corner (see Figure \ref{Figure: rook place j+2} (C)).
			
	\begin{figure}[h!]
		\centering
				\subfloat[]{\includegraphics[scale=0.6]{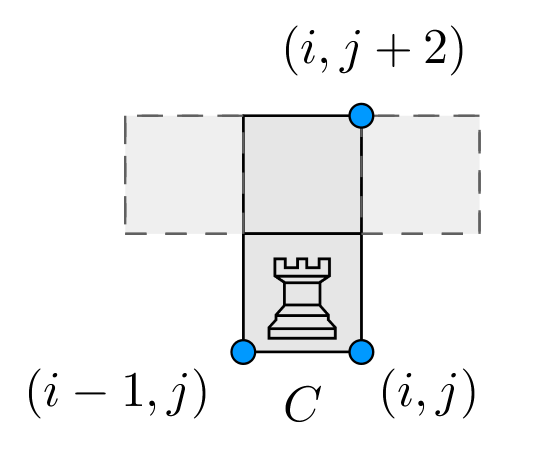}}\quad
				\subfloat[]{\includegraphics[scale=0.6]{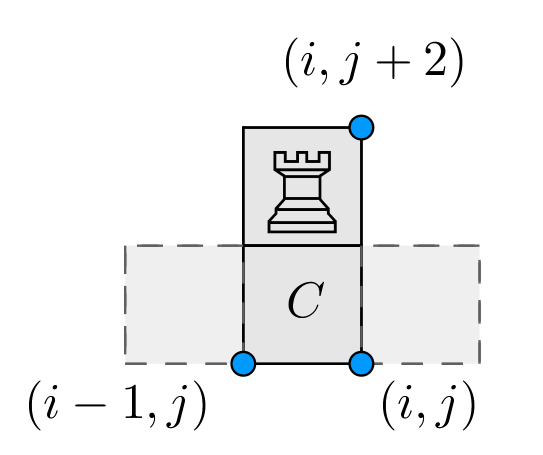}}\quad
				\subfloat[]{\includegraphics[scale=0.6]{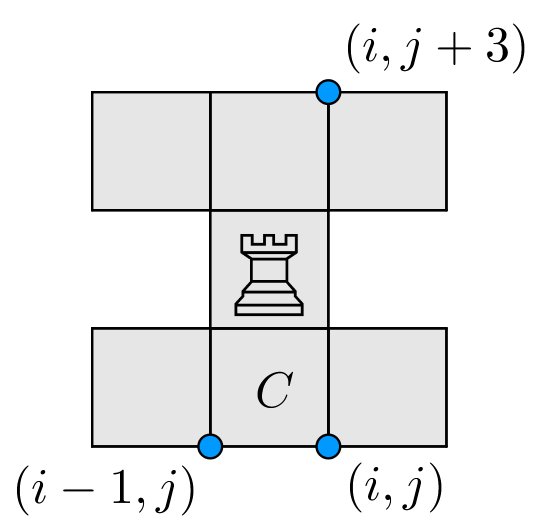}}
				\caption{Rook placements for the cases (c) and (d).}
				\label{Figure: rook place j+2}
			\end{figure} 
		\end{enumerate}
	\end{enumerate}
	We provide Table \ref{Table3}, which illustrates the steps and the corresponding rook placements at the changes of direction in a grid polyomino.
    
Finally, the reader can easily verify that if $F$ is a facet of $\Delta_{\cP}$ with $k$ generalized steps ($k \geq 1$) and $\cR(F)$ is the associated configuration of $k$ rooks defined as before, then the rooks in $\cR(F)$ are in a non-attacking position. Therefore, it is natural to define a function $\cR(-)$ that assigns to a facet $F$ of $\Delta_{\cP}$ with $k$ generalized steps a configuration $\cR(F)$ of $k$ rooks in non-attacking position in $\cP$, as defined previously (with the convention that $\cR(F_0)=\emptyset$, where $F_0$ is the facet $ F_{I} \cup \left(\bigcup_{i\in[r]\atop j\in[s]} B_{ij}\right)$ defined in the proof of Theorem $\ref{dim}$).    
	\end{defn}

 	\begin{table}
		\centering
		\renewcommand\arraystretch{1.3}{	
			\begin{tabular}{l|p{2.2cm}|p{2.2cm}|p{2.2cm}|p{2.2cm}|p{2.2cm}|p{2.2cm}}
			\hline & \centering A &\centering B &\centering C &\centering D &\centering E &\centerline	F\\
				 \hline
				I & \begin{minipage}{0.23\textwidth}
					\includegraphics[scale=0.41]{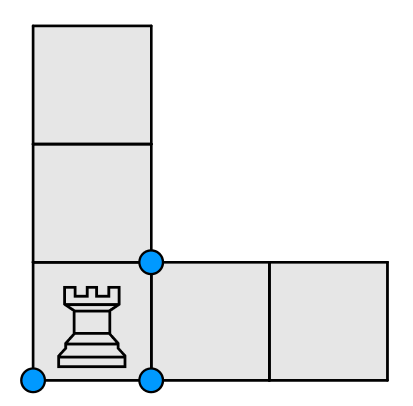}
				\end{minipage}  &  \begin{minipage}{0.23\textwidth}
					\includegraphics[scale=0.41]{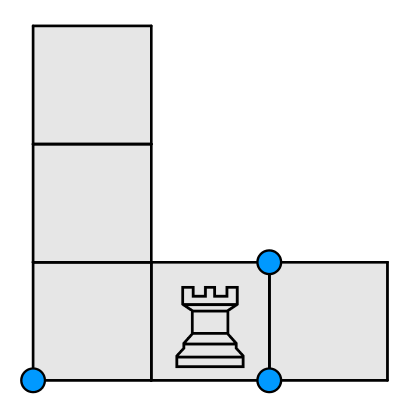}
				\end{minipage}  &  \begin{minipage}{0.23\textwidth}
				\includegraphics[scale=0.41]{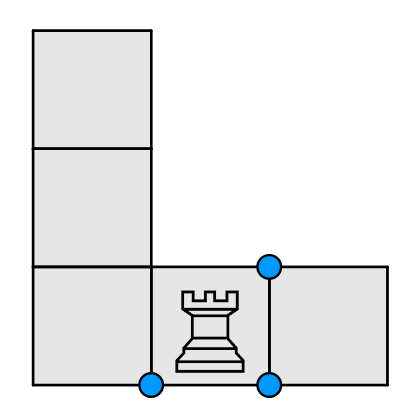}
				\end{minipage}&  \begin{minipage}{0.23\textwidth}
				\includegraphics[scale=0.41]{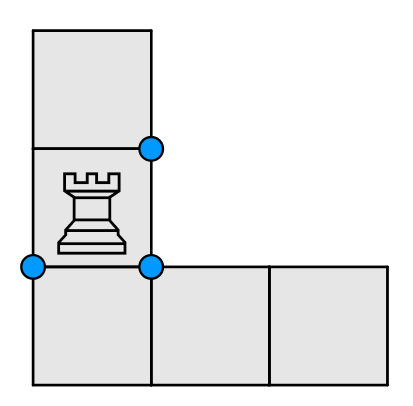}
				\end{minipage}&  \begin{minipage}{0.23\textwidth}
				\includegraphics[scale=0.41]{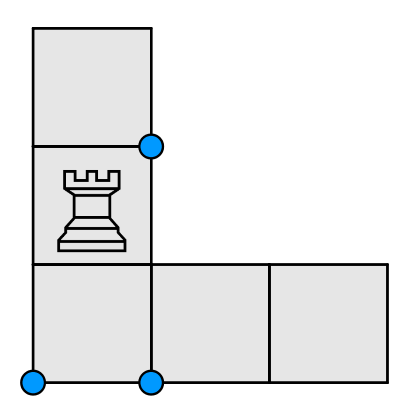}
				\end{minipage}\\
				\hline
				
				II & \begin{minipage}{0.23\textwidth}
					\includegraphics[scale=0.41]{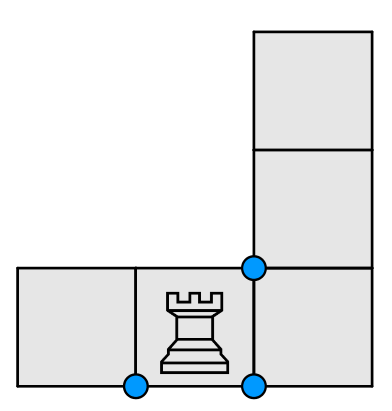}
				\end{minipage}  &  \begin{minipage}{0.23\textwidth}
					\includegraphics[scale=0.41]{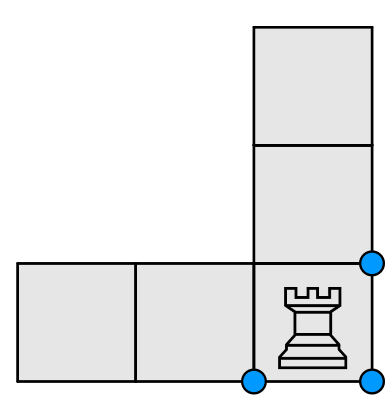}
				\end{minipage} &  \begin{minipage}{0.23\textwidth}
				\includegraphics[scale=0.41]{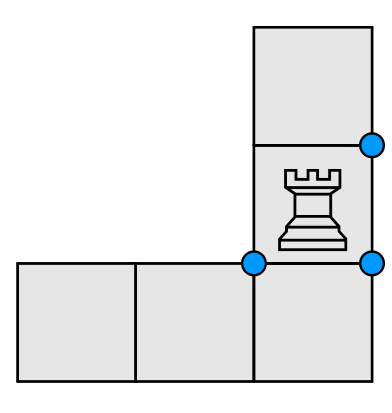}
				\end{minipage} &  \begin{minipage}{0.23\textwidth}
				\includegraphics[scale=0.41]{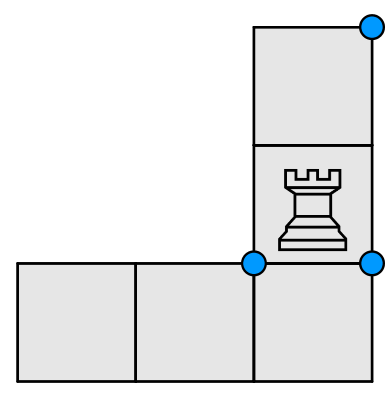}
			\end{minipage}&  \begin{minipage}{0.23\textwidth}
			\includegraphics[scale=0.41]{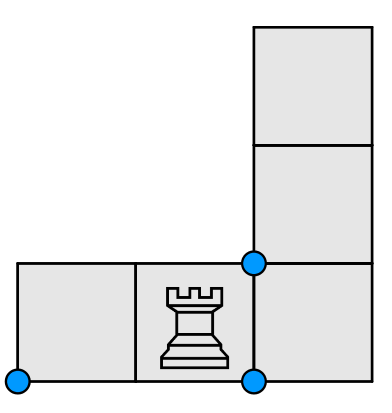}
			\end{minipage}\\
				\hline
				
				III & \begin{minipage}{0.23\textwidth}
					\includegraphics[scale=0.41]{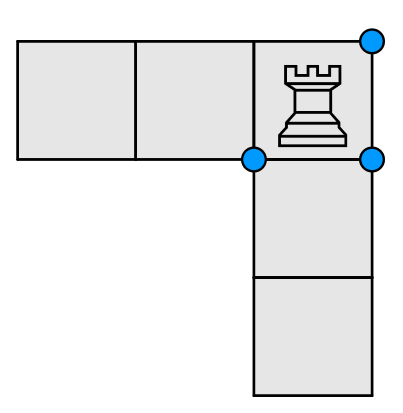}
				\end{minipage}  &  \begin{minipage}{0.23\textwidth}
					\includegraphics[scale=0.41]{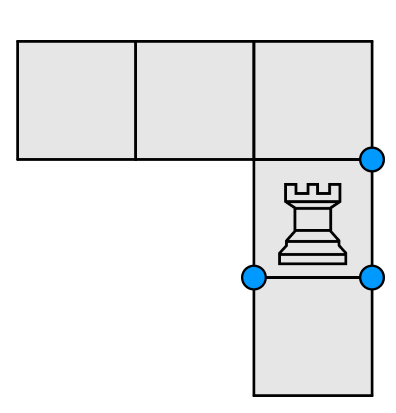}
				\end{minipage} &  \begin{minipage}{0.23\textwidth}
					\includegraphics[scale=0.41]{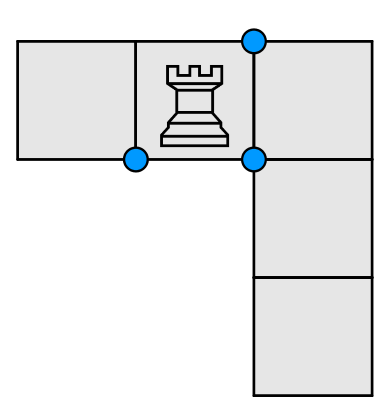}
				\end{minipage} &  \begin{minipage}{0.23\textwidth}
				\includegraphics[scale=0.41]{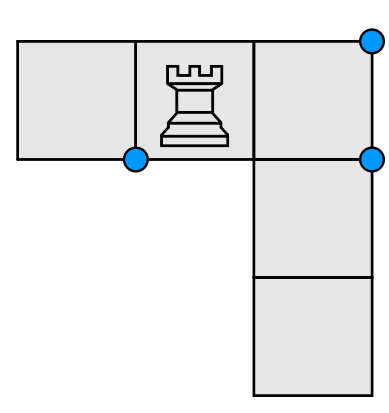}
				\end{minipage} &  \begin{minipage}{0.23\textwidth}
				\includegraphics[scale=0.41]{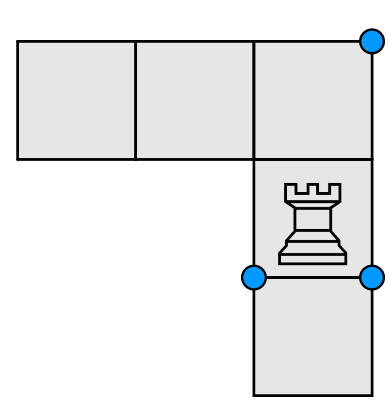}
				\end{minipage}\\
				\hline
				
				IV & \begin{minipage}{0.23\textwidth}
					\includegraphics[scale=0.41]{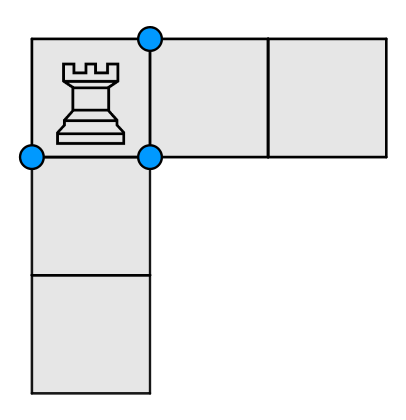}
				\end{minipage}  &  \begin{minipage}{0.23\textwidth}
					\includegraphics[scale=0.41]{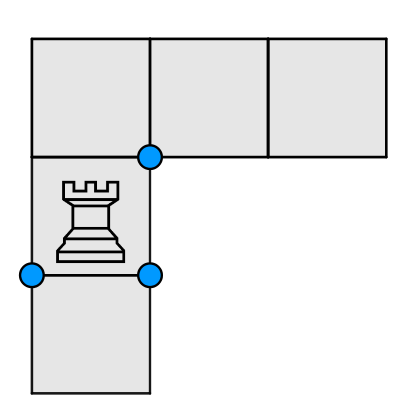}
				\end{minipage} &  \begin{minipage}{0.23\textwidth}
					\includegraphics[scale=0.41]{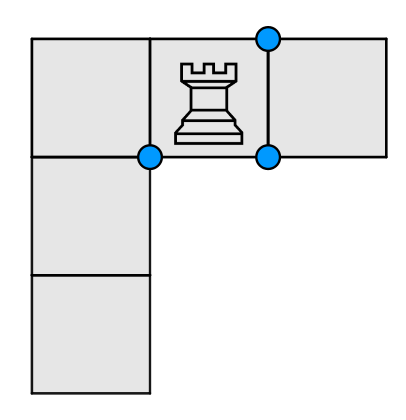}
				\end{minipage} \\
			\hline

				V & \begin{minipage}{0.23\textwidth}
					\includegraphics[scale=0.41]{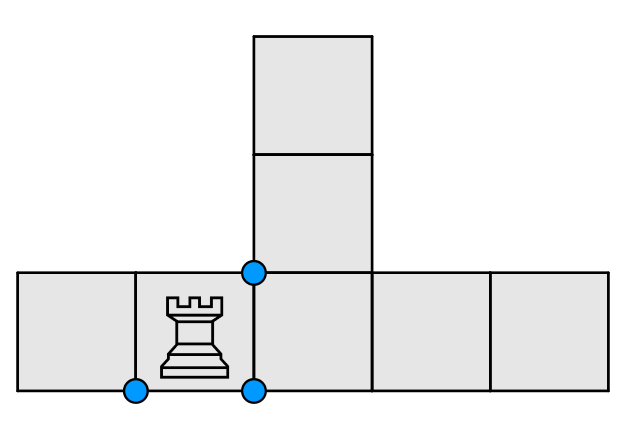}
				\end{minipage} & \begin{minipage}{0.23\textwidth}
				\includegraphics[scale=0.41]{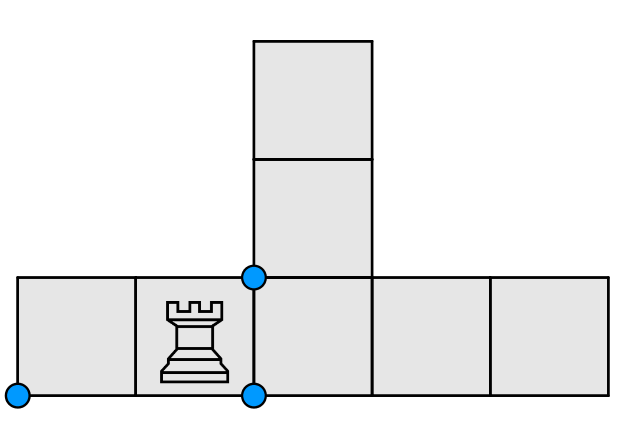}
				\end{minipage}  &  \begin{minipage}{0.23\textwidth}
					\includegraphics[scale=0.41]{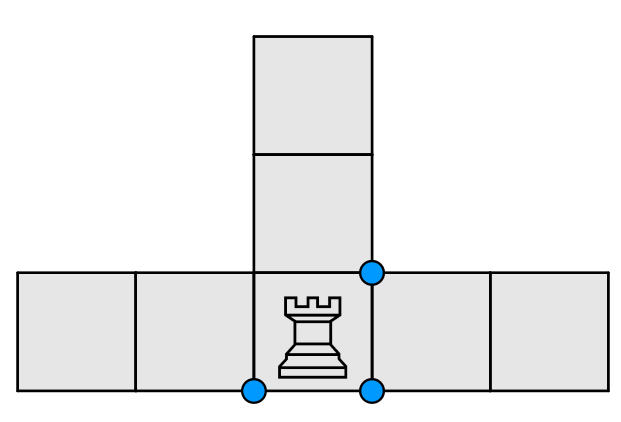}
				\end{minipage}  &  \begin{minipage}{0.23\textwidth}
					\includegraphics[scale=0.41]{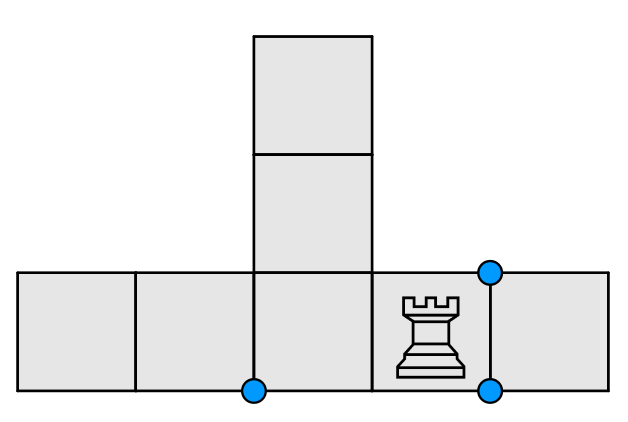}
				\end{minipage}&  \begin{minipage}{0.23\textwidth}
					\includegraphics[scale=0.41]{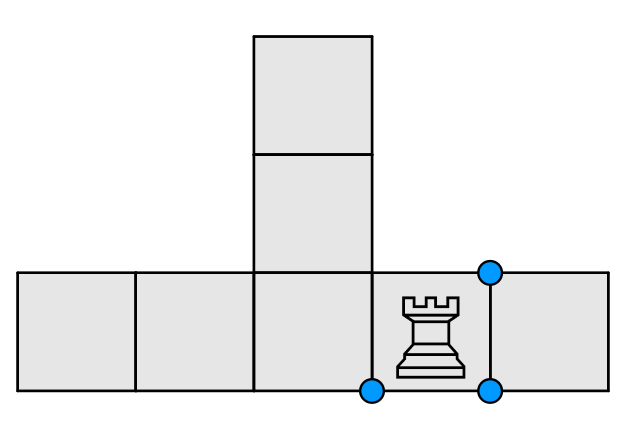}
				\end{minipage}\\
				\hline
				
					V$'$ &  \begin{minipage}{0.23\textwidth}
					\includegraphics[scale=0.41]{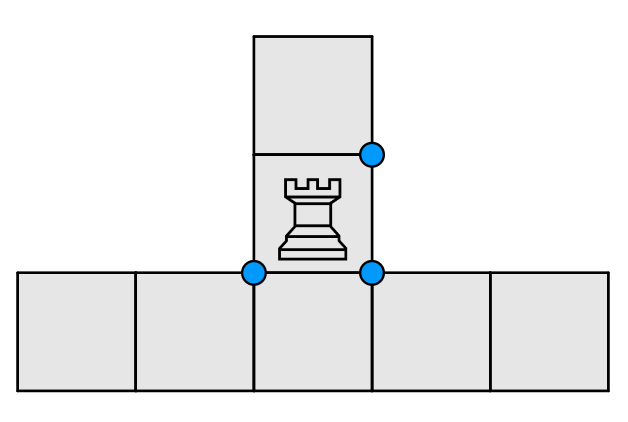}
				\end{minipage}&  \begin{minipage}{0.23\textwidth}
					\includegraphics[scale=0.41]{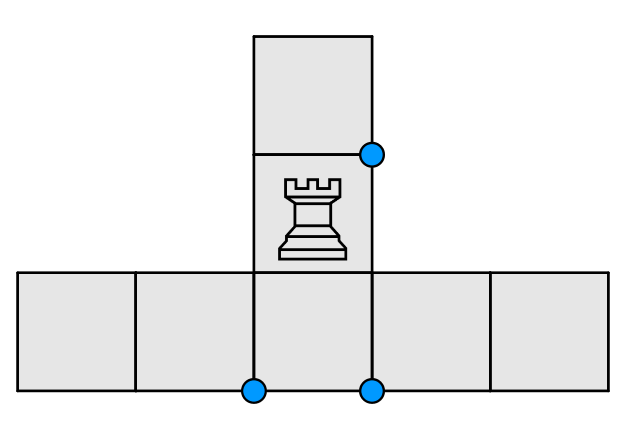}
				\end{minipage}&  \begin{minipage}{0.23\textwidth}
					\includegraphics[scale=0.41]{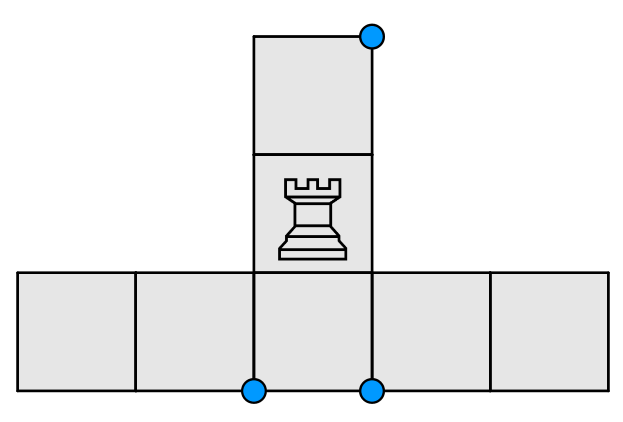}
				\end{minipage}&  \begin{minipage}{0.23\textwidth}
				\includegraphics[scale=0.41]{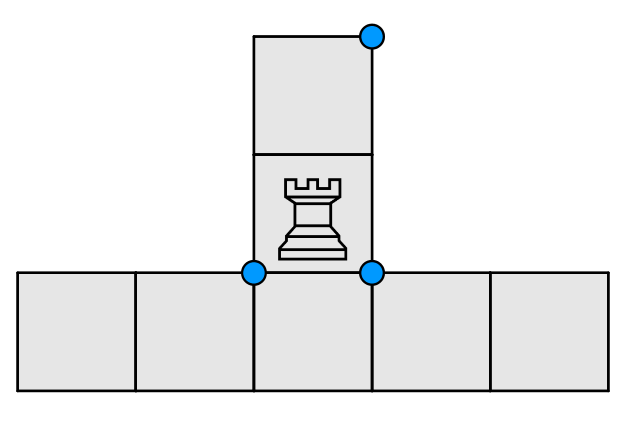}
			\end{minipage}\\
				\hline
				
				VI & \begin{minipage}{0.23\textwidth}
				\includegraphics[scale=0.41]{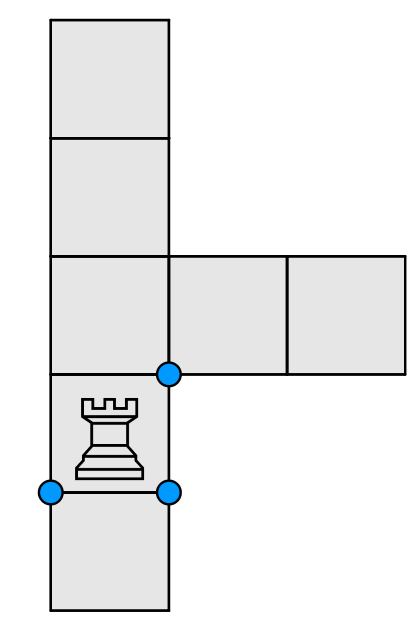}
				\end{minipage}  &  \begin{minipage}{0.23\textwidth}
				\includegraphics[scale=0.41]{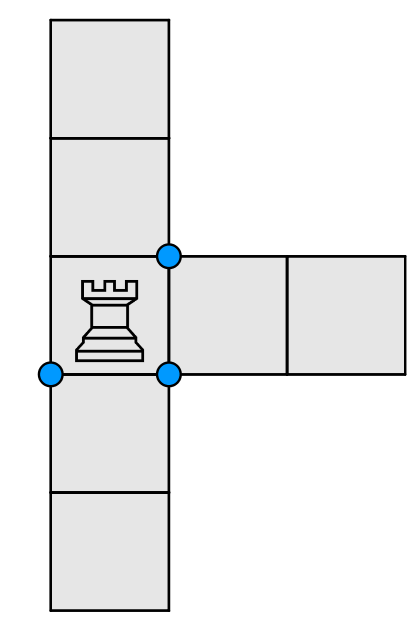}
				\end{minipage} &  \begin{minipage}{0.23\textwidth}
				\includegraphics[scale=0.41]{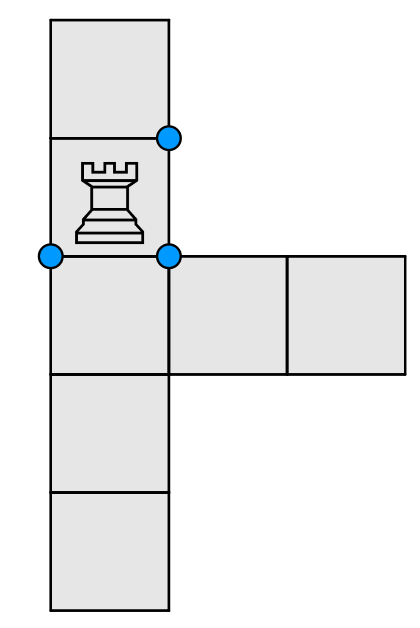}
				\end{minipage}	&  \begin{minipage}{0.23\textwidth}
				\includegraphics[scale=0.41]{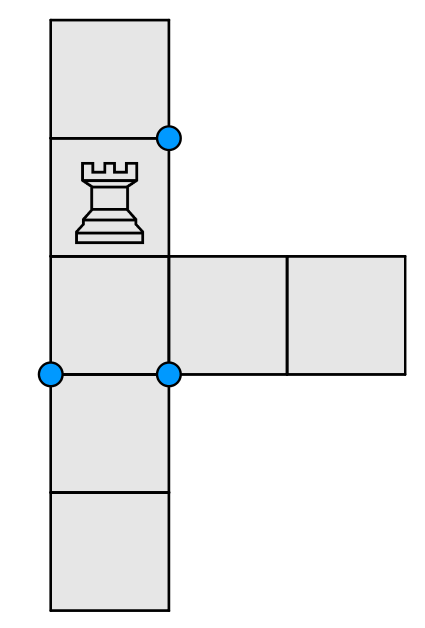}
				\end{minipage} & \begin{minipage}{0.23\textwidth}
				\includegraphics[scale=0.41]{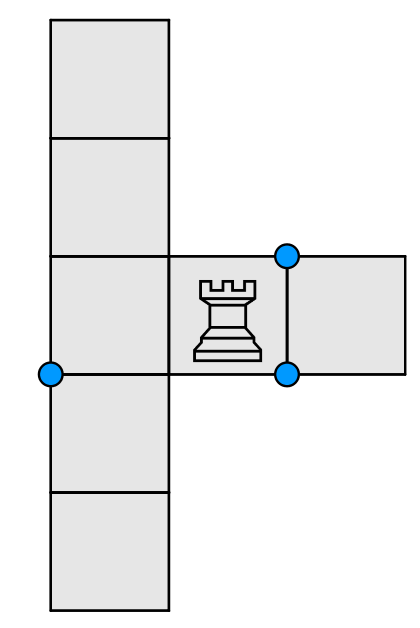}
				\end{minipage}  &  \begin{minipage}{0.23\textwidth}
				\includegraphics[scale=0.41]{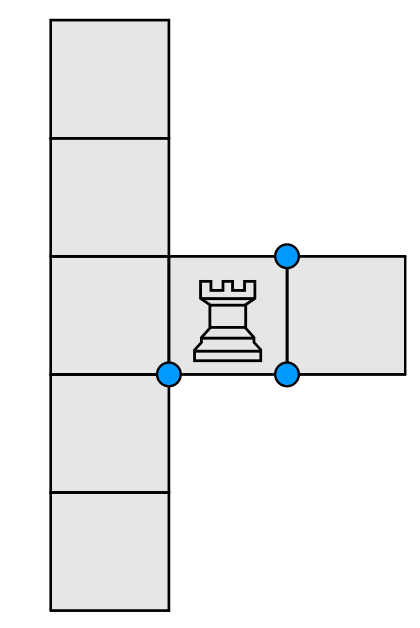}
				\end{minipage}  \\
				\hline
				
				VII & \begin{minipage}{0.23\textwidth}
					\includegraphics[scale=0.39]{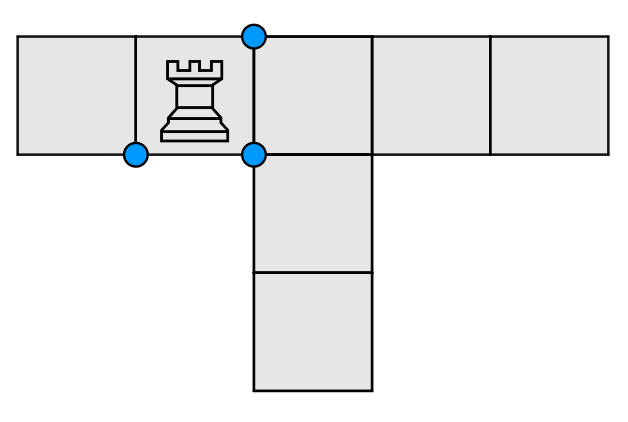}
				\end{minipage}  &  \begin{minipage}{0.23\textwidth}
					\includegraphics[scale=0.39]{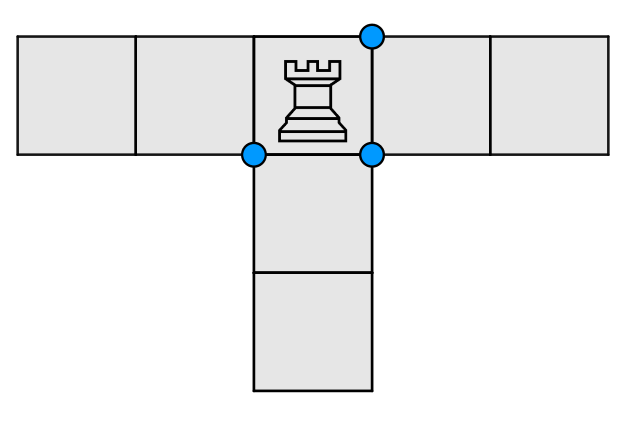}
				\end{minipage} &  \begin{minipage}{0.23\textwidth}
					\includegraphics[scale=0.39]{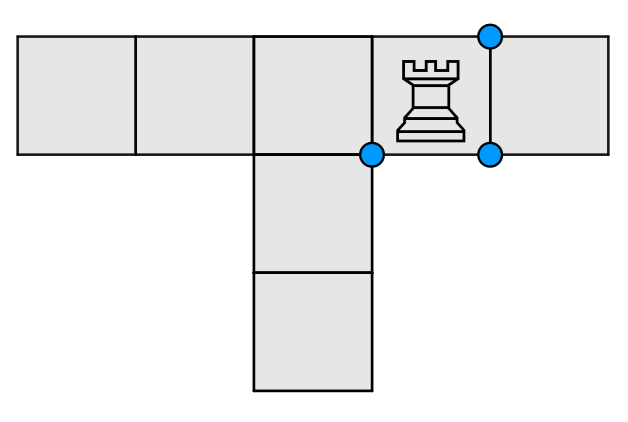}
				\end{minipage} &  \begin{minipage}{0.23\textwidth}
					\includegraphics[scale=0.39]{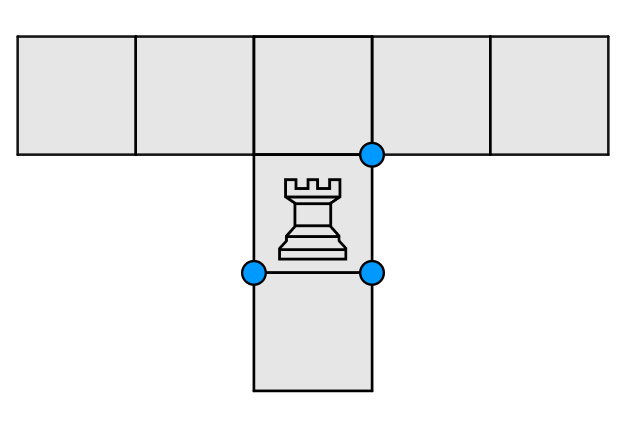}
				\end{minipage} &  \begin{minipage}{0.23\textwidth}
					\includegraphics[scale=0.39]{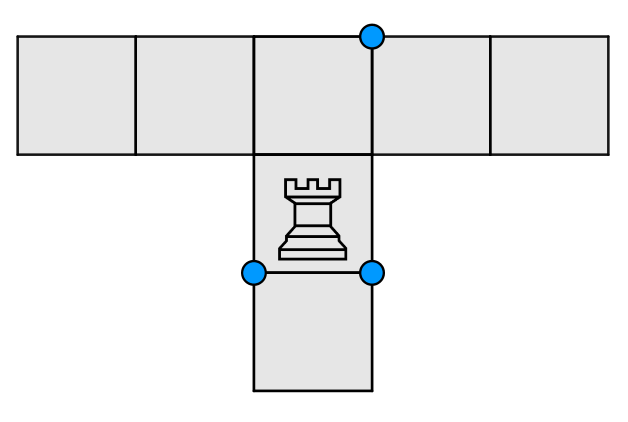}
				\end{minipage}&  \begin{minipage}{0.23\textwidth}
				\includegraphics[scale=0.39]{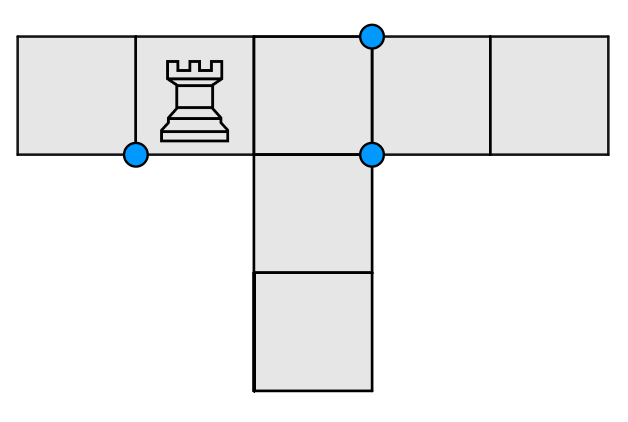}
				\end{minipage}\\
				\hline
				
				VIII &  \begin{minipage}{0.23\textwidth}
					\includegraphics[scale=0.42]{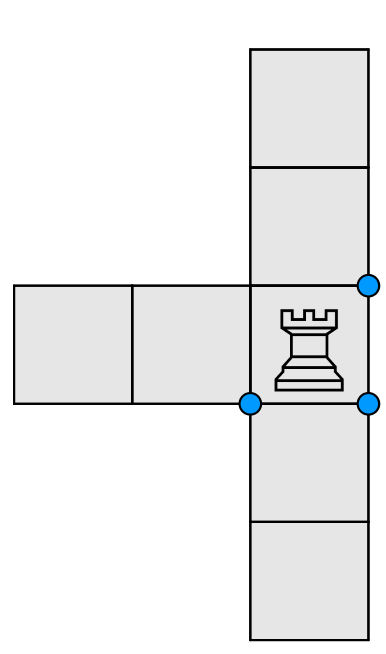}
				\end{minipage} &  \begin{minipage}{0.23\textwidth}
					\includegraphics[scale=0.42]{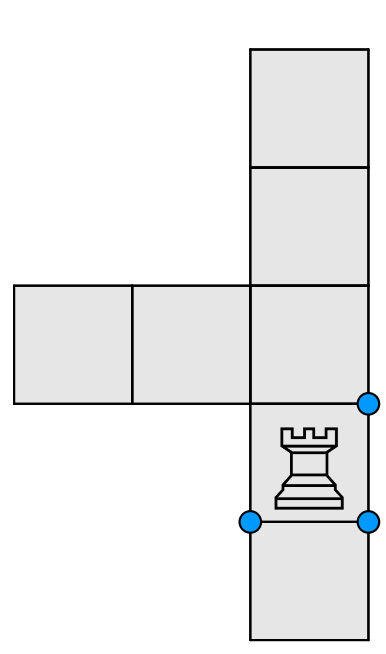}
				\end{minipage} & \begin{minipage}{0.23\textwidth}
				\includegraphics[scale=0.42]{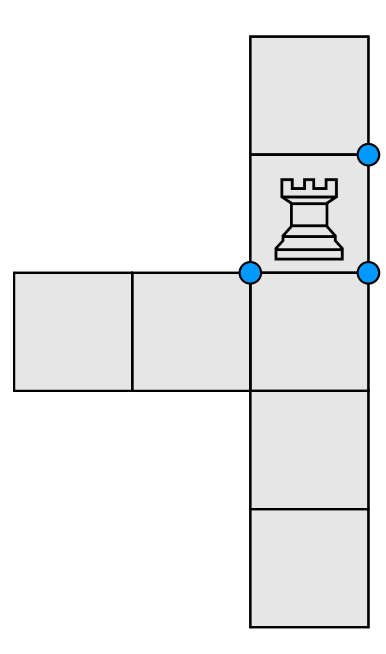}
				\end{minipage}  &   \begin{minipage}{0.23\textwidth}
				\includegraphics[scale=0.42]{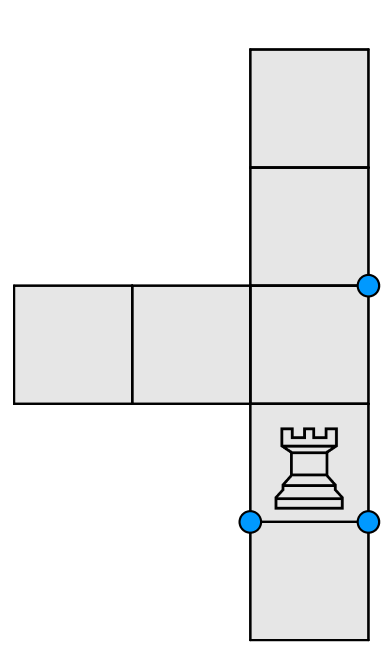}
			\end{minipage} \\
			\hline
			
				VIII$'$ & \begin{minipage}{0.23\textwidth}
				\includegraphics[scale=0.42]{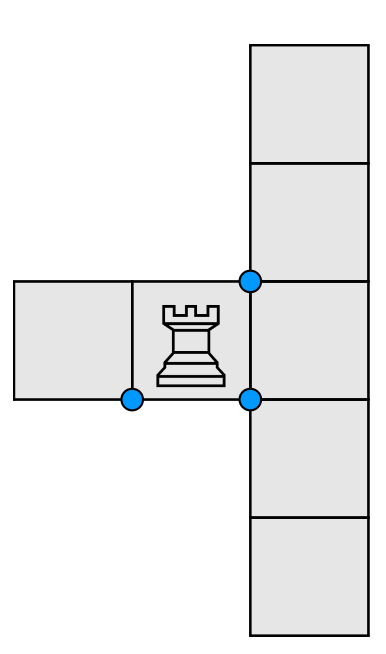}
			\end{minipage}  &          \begin{minipage}           {0.23\textwidth}
				\includegraphics[scale=0.42]{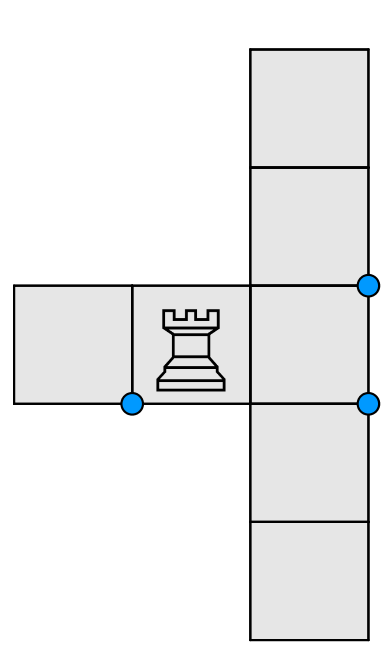}
			\end{minipage} &  \begin{minipage}{0.23\textwidth}
				\includegraphics[scale=0.42]{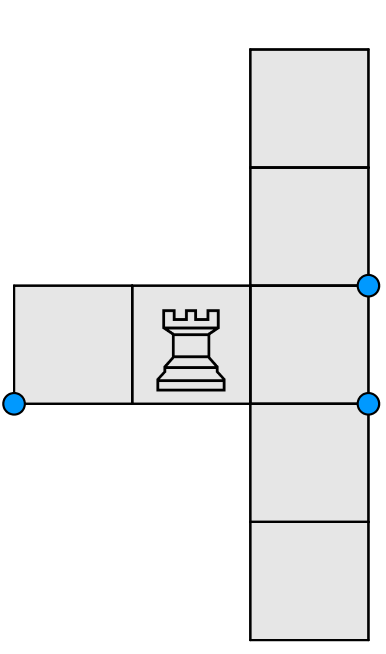}
			\end{minipage}  &  \begin{minipage}{0.23\textwidth}
			\includegraphics[scale=0.42]{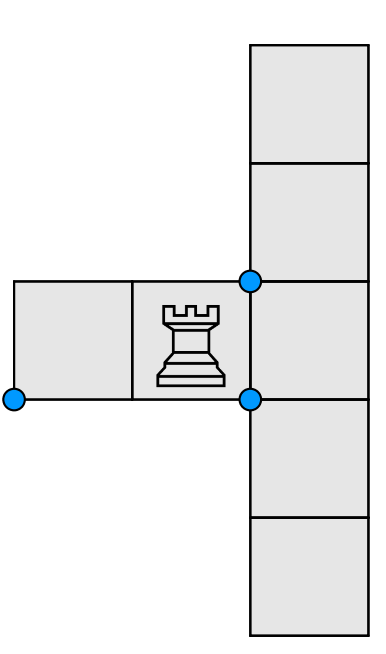}
		\end{minipage}\\
			\hline

				IX & \begin{minipage}{0.23\textwidth}
					\includegraphics[scale=0.44]{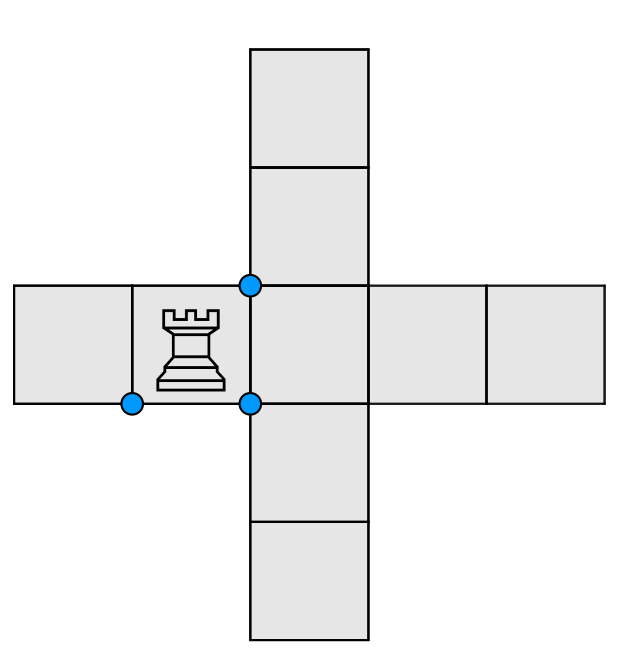}
				\end{minipage}  &  \begin{minipage}{0.23\textwidth}
					\includegraphics[scale=0.44]{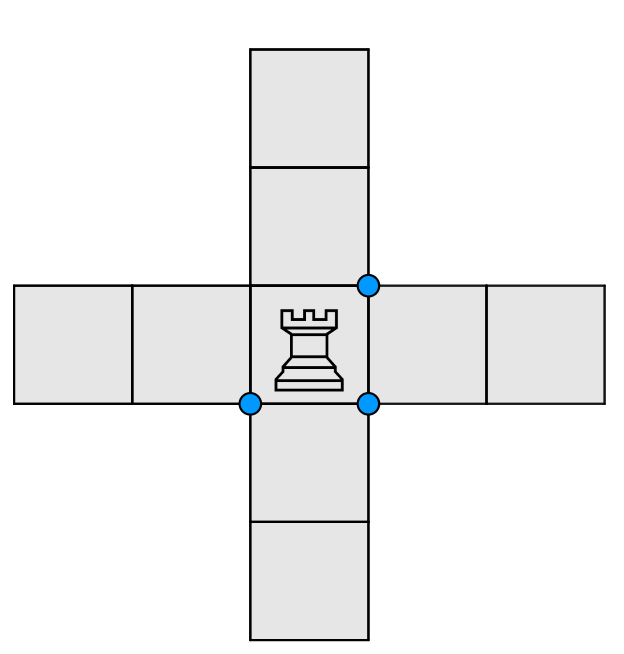}
				\end{minipage}  &  \begin{minipage}{0.23\textwidth}
					\includegraphics[scale=0.44]{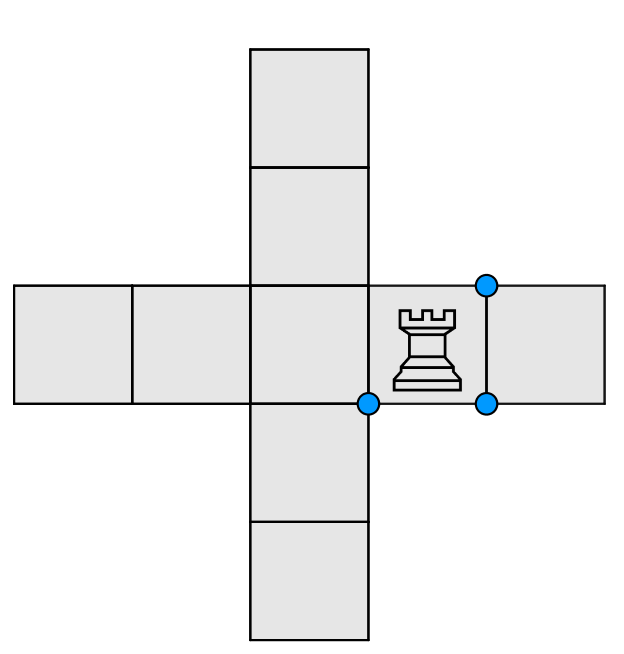}
				\end{minipage}&  \begin{minipage}{0.23\textwidth}
					\includegraphics[scale=0.44]{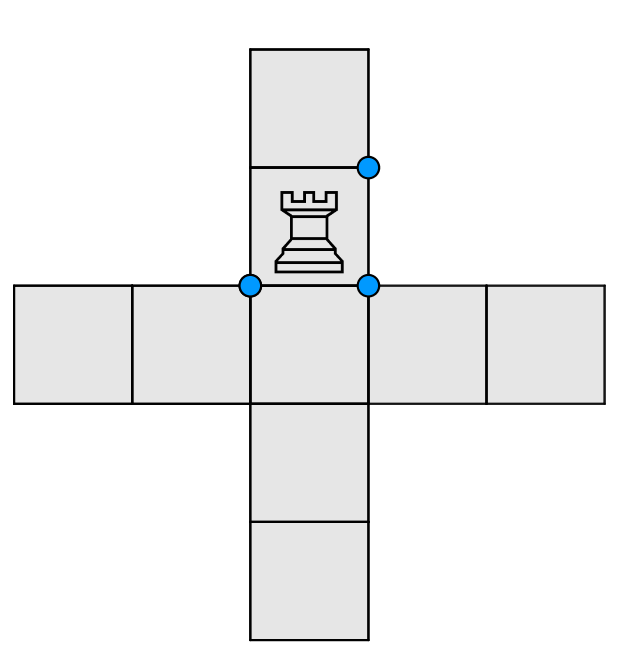}
				\end{minipage}&  \begin{minipage}{0.23\textwidth}
					\includegraphics[scale=0.44]{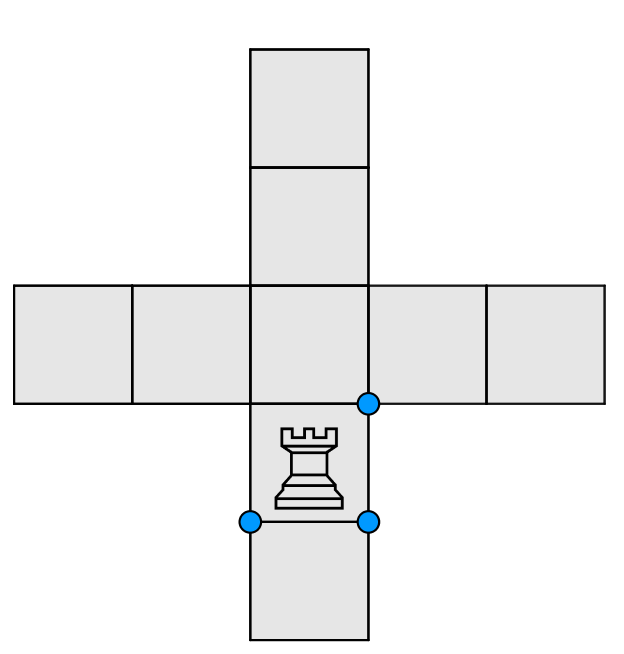}
				\end{minipage}&  \begin{minipage}{0.23\textwidth}
				\includegraphics[scale=0.44]{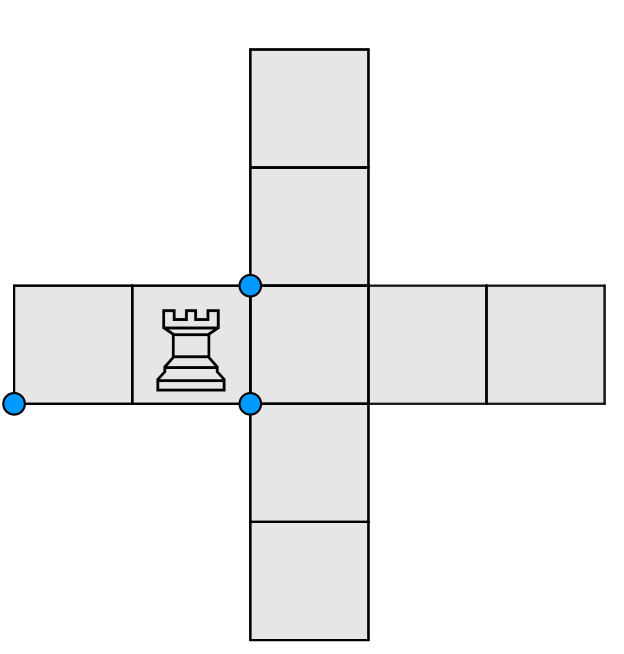}
			\end{minipage}\\
				\hline
				
				IX$'$ & \begin{minipage}{0.23\textwidth}
					\includegraphics[scale=0.44]{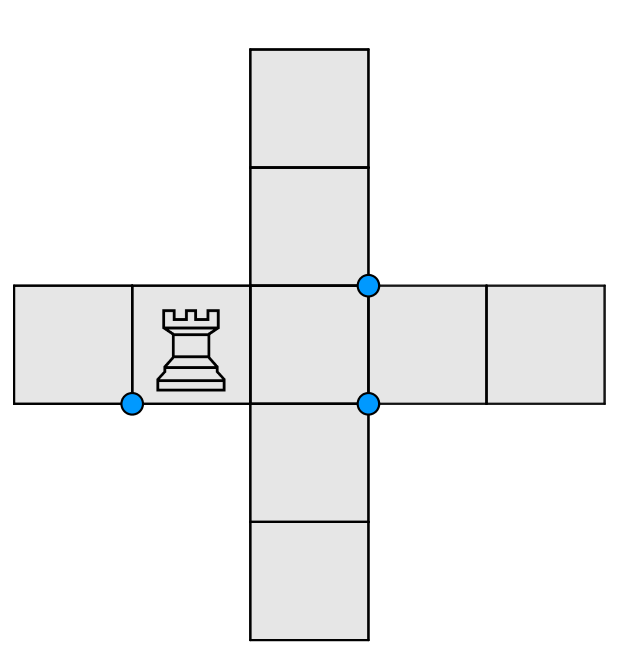}
				\end{minipage}  &  \begin{minipage}{0.23\textwidth}
					\includegraphics[scale=0.44]{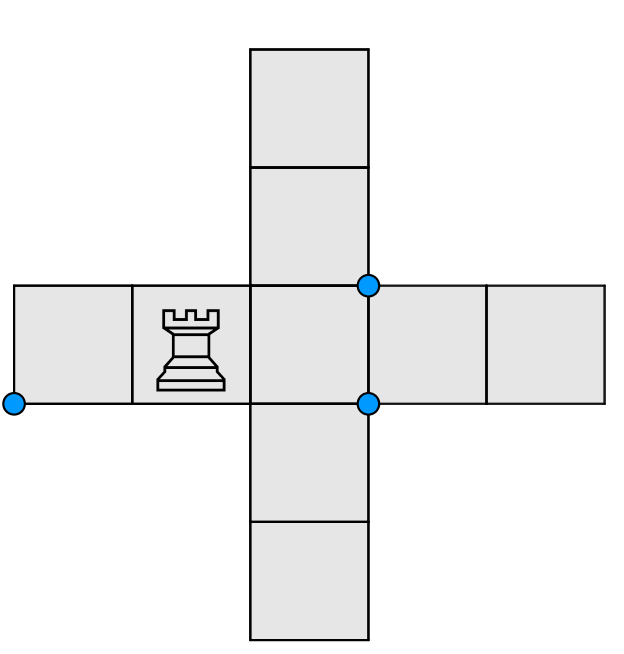}
				\end{minipage} &  \begin{minipage}{0.23\textwidth}
					\includegraphics[scale=0.44]{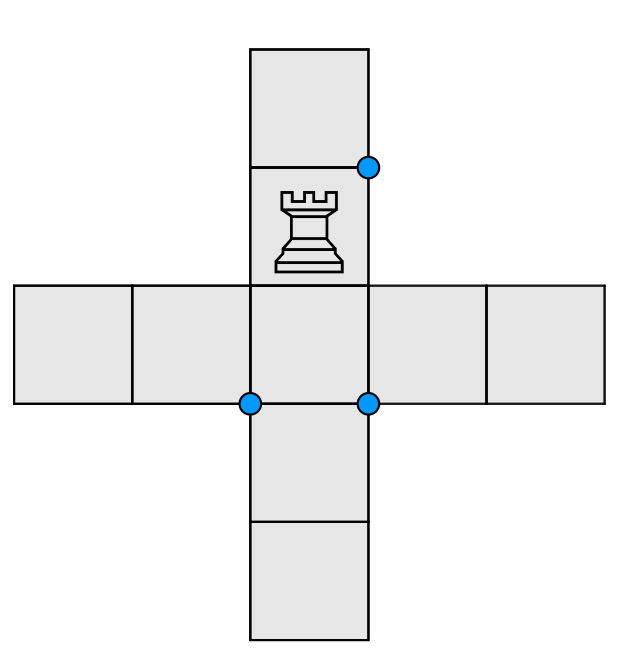}
				\end{minipage}&  \begin{minipage}{0.23\textwidth}
				\includegraphics[scale=0.44]{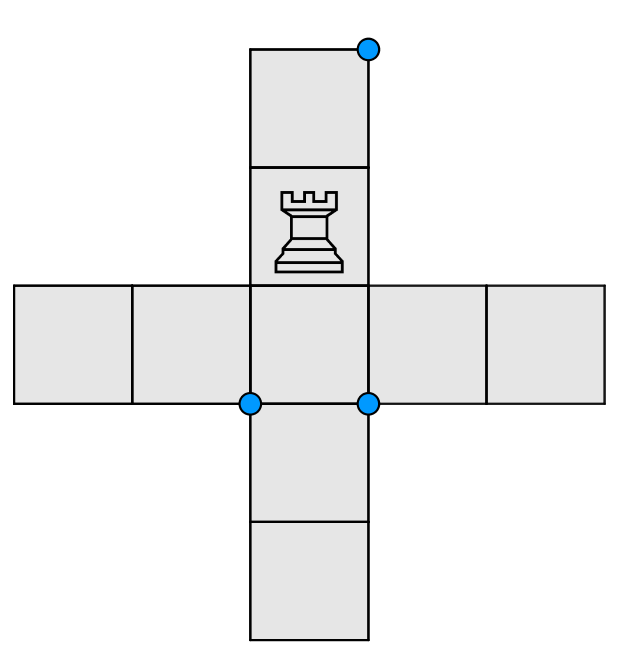}
				\end{minipage}&  \begin{minipage}{0.23\textwidth}
				\includegraphics[scale=0.44]{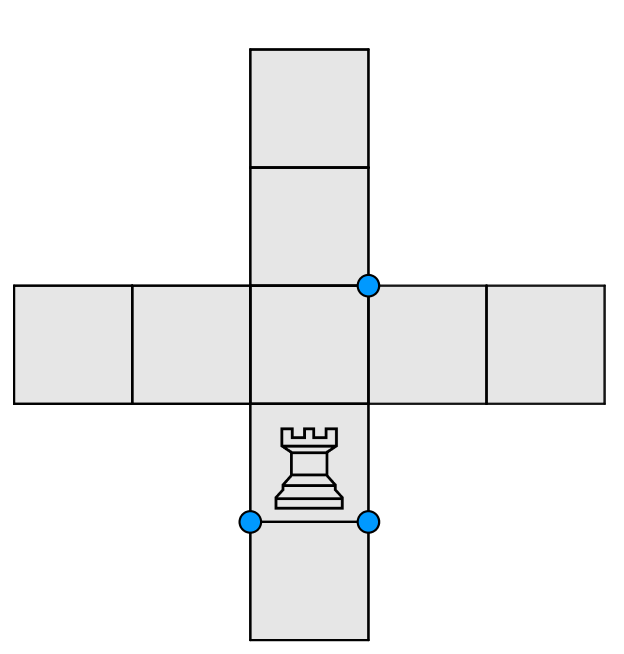}
			\end{minipage}&  \begin{minipage}{0.23\textwidth}
			\includegraphics[scale=0.44]{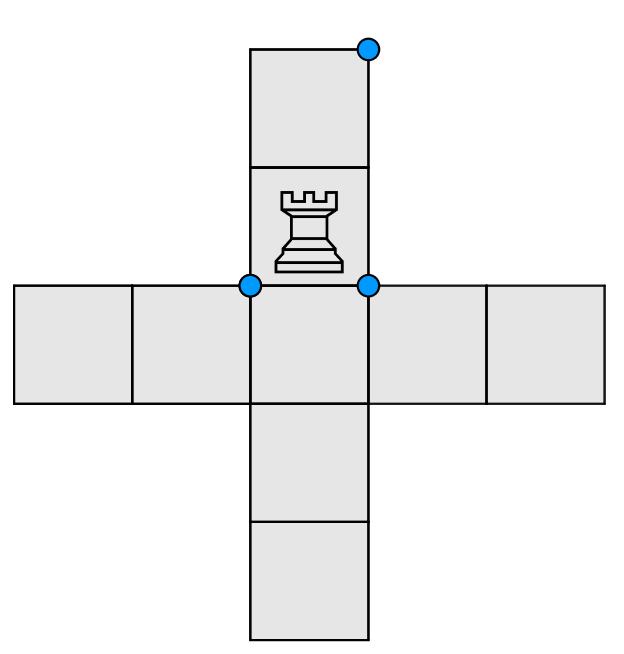}
		\end{minipage}\\
			\hline
		\end{tabular}}
		\caption{}
		\label{Table3}
	\end{table}

The aim of the following results is to demonstrate that $\cR(-)$ is bijective. We begin with the following lemma to prove its injectivity.

	\begin{prop}\label{Prop: For injective}
	Let $\cP$ be a grid polyomino and $\Delta_{\cP}$ be the attached simplicial complex. If $F_1$ and $F_2$ are two distinct facets of $\Delta_{\cP}$ then $\cR(F_1)\neq \cR(F_2)$.  
	\end{prop}

	\begin{proof}
	Let us assume that the numbers of generalized steps of $F_1$ and $F_2$ are $m$ and $n$, respectively, where $m, n \geq 1$. If $m \neq n$, then it is trivial that $\cR(F_1) \neq \cR(F_2)$. Assume now that $m = n$. Suppose there exists a generalized step $F'$ of $F_1$ which is not in $F_2$. If the rook $R_1$ related to $F'$ is not in $\cR(F_2)$, then the proof is complete. We therefore assume that $R_1 \in \cR(F_2)$. We may restrict our examination to cases (IB) and (IC) of Table \ref{Table3}, as the technique used to prove the other cases is analogous. Assume that $F' = \{(2,1), (3,1), (3,2)\} \subset F_1$, and let the generalized step of $F_2$ related to $R_1$ be $F'' = \{(1,1), (3,1), (3,2)\}$. Observe that $(1,1), (2,3) \in F_1$ by the maximality of $F_1$, and $(2,2), (1,2), (1,3) \notin F_1$ because $(2,1), (3,1) \in F_1$. Thus, $\{(1,1), (2,1), (2,3)\}$ is a generalized step of $F_1$, which implies that a rook is placed in the cell with lower-right corner $(2,2)$. On the other hand, observe that $(2,1) \notin F_2$, since $F''$ is a generalized step of $F_2$, and $(1,2), (2,2) \notin F_2$ because $(3,1) \in F_2$. Moreover, $(1,3) \in F_2$ by the maximality of $F_2$. There is no generalized step in $F_2$ that places a rook in the cell with lower-right corner $(2,2)$. Hence, $\cR(F_1) \neq \cR(F_2)$. With the preceding discussion in mind, it is evident that all other cases arising from Table \ref{Table3} can be addressed in a similar manner.\\
	Note that the statement is completely proved if we show that the case where $F_1$ and $F_2$ have the same generalized steps cannot hold. It is easy to observe the following claim: if $F_1$ and $F_2$ have the same generalized steps, then $F_1 = F_2$. In fact, suppose by contradiction that $F_1 \neq F_2$. Then, without loss of generality, assume that $F_2 = (F_1 \setminus \{f_1\}) \cup \{f_2\}$ and that $F_1 <_{\mathrm{lex}} F_2$. Using similar arguments as in the $\subseteq$ case of Theorem \ref{Thm: The lexicographic order gives a shelling order}, we can deduce that $f_1$ is the lower-right corner of a generalized step of $F_1$. Consequently, $F_1$ and $F_2$ cannot have the same generalized steps, which contradicts the assumption. Hence, we conclude that $F_1 = F_2$. \\
    In conclusion, we cannot assume that $F_1$ and $F_2$ have the same generalized steps; otherwise, we would arrive at a contradiction with the fact that $F_1$ and $F_2$ are two distinct facets of $\Delta_{\cP}$.
	\end{proof}
    
       The following proposition is is crucial as it establishes the surjectivity of $\cR(-)$. Before presenting the proof, we introduce some notations related to the structure of a grid polyomino that will be useful. For $i\in[r]$ and $j\in[s]$, let $\cH_{ij}=[a_{ij}, b_{ij}]$ be the holes of $\cP$, where $a_{ij}=((a_{ij})_1, (a_{ij})_2)$, $b_{ij}=((b_{ij})_1, (b_{ij})_2)$ with $1<(a_{ij})_1<(b_{ij})_1<m$ and $1<(a_{ij})_2<(b_{ij})_2<n$. The anti-diagonal corners of $\cH_{ij}$ are $c_{ij}:=((a_{ij})_1, (b_{ij})_2)$ and $d_{ij}:=((b_{ij})_1, (a_{ij})_2)$. Now, we denote by:
    \begin{itemize}
        \item $A_{ij}$ the cell of $\cP$ having $a_{ij}$ as upper right corner, for all $i\in[r]$ and for all $j\in[s]$;
        \item $A_{r+1,j}$ the cell of $\cP$ having $d_{rj}$ as upper left corner, for all $j\in[s]$;
        \item $A_{i,s+1}$ the cell of $\cP$ having $c_{is}$ as lower right corner, for all $i\in[r]$;
        \item $A_{r+1,s+1}$ the cell of $\cP$ having $b_{rs}$ as lower left corner.
    \end{itemize}
    Additionally, we set $V_i=[A_{i1},A_{i,s+1}]$ for all $i\in [r+1]$ and $H_j=[A_{1j},A_{r+1,j}]$ for all $j\in [s+1]$. Observe that $V_i\cap H_j=\{A_{ij}\}$ for all $i\in [r+1]$ and $j\in [s+1]$. For further clarity, refer to Figure \ref{Figure: some notations on grid polyominoes}.

     \begin{figure}[h!]
			\centering
			\includegraphics[scale=0.7]{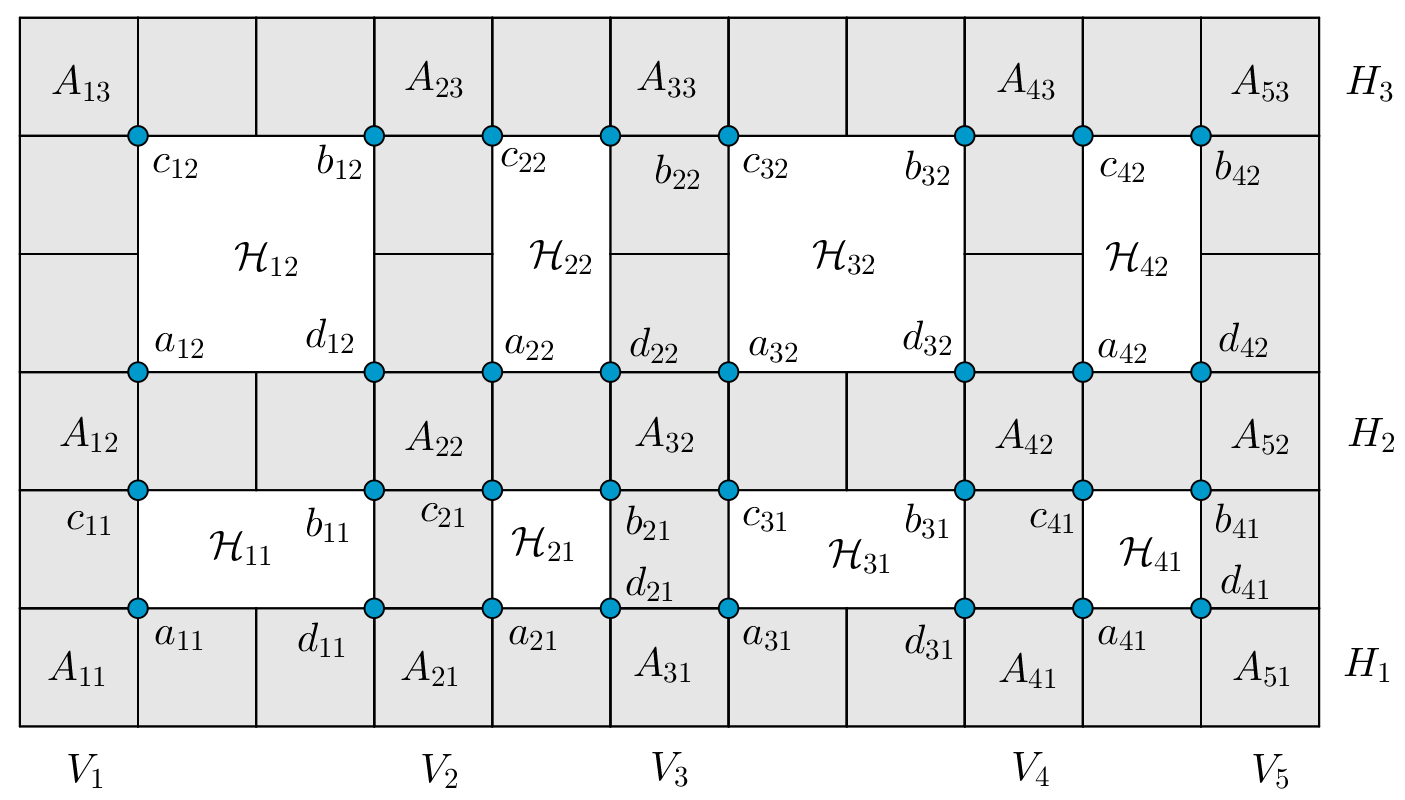}
			\caption{A grid polyomino and the related notations.}
			\label{Figure: some notations on grid polyominoes}
    \end{figure}

    \begin{prop}\label{Prop: For surjective}
    Let $\cP$ be a grid polyomino and $\Delta_{\cP}$ be the attached simplicial complex. If $\cT$ is a $k$-rook configuration in $\cP$, then there exists a facet $F$ of $\Delta_{\cP}$ with $k$ generalized steps such that $\cR(F)=\cT$.  
    \end{prop}

    \begin{proof}
    Let $\cT$ be a $k$-rook configuration in $\cP$. Starting from the bottom of the grid polyomino, we proceed step by step, examining all possible cases based on the presence of rooks of $\cT$ in the intervals $H_j$ and $V_i$. At each step, we define a suitable subset of the facet $F$ such that $\cR(F) = \cT$. The goal is to provide a procedure to inductively define $F$ from the bottom to the top of the polyomino. For clarity, we illustrate each step with suitable figures to visualize the configurations.\\
    
    (1) Assume that there is no rook in $H_1$. Then we set $F_1=\{(1,1),(1,2)\}\cup \bigcup_{k=1}^r[a_{k1}+(1,0),d_{k1}]$. Observe that no step occurs in $F_1$. We illustrate it in Figure~\ref{Figure15}.
    \begin{figure}[h]
    \includegraphics[scale=0.53]{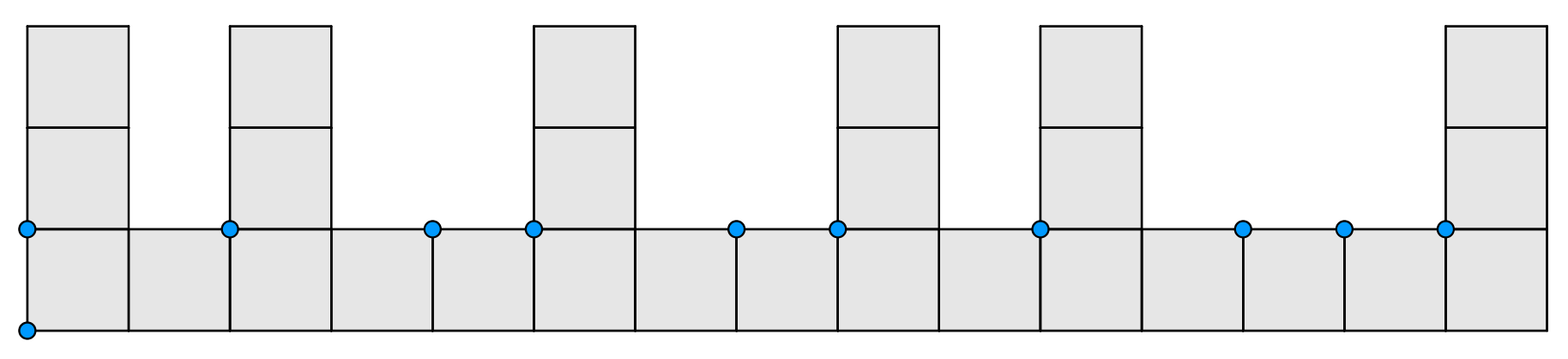}
		\caption{Example of case (1).}
		\label{Figure15}
        \end{figure} 
        
     \noindent Now, we examine two cases (A) and (B) in order to define a new set $F_2$ by adding suitable points to $F_1$. 
    \begin{enumerate}[(A)]
     \item Suppose that there is no rook in $]A_{i1},A_{i2}[$ for all $i\in [r+1]$. Then $F_2=[(1,2),c_{11}-(1,0)]\cup (\bigcup_{k=1}^r [d_{k1},b_{k1}])\cup F_1$. In this way there are no steps in $F_2$
     (see Figure~\ref{Figure16}).
     
       \begin{figure}[h!]
    \includegraphics[scale=0.53]{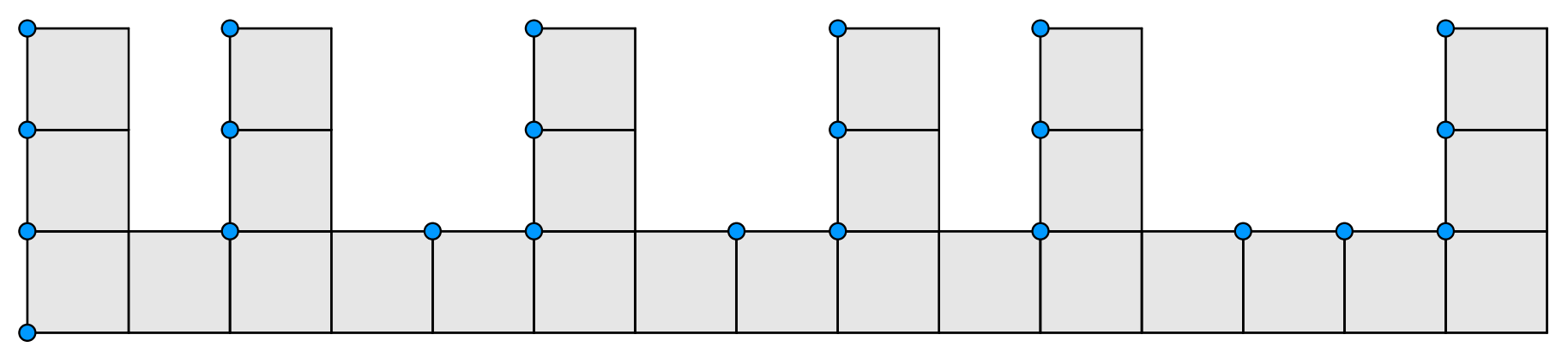}
		\caption{Example of case (1)-(A).}
		\label{Figure16}
        \end{figure} 
     
     \item Suppose that there exists $I_t = \{i_1 < i_2 < \dots < i_t\} \subseteq [r+1]$ such that there is a rook in $]A_{i_k1}, A_{i_k2}[$. We set $J_t = [r] \setminus I_t$. For all $i \in I_t$, let $a_i$ and $b_i$ denote, respectively, the lower left and lower right corners of the cell in $]A_{i1}, A_{i2}[$ where the rook is placed. We consider four sub-cases.
     \begin{enumerate}[(a)]
         \item Assume that $i_1=1$ and $i_t=r+1$. We set $F_2=([(1,2),a_{11}]\cup[a_{11},c_{11}])\cup \big(\bigcup_{i\in I_t\setminus\{1,r+1\}}([a_{i1}-(1,0),a_{i}]\cup[b_{i},c_{i1}])\big)\cup \big(\cup_{j\in J_t}[a_{j1}-(1,0),c_{j1}-(1,0)]\big) \cup ([d_{r1},d_{r1}+(1,0)]\cup[d_{r1}+(1,0),b_{r1}+(1,0)])\cup F_1$.  Note that $\{a_{i},b_{i},b_{i}+(0,1)\}$ is a step of $F_2$ for all $i\in I_t$. We refer to Figure~\ref{Figure1Ba}.
           \begin{figure}[h!]
 \includegraphics[scale=0.53]{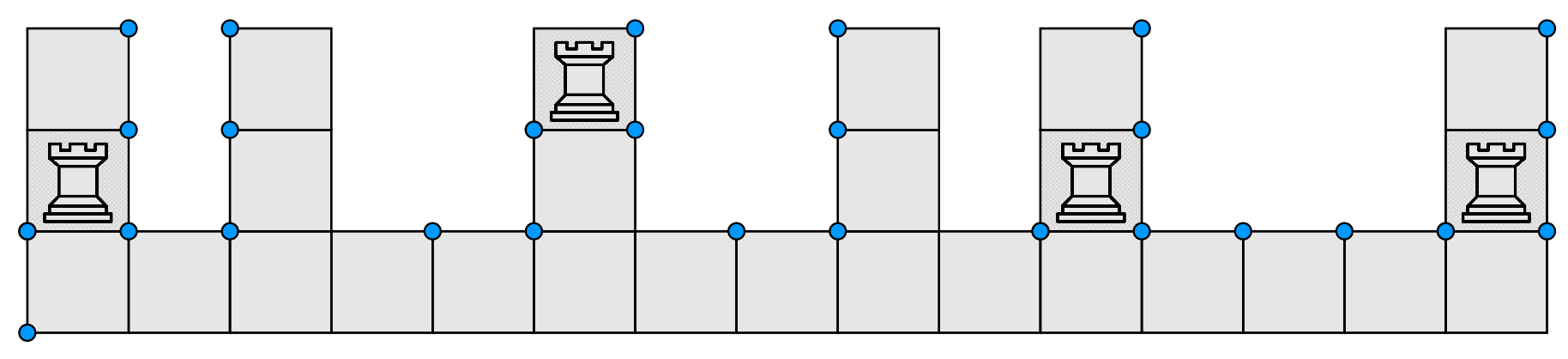}
		\caption{Example of case (1)-(B)-(a).}
		\label{Figure1Ba}
        \end{figure} 
         \item If $i_1\neq1$ and $i_t=r+1$, then
         $F_2=[(1,2),c_{11}-(1,0)]\cup \big(\bigcup_{i\in I_t\setminus\{r+1\}}([a_{i1}-(1,0),a_{i}]\cup[b_{i},c_{i1}])\big)\cup \big(\cup_{j\in J_t}[a_{j1}-(1,0),c_{j1}-(1,0)]\big) \cup ([d_{r1},a_{r+1}]\cup[b_{r+1},b_{r1}+(1,0)])\cup F_1$. 
         \item If $i_1\neq1$ and $i_t\neq r+1$, then 
         $F_2=[(1,2),c_{11}-(1,0)]\cup \big(\bigcup_{i\in I_t}([a_{i1}-(1,0),a_{i}]\cup[b_{i},c_{i1}])\big)\cup \big(\cup_{j\in J_t}[a_{j1}-(1,0),c_{j1}-(1,0)]\big) \cup [d_{r1},b_{r1}]\cup F_1$. 
         \item If $i_1=1$ and $i_t\neq r+1$, then $F_2=([(1,2),a_{11}]\cup[b_{11},c_{11}])\cup \big(\bigcup_{i\in I_t\setminus\{1\}}([a_{i1}-(1,0),a_{i}]\cup[b_{i},c_{i1}])\big)\cup \big(\cup_{j\in J_t}[a_{j1}-(1,0),c_{j1}-(1,0)]\big) \cup [d_{r1},b_{r1}]\cup F_1$.
     \end{enumerate}
    \end{enumerate}

    (2) Assume that a rook $R_1$ is located in $H_1$. Let $c_i$ and $d_i$ represent the lower right and upper right corners, respectively, of the cell in $H_1$ where the rook is placed. We define $F_1 = [(1,1), c_i] \cup [d_i, (m,2)]$. Next, we examine the cases where a rook is either present or absent in certain vertical cell intervals, aiming to define a new set $F_2$ based on the previously defined $F_1$.
    
    \begin{enumerate}[(A)]
        \item  Suppose that there is no rook in $]A_{i1},A_{i2}[$ for all $i\in [r+1]$. We consider the following four sub-cases.
        \begin{enumerate}[(a)]
        
            \item if $R_1$ is in $A_{i1}$ for some $i=2,\dots,r$ then we set $F_2=[(1,3),c_{11}-(1,0)]\cup \big( \bigcup_{k=1}^{i-1} [d_{k1}+(0,1),b_{k1}]\big)\cup [a_{i1},c_{i1}]\cup\big( \bigcup_{k=i}^{r} [d_{k1},b_{k1}]\big)\cup \big(F_1\setminus (\{a_{k1}-(0,1):k=1,\dots,i-1\}\cup \{a_{k1}:k=i+1,\dots,r\}\cup \{d_{r1}+(1,0)\})\big)$, see Figure~\ref{Figure2Aa}.

            \begin{figure}[h!]
		\centering	\includegraphics[scale=0.53]{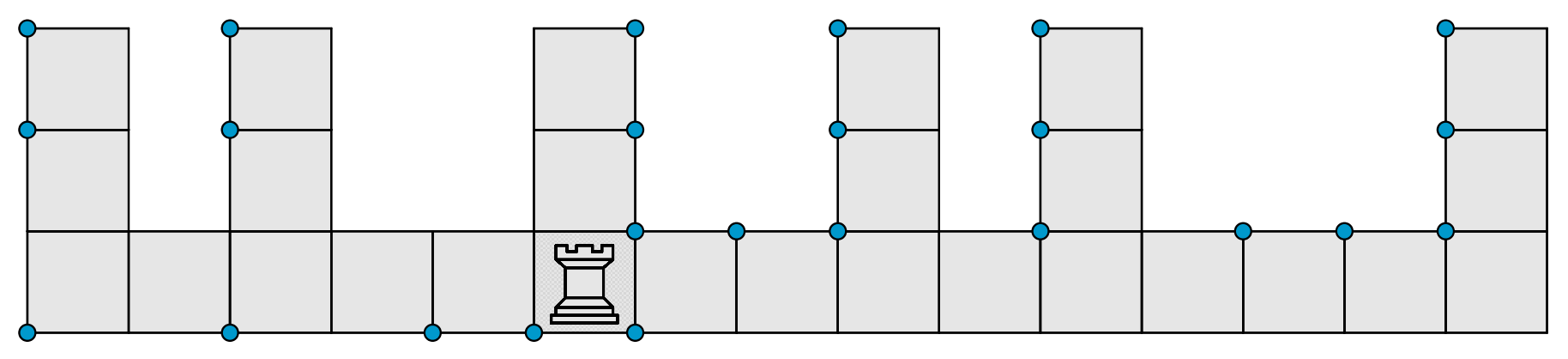}
		\caption{Example of case (2)-(A)-(a).}
		\label{Figure2Aa}
        \end{figure} 
        
             \item if the rook $R_1$ is in $A_{11}$, then we set $F_2=[a_{11}+(0,1),c_{11}]\cup (\bigcup_{k=1}^r [d_{k1},b_{k1}])\cup\big(F_1\setminus(\{a_{11}-(0,1),d_{r1}+(1,0)\}\cup \{a_{k1}:k=2,\dots,r\})\big)$. 
            \item if the rook $R_1$ is in $A_{r+1,1}$, then $F_2=[(1,3),c_{11}-(1,0)]\cup (\bigcup_{k=1}^{r-1} [d_{k1}+(0,1),b_{k1}])\cup [d_{r1}+(1,0),b_{r1}+(1,0)]\cup \big(F_1\setminus(\{a_{k1}:k=1,\dots,r\})\big)$.
            \item if $R_1$ is in $]A_i,A_{i+1}[$ for some $i\in [r]$ then we set $F_2=[(1,3),c_{11}-(1,0)]\cup (\bigcup_{k=1}^r [d_{k1},b_{k1}])\cup \big(F_1\setminus\ (\{a_{k1}-(0,1):k=1,\dots,i\} \cup \{a_{k1}:k=i+1,\dots,r \}\cup \{d_{r1}+(1,0)\})\big)$. We illustrate this configuration in Figure~\ref{Figure2Ad}.
            
        \begin{figure}[h!]
	\centering	\includegraphics[scale=0.53]{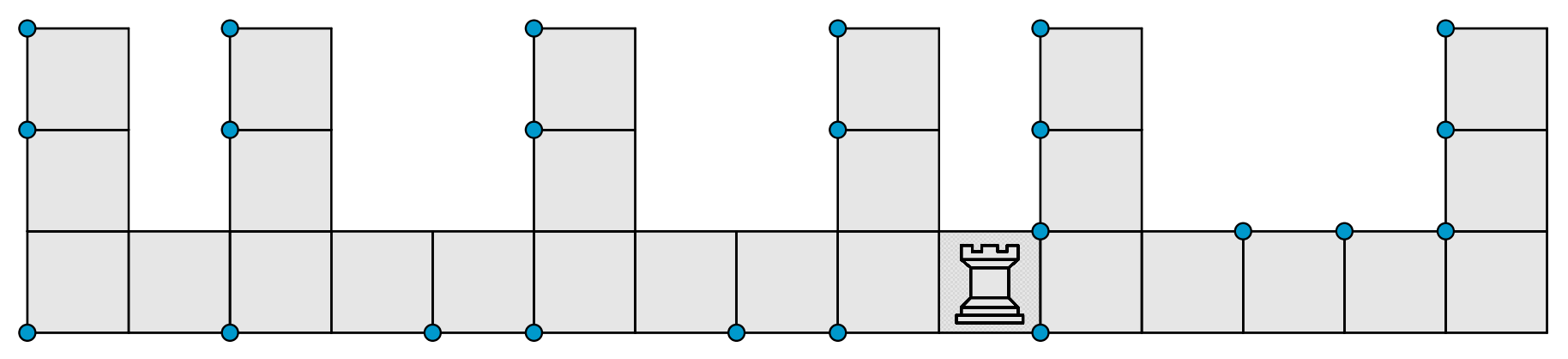}
		\caption{Example of case (2)-(A)-(d).}
		\label{Figure2Ad}
        \end{figure} 
        
        \end{enumerate}
        
        \item Suppose that there exists $I_t=\{i_1<i_2<\dots<i_t\}\subseteq [r+1]$ such that a rook is positioned within $]A_{i_k1}, A_{i_k2}[$. We define $J_t = [r] \setminus I_t$. For each $i \in I_t$, let $a_i$ and $b_i$ represent the lower left and lower right corners, respectively, of the cell in $]A_{i1}, A_{i2}[$ where the rook is located. Additionally, we set $K_t=\big\{k\in I_t:\text{a rook is placed in the cell at North of $A_{k1}$}\big\}$. The following sub-cases are then considered.
        \begin{enumerate}[(a)]
            \item  Assume that $R_1$ is in $A_{i1}$ for some $i=2,\dots,r$. 
            \begin{enumerate}[(i)]
                \item If $i_1=1$, $i_t=r+1$ and $i_1,i_t\in K_t$ (see Figure \ref{Figure 2Bi}), then
            $F_2=[(2,3),c_{11}]\cup\big(\bigcup_{k\in K_t\setminus\{1,r+1\}} [a_{k1}+(0,1),c_{k1}]\big)\cup\big(\bigcup_{k\in I_t\setminus K_t} ([a_{k1}+(-1 ,1),a_{k}]\cup [b_k,c_{k1}])\big)\cup \big(\cup_{j\in J_t}[a_{j1}+(-1,1),c_{j1}-(1,0)]\big)\cup [a_{i1},c_{i1}] \cup [d_{r1}+(1,0),b_{r1}+(1,0)]\cup \Big(F_1\setminus\big(\{a_{k1}-(0,1):k\in I_t\setminus K_t, k<i\}\cup\{a_{k1}:k\in I_t\setminus K_t, k>i\}\cup \{a_{j1}-(0,1):j\in J_t,j<i\}\cup\{a_{j1}:j\in J_t,j>i\}  \big)\Big)$.
            
        \begin{figure}[h!]
		\centering	\includegraphics[scale=0.53]{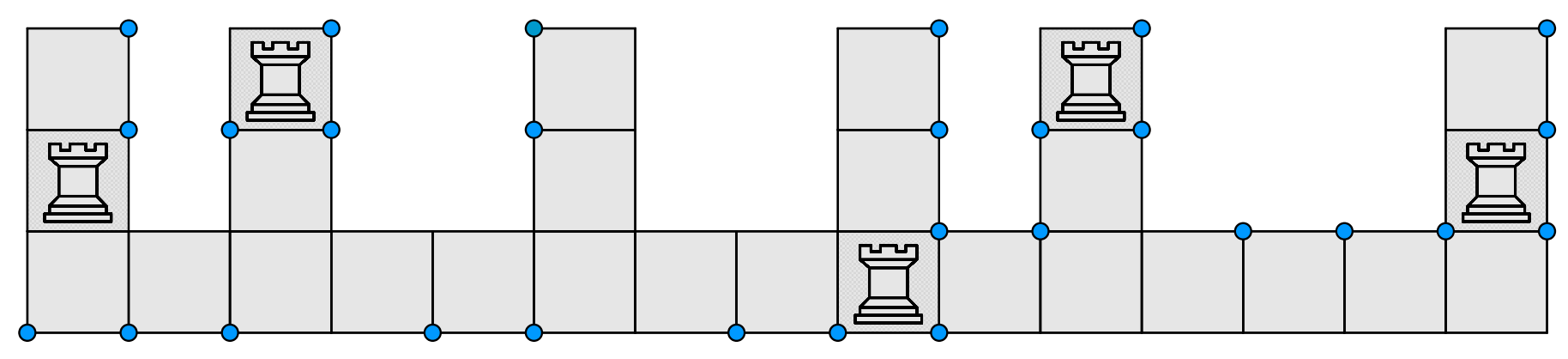}
		\caption{Example of case (2)-(B)-(a)-(i).}
		\label{Figure 2Bi}
        \end{figure} 
        
        \item If $i_1\neq 1$, $i_t=r+1$ and $i_t\in K_t$, then $F_2$ is as presented in Figure~\ref{Figure 2Ba ii}.

        \begin{figure}[h!]
		\centering	\includegraphics[scale=0.53]{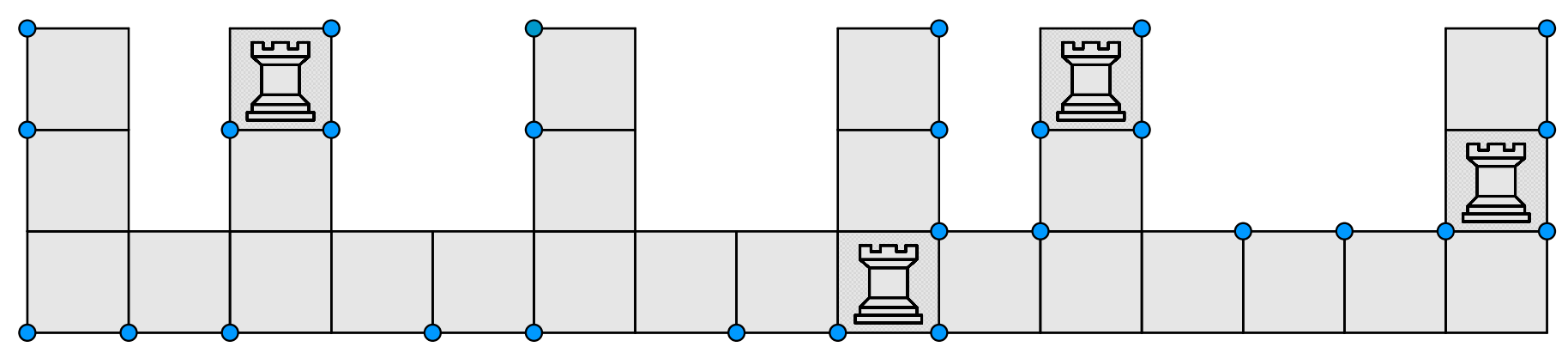}
		\caption{Example of case (2)-(B)-(a)-(ii).}
		\label{Figure 2Ba ii}
        \end{figure} 

        \item If $i_1\neq 1$, $i_t=r+1$ and $i_t\in I_t\setminus K_t$, then $F_2$ is as in Figure~\ref{Figure 2Ba iii}.

        \begin{figure}[h!]
		\centering	\includegraphics[scale=0.53]{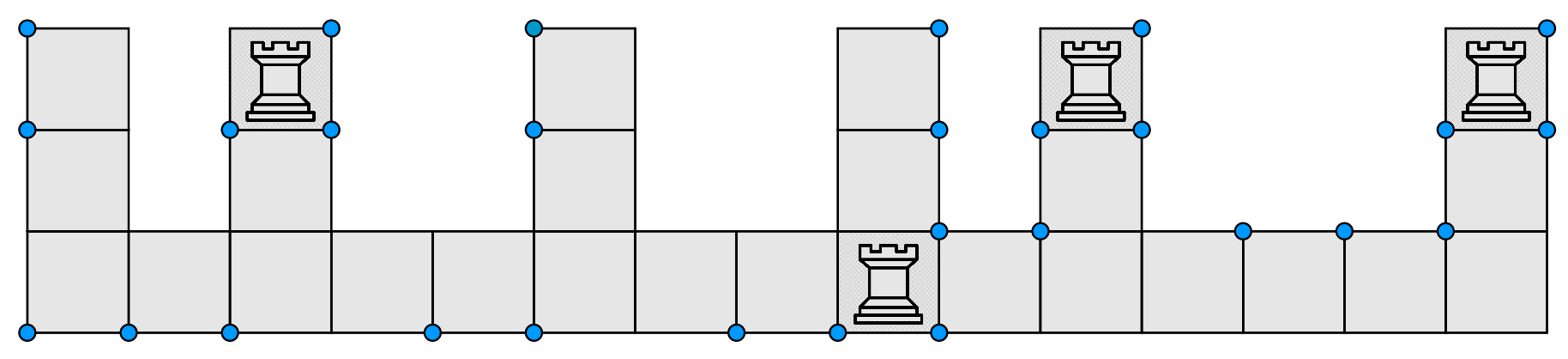}
		\caption{Example of case (2)-(B)-(a)-(iii).}
		\label{Figure 2Ba iii}
        \end{figure} 

        \item If $i_1\neq 1$, $i_t\neq r+1$, then $F_2$ is as in Figure~\ref{Figure 2Ba iv}.

        \begin{figure}[h!]
		\centering	\includegraphics[scale=0.53]{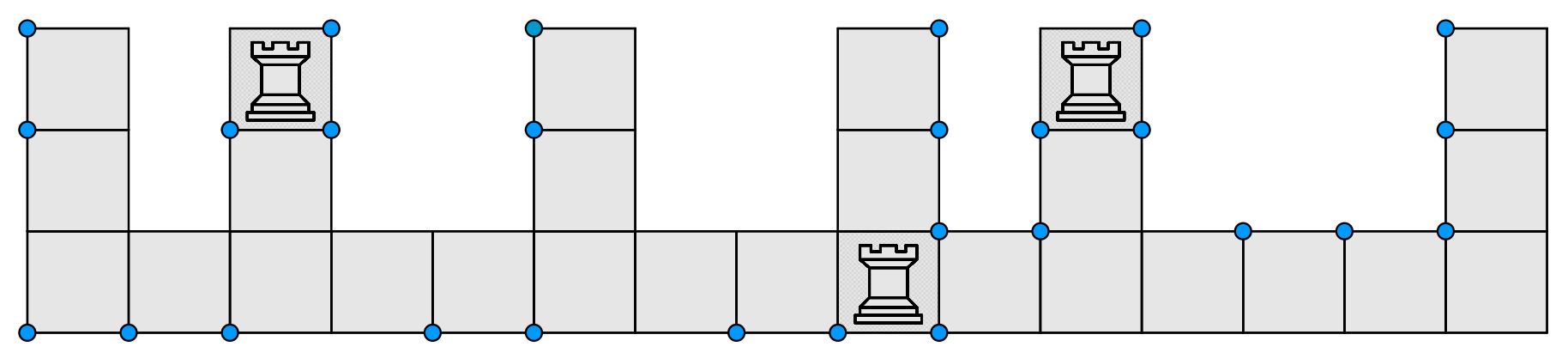}
		\caption{Example of case (2)-(B)-(a)-(iv).}
		\label{Figure 2Ba iv}
        \end{figure} 
        
            \end{enumerate}
        In the other remaining cases, $F_2$ can be defined in a similar manner.
        
   \item Assume that $R_1$ is in $]A_i,A_{i+1}[$ for some $i\in [r]$.  
   \begin{enumerate}[(i)]
       \item If $i_1=1$, $i_t=r+1$ and $i_1,i_t\in K_t$ (see Figure \ref{Figure 2Bb i}), then
            $F_2=[(2,3),c_{11}]\cup\big(\bigcup_{k\in K_t\setminus\{1,r+1\}} [a_{k1}+(0,1),c_{k1}]\big)\cup\big(\bigcup_{k\in I_t\setminus K_t} ([a_{k1}+(-1,1),a_{k}]\cup [b_k,c_{k1}])\big)\cup \big(\cup_{j\in J_t}[a_{j1}+(-1,1),c_{j1}-(1,0)]\big) \cup [d_{r1}+(1,0),b_{r1}+(1,0)]\cup \Big(F_1\setminus\big(\{a_{k1}-(0,1):k\in I_t\setminus K_t, k<i+1\}\cup\{a_{k1}:k\in I_t\setminus K_t, k>i\}\cup \{a_{j1}-(0,1):j\in J_t,j<i+1\}\cup\{a_{j1}:j\in J_t,j>i\}  \big)\Big)$.
            \begin{figure}[h!]
	\centering	\includegraphics[scale=0.53]{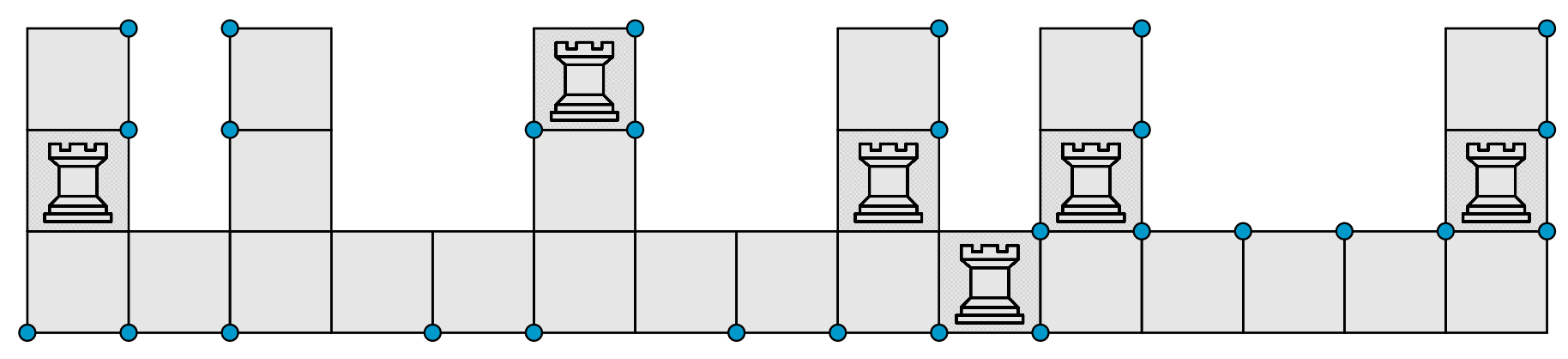}
		\caption{Example of case (2)-(B)-(b)-(i).}
		\label{Figure 2Bb i}
        \end{figure} 
        \item $F_2$ can be defined similarly in all other cases.     
   \end{enumerate}

    \item Assume that $R_1$ is in $A_{11}$.
    
    \begin{enumerate}[(i)]
        \item If $i_t=r+1$ and $i_t\in K_t$ (see Figure \ref{Figure 2Bc i}), then we set $F_2=[(2,3),c_{11}]\cup\big(\bigcup_{k\in K_t\setminus\{r+1\}} [a_{k1}+(0,1),c_{k1}]\big)\cup\big(\bigcup_{k\in I_t\setminus K_t} ([a_{k1}+(-1,1),a_{k}]\cup [b_k,c_{k1}])\big)\cup \big(\cup_{j\in J_t}[a_{j1}+(-1,1),c_{j1}-(1,0)]\big) \cup [d_{r1}+(1,0),b_{r1}+(1,0)]\cup \big(F_1\setminus(\{a_{k1}:k\in (I_t\setminus K_t)\cup J_t\})\big)$.
           \begin{figure}[h!]
	\centering	\includegraphics[scale=0.53]{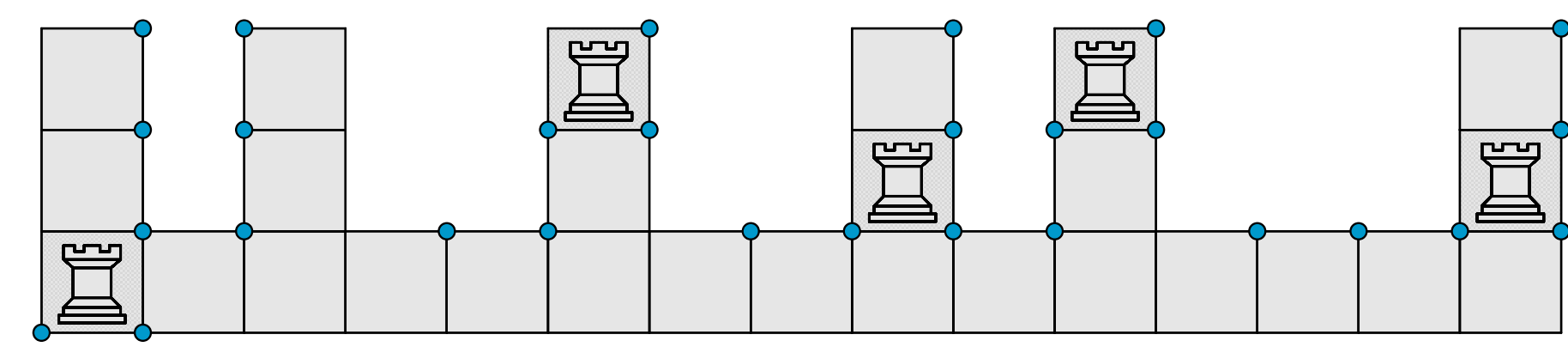}
		\caption{Example of case (2)-(B)-(c)-(i).}
		\label{Figure 2Bc i}
        \end{figure} 
        \item  In the other cases $F_2$ can be defined in a similar way.
    \end{enumerate}

    \item Assume that $R_1$ is in $A_{r+1,1}$.
    
    \begin{enumerate}[(i)]
        \item If $i_1=1$ and $i_1\in K_t$ (see Figure \ref{Figure 2Bd i}), then we set 
        $F_2=[(2,3),c_{11}]\cup\big(\bigcup_{k\in K_t\setminus\{1,r+1\}} [a_{k1}+(0,1),c_{k1}]\big)\cup\big(\bigcup_{k\in I_t\setminus K_t} ([a_{k1}+(-1,1),a_{k}]\cup [b_k,c_{k1}])\big)\cup \big(\cup_{j\in J_t}[a_{j1}+(-1,1),c_{j1}-(1,0)]\big)\cup [d_{r1}+(1,0),b_{r1}+(1,0)]\cup \Big(F_1\setminus\big(\{a_{k1}-(0,1):k\in (I_t\setminus K_t)\cup J_t\}\big)\Big)$.
        
           \begin{figure}[h!]
	\centering	\includegraphics[scale=0.53]{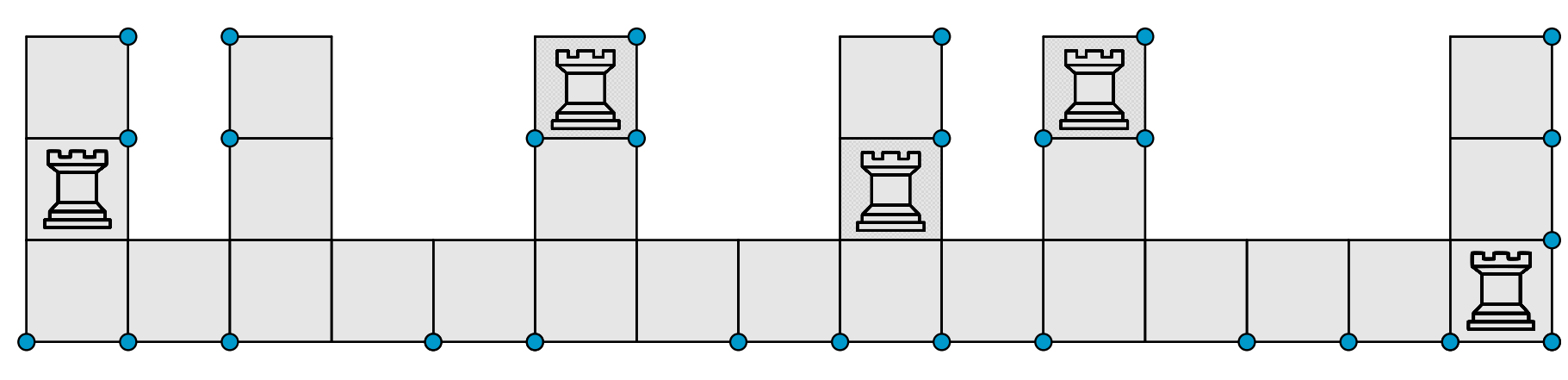}
		\caption{Example of case (2)-(B)-(c)-(i).}
		\label{Figure 2Bd i}
        \end{figure} 
        \item  In the other cases $F_2$ can be defined in a similar way.
    \end{enumerate}
    \end{enumerate}
    \end{enumerate}

     Now, consider $H_2$. To define a new suitable set $F_3$ starting from $F_2$, we distinguish and analyze the cases when a rook is placed in a cell of $H_2$ or not, based on the cases we have already examined. It is not restrictive to assume that $s > 2$, because if $s = 1$, the setting of $F_3$ is similar. Therefore, $\cP$ contains the vertical blocks $[A_{i2},A_{i3}]$ for $i \in [r+1]$.

       \noindent  Case 1. Suppose that there is no rook placed in $H_2$.
        
        \begin{enumerate}
        
            \item Assume that we are in the case (1)-(A). Then $F_3=\{a_{12}-(1,0)\}\cup (\bigcup_{k=1}^r[a_{k2}+(1,0),d_{k2}])\cup F_2$ (see Figure \ref{Figure Case 1 (1)}). We note that, if $s=1$, then $F_3=\{a_{12}-(1,0),a_{12},d_{r2}+(1,0)\}\cup (\bigcup_{k=1}^r[a_{k2}+(1,0),d_{k2}])\cup F_2$.
            
            \begin{figure}[h!]
	\centering	
        \includegraphics[scale=0.53]{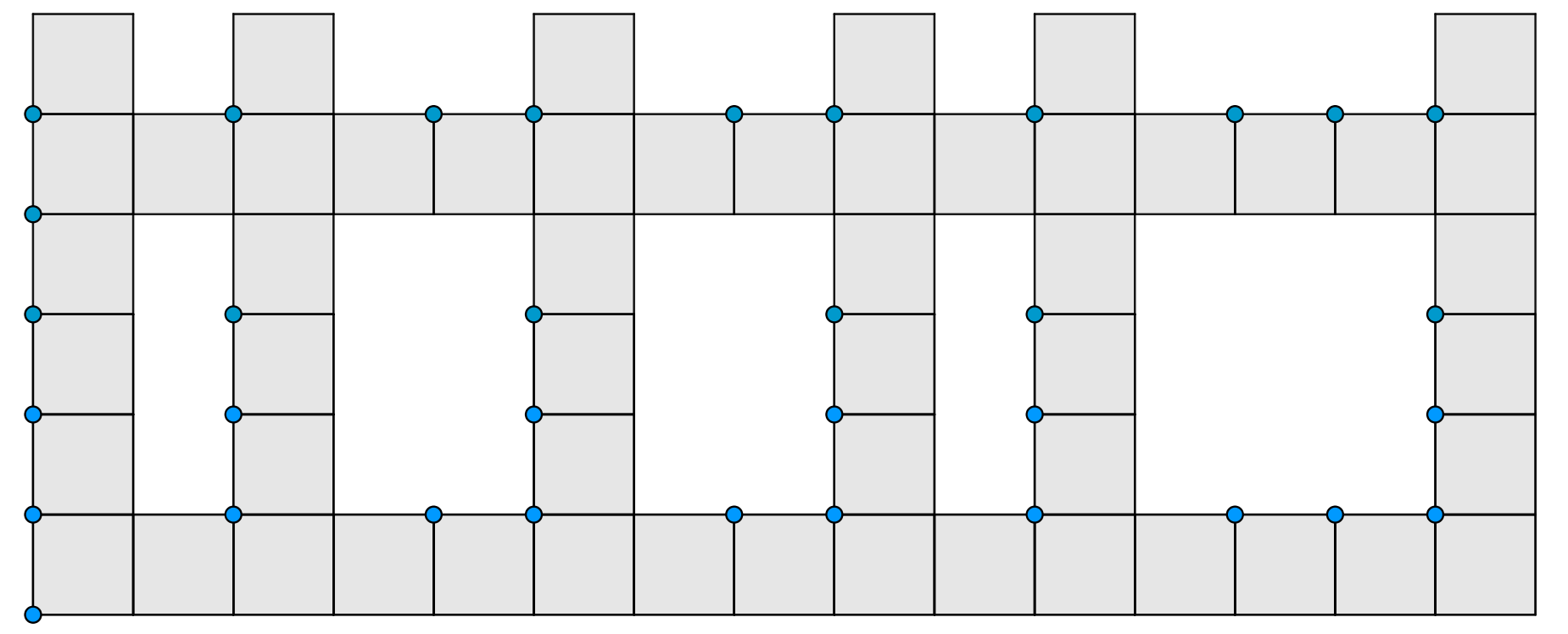}
	\caption{Example of Case 1-(1).}
	\label{Figure Case 1 (1)}
        \end{figure} 

        \begin{enumerate}[(A)]
        
            \item Suppose that there is no rook in $]A_{i2},A_{i3}[$ for all $i\in [r+1].$
            Then $F_4=[a_{12}-(1,0),c_{12}-(1,0)]\cup (\bigcup_{k=1}^r[d_{k2},b_{k2}])\cup F_3$ (see Figure \ref{Figure Case 1 (1) A}).
        
            \begin{figure}[h!]
	\centering	
        \includegraphics[scale=0.53]{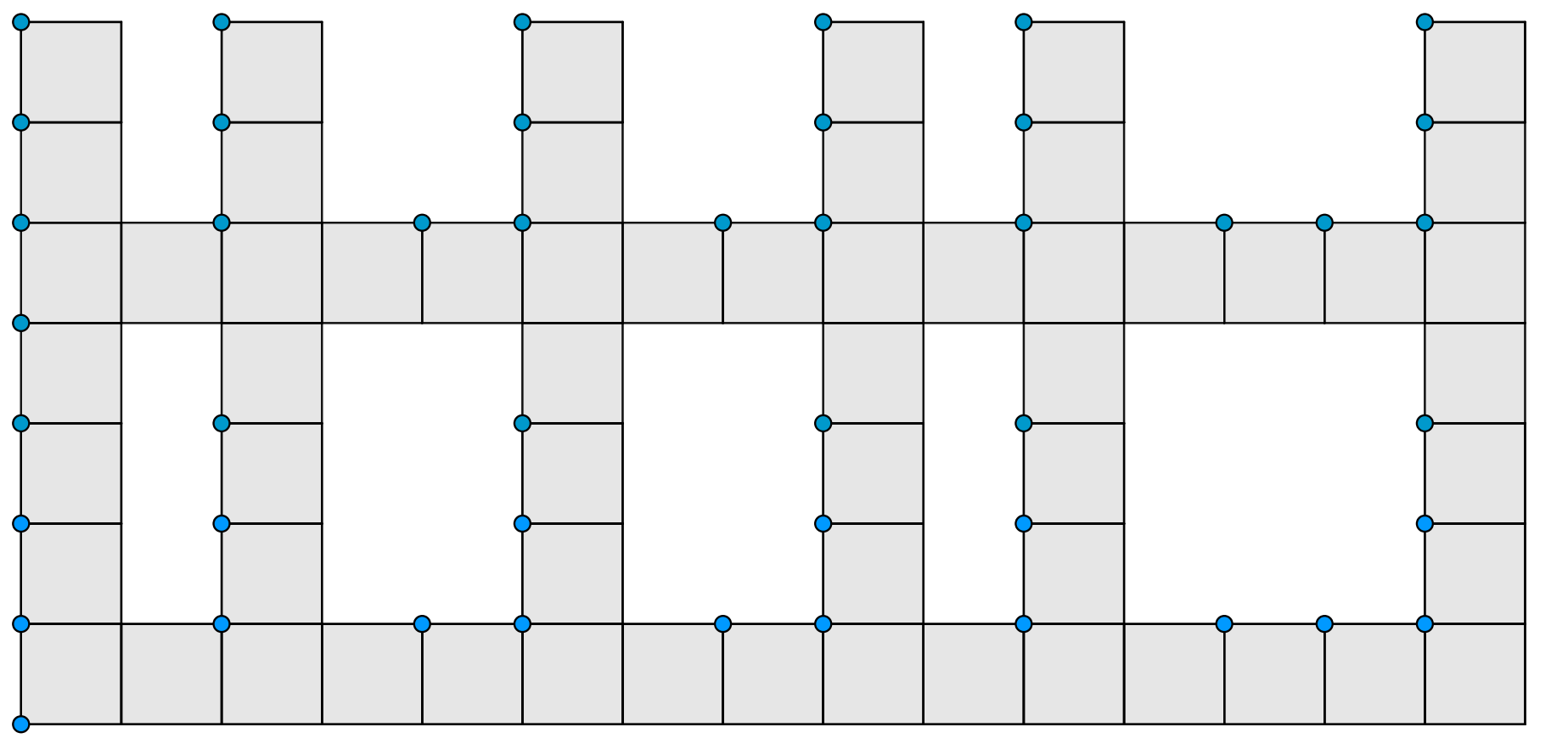}
	\caption{Example of Case 1-(1)-(A).}
	\label{Figure Case 1 (1) A}
        \end{figure} 
        
        \item Suppose that there exists a subset $L_p = \{h_1 < \dots < h_p\} \subseteq J_t$ such that there is a rook in the interval $]A_{h_j2}, A_{h_j3}]$ for some $j \in [p]$. We set $M_p = [r] \setminus L_p$. For all $i \in L_p$, we denote by $a_i$ and $b_i$ the lower left and lower right corners of the cell in the interval $]A_{i2}, A_{i3}]$, where the rook is placed. We may assume that $h_1 = 1$ and $h_p = r + 1$, because in all the other sub-cases, $F_4$ can be defined by following cases (b), (c), and (d) from (1)-(B). Then $F_4=([a_{12}-(1,0),a_{12}]\cup[a_{12},c_{12}])\cup \big(\bigcup_{i\in L_p\setminus\{1,r+1\}}([d_{i2},a_{i}]\cup[b_{i},c_{i2}])\big)\cup \big(\bigcup_{j\in M_p}[a_{j2}-(1,0),c_{j2}-(1,0)]\big) \cup ([d_{r2},d_{r2}+(1,0)]\cup[d_{r2}+(1,0),b_{r1}+(1,0)])\cup F_3$ (see Figure \ref{Figure Case 1 (1) B}).

        \begin{figure}[h!]
	\centering	
        \includegraphics[scale=0.53]{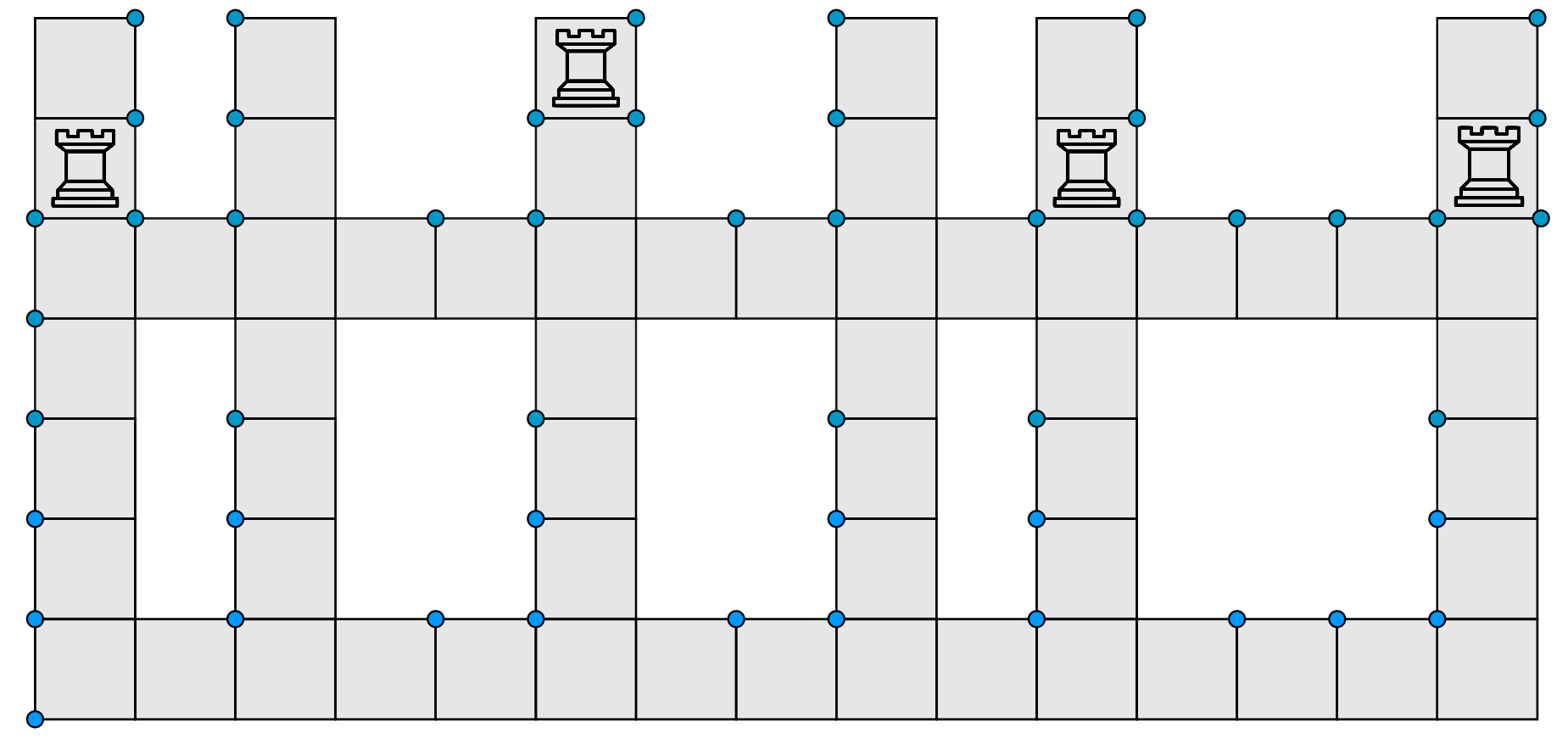}
	\caption{Example of Case 1-(1)-(B).}
	\label{Figure Case 1 (1) B}
        \end{figure} 
        
        \end{enumerate}

         \item Assume we are in case (1)-(B). Specifically, we can assume we are in sub-case (a), as in the other sub-cases (b), (c), and (d), $F_3$ can be defined in a similar manner. Then $F_3=\big([a_{12},d_{r2}+(1,0)]\setminus(\{a_{k2}:k\in J_t \}\cup \{a_{k2}-(1,0):k\in I_t\})\cup \big(F_2\setminus(\{c_{k1}:k\in I_t\}\cup \{a_{k1}-(1,0):k\in J_t\}\cup \{d_{r1}+(1,0)\})\big)$. (see Figure \ref{Figure Case 1 (2)}).
           
            \begin{figure}[h!]
	\centering	
    \includegraphics[scale=0.53]{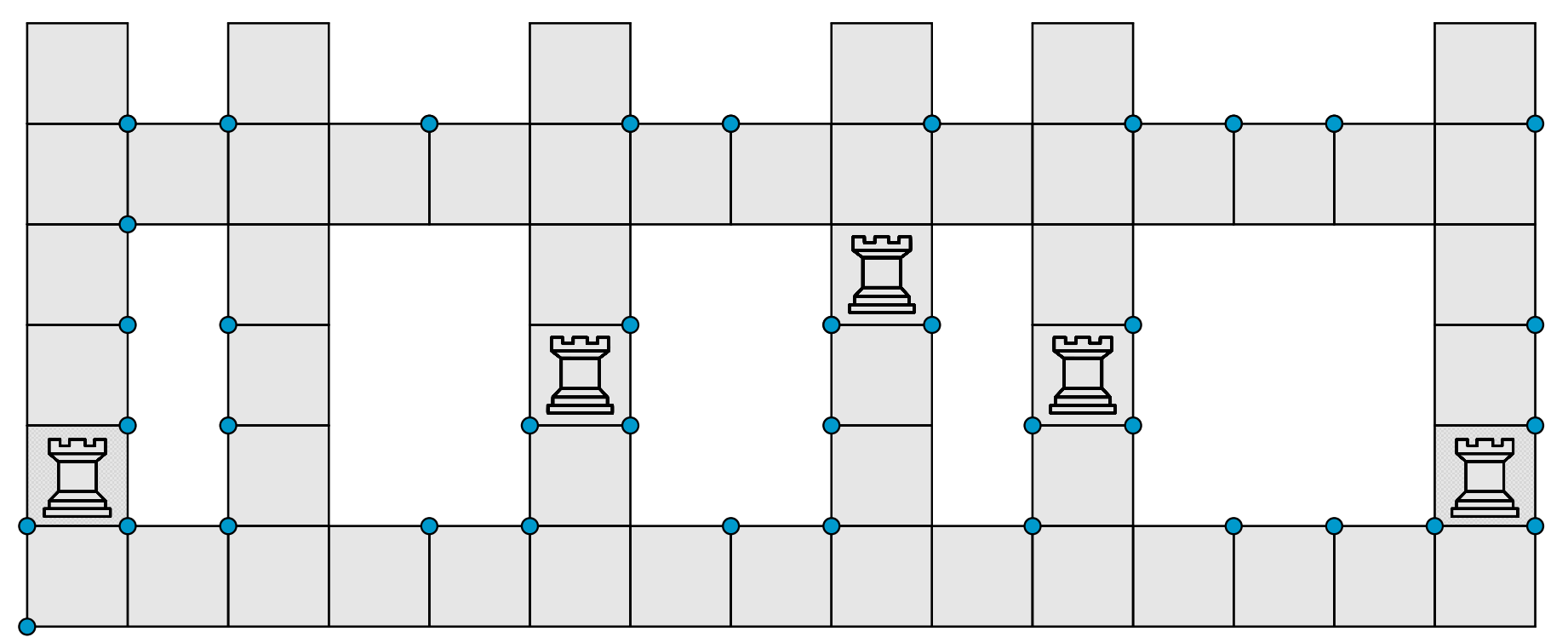}
	\caption{Example of Case 1-(2).}
	\label{Figure Case 1 (2)}
        \end{figure} 

        \begin{enumerate}[(A)]
        
            \item Suppose that there is no rook in $]A_{i2},A_{i3}[$ for all $i\in [r+1].$
            Then $F_4=[a_{12},c_{12}]\cup (\bigcup_{k\in I_t\setminus\{1,r+1\}}[a_{k2},c_{k2}])\cup (\bigcup_{k\in J_t}[a_{k2}-(1,0),c_{k2}-(1,0)]) \cup [d_{r2}+(1,0),b_{r2}+(1,0)] \cup F_3$ (see Figure \ref{Figure Case 1 (2) A}).
        
            \begin{figure}[h!]
	\centering	
        \includegraphics[scale=0.53]{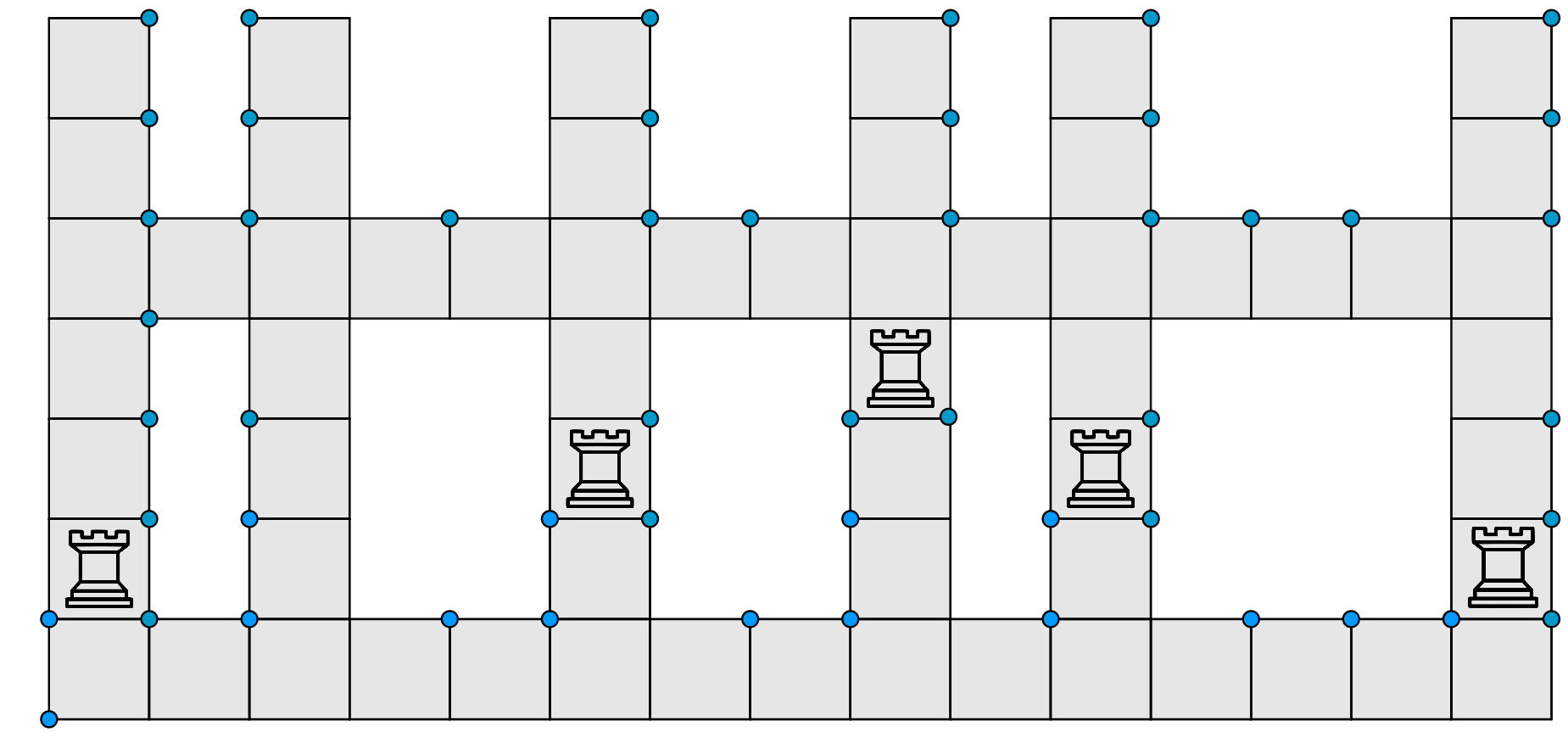}
	\caption{Example of Case 1-(2)-(A).}
	\label{Figure Case 1 (2) A}
        \end{figure} 

         \item Suppose that there exists $L_p=\{h_1<\dots<h_p\}\subseteq J_t$ such that there is a rook in $]A_{h_j2},A_{h_j3}[$. Since $i_1=1$ and $i_t=r+1$, we have that $h_1\neq 1$ and $h_p\neq r+1$. We set $N_p=\{h\in L_p:\text{a rook is placed in the cell at North of } A_{h2}\}$. For all $i \in L_p \setminus N_p$, $a_i$ and $b_i$ denote the lower left and lower right corners of the cell $]A_{i2}, A_{i3}[$, where the rook is placed. Then $F_4=[a_{12},c_{12}]\cup (\bigcup_{i\in L_p\setminus N_p}([a_{i2}-(1,0),a_{i}])\cup[b_{i},c_{i2}])\cup (\bigcup_{j\in N_p}[a_{j2},c_{j2}]) \cup (\bigcup_{j\in J_t\setminus L_p} [a_{j2}-(1,0),c_{j2}-(1,0)]) \cup (\bigcup_{i\in I_t}[a_{j2},c_{j2}])\cup  [d_{r2}+(1,0),b_{r1}+(1,0)]\cup F_3$ (see Figure \ref{Figure Case 1 (2) B}).

        \begin{figure}[h!]
	\centering	
        \includegraphics[scale=0.53]{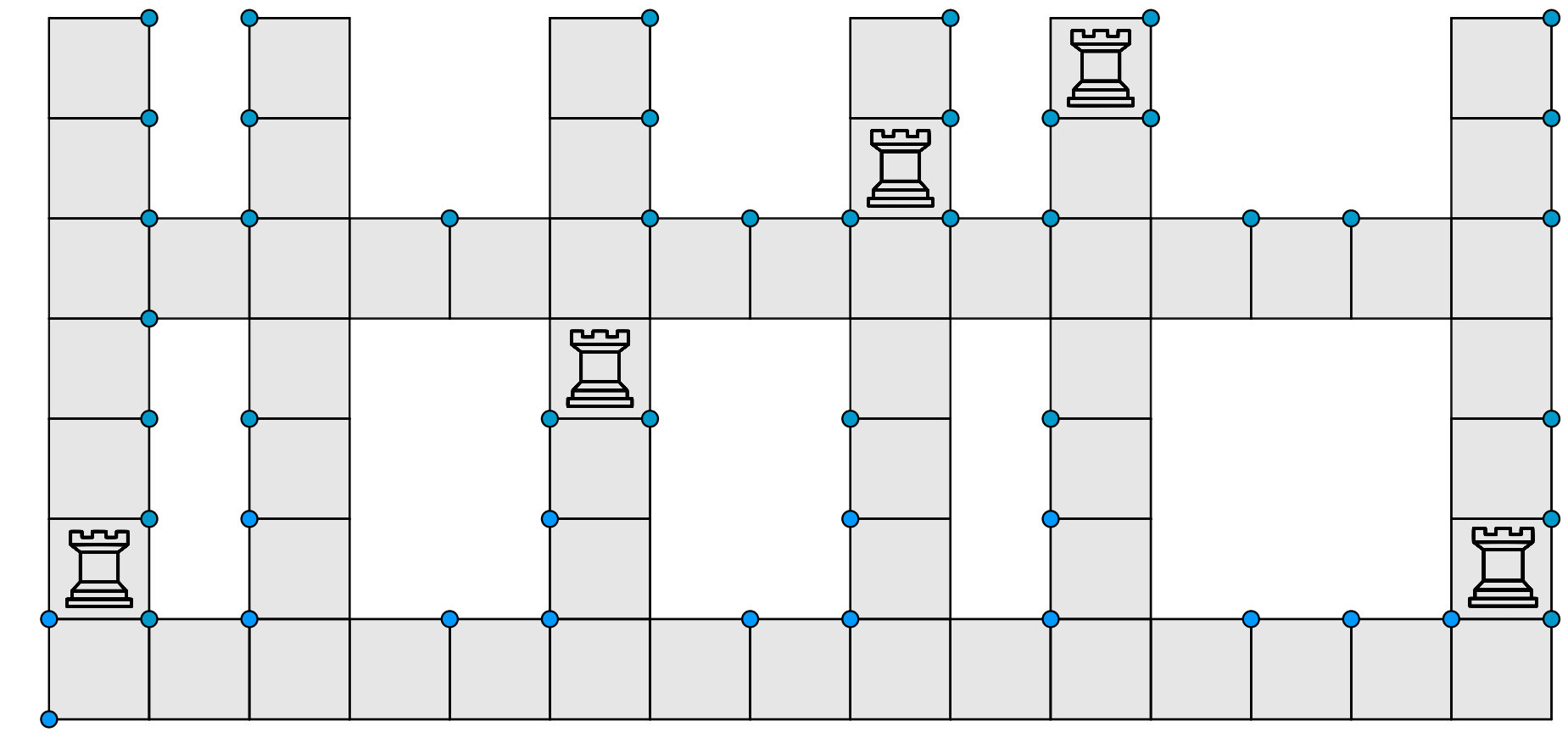}
	\caption{Example of Case 1-(2)-(B).}
	\label{Figure Case 1 (2) B}
        \end{figure}

        \end{enumerate}
        \item Observe that $F_4$ can be defined similarly to the previous two sets for all cases arising from the placement of a rook in $H_1$. For instance, consider the case (2)-(B)-(a)-(i). If we assume that there is no rook in the interval $]A_{i2}, A_{i3}[$ for all $i \in [r+1]$, then $F_4$ is defined as shown in Figure~\ref{Figure Case 1 (3)}.
        \begin{figure}[h!]
	\centering	
        \includegraphics[scale=0.53]{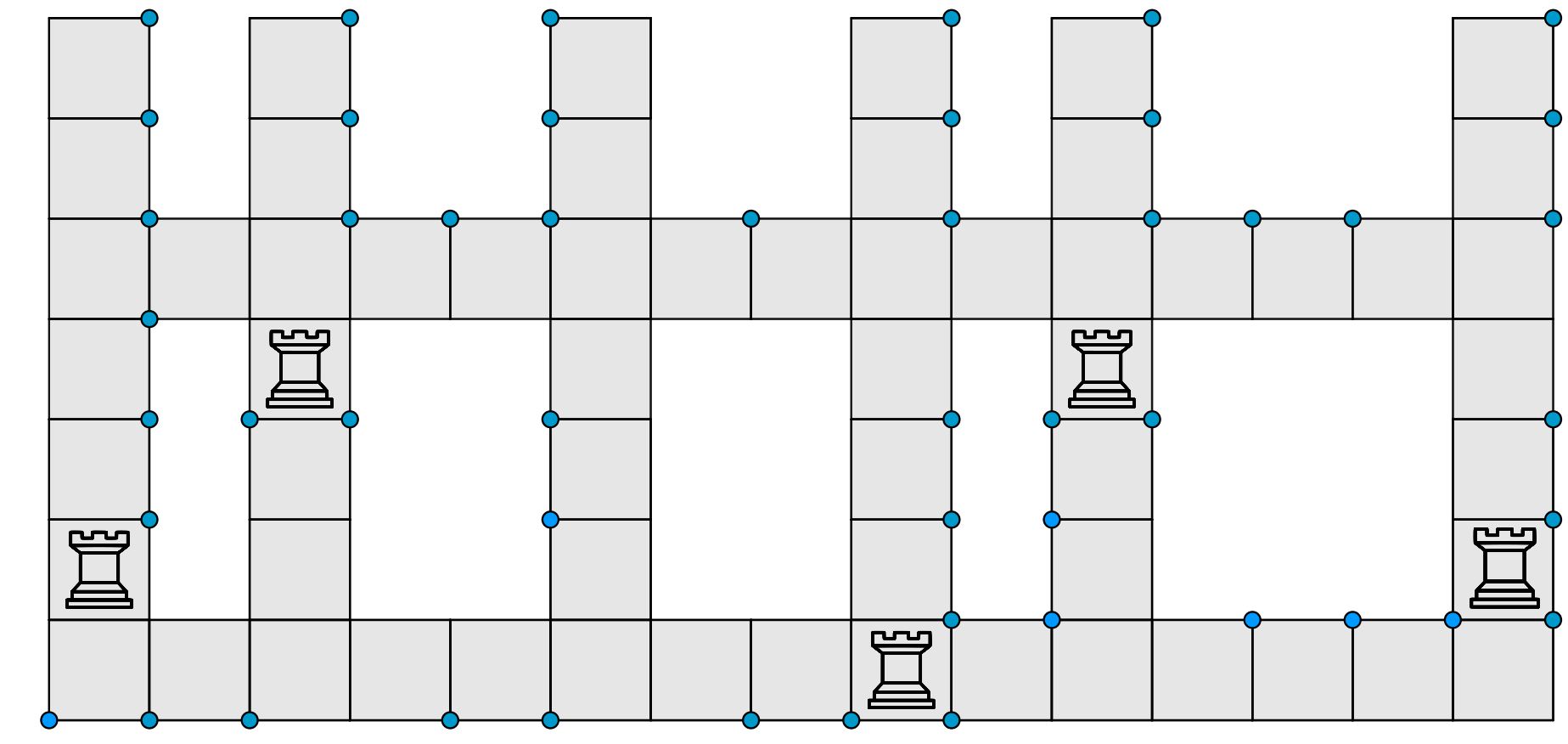}
	\caption{Example of Case 1-(3).}
	\label{Figure Case 1 (3)}
        \end{figure} 
        \end{enumerate}

        \noindent Case 2. Assume that there is a rook in $H_2$. 
        \begin{enumerate}
            \item Assume that we are in the case (1)-(A).  We denote by $c_i$ and $d_i$ respectively the lower right and the upper right corners of the cell of $H_2$ where such a rook is placed. Then $F_3=[a_{12}-(1,1),c_i]\cup [d_i,d_{r2}+(1,0)]$ (see Figure \ref{Figure Case 2 (1)}).
 
        \begin{figure}[h!]
	\centering	
        \includegraphics[scale=0.53]{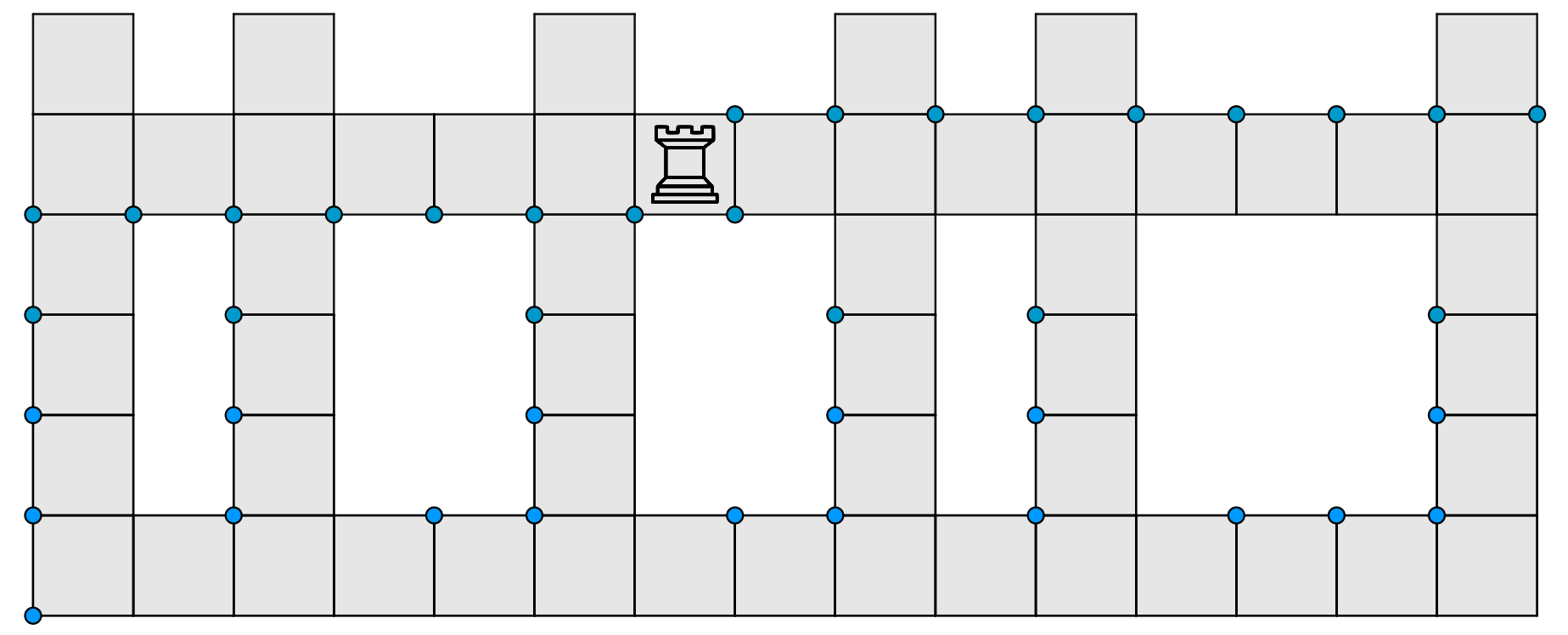}
	\caption{Example of Case 2-(1).}
	\label{Figure Case 2 (1)}
        \end{figure} 
        
        \begin{itemize}
         \item Both when no rook is in $]A_{i2},A_{i3}[$ and when a rook is in $]A_{i2},A_{i3}[$, $F_4$ can be set following a strategy similar to that one provided for $F_2$ in (2)-(A) and (2)-(B). 
        \end{itemize}

        \item Assume that we are in the case (1)-(B). As done before, we suppose that we are in particular in the sub-cases (a), because in the other sub-cases (b), (c), and (d), $F_3$ can be defined similarly. We denote by $c_i$ and $d_i$ respectively the lower right and the upper right corners of the cell of $H_2$ where such a rook is placed. Here we need to distinguish some sub-cases.
         \begin{enumerate}[(A)]
             \item Suppose that a rook is placed in $A_{i2}$ for some $i\in[r]$. One may easily observe that such a rook cannot be placed in $A_{12}$ or $A_{r+1,1}$.
             
        Then $F_3=\big([a_{12}-(0,1),c_i]\setminus\{c_{k2}-(1,0):k\in I_t,k<i\}\big)\cup \big([d_i,d_{r2}+(1,0)]\setminus(\{a_{k2}-(1,0):k\in I_t,k>i\}\cup \{d_{r1}+(1,0)\})\big)\cup F_2$. (see Figure \ref{Figure Case 2 (2) a}).
           
        \begin{figure}[h!]
	\centering	
    \includegraphics[scale=0.53]{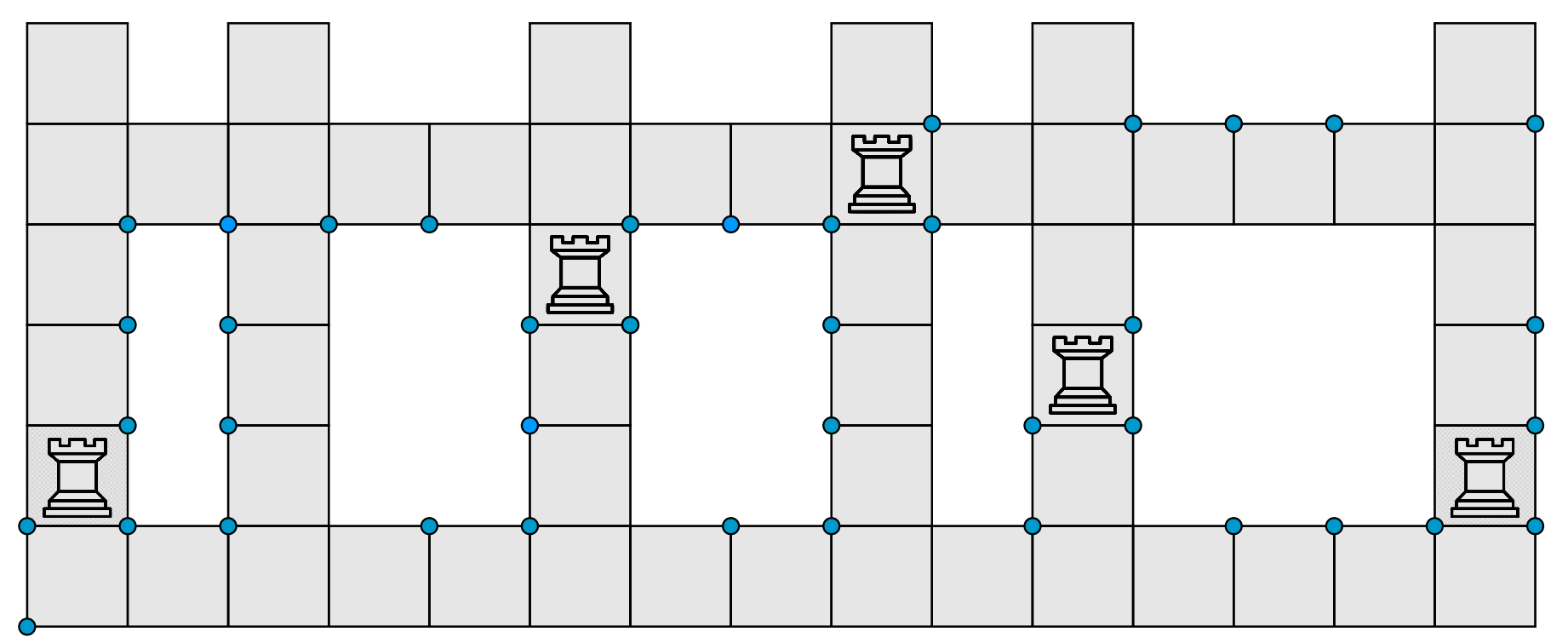}
	\caption{Example of Case 2-(2)-(A).}
	\label{Figure Case 2 (2) a}
        \end{figure}

        \item Suppose that a rook $R_1$ is in $]A_{i2},A_{i+1,2}[$ for some $i\in [r]$. We denote by $L$ and $R$ the cells of $\cP$ which are respectively at East of $A_{i2}$ and at West of $A_{i+1,2}$. Here we need to distinguish some cases in order to define $F_3$. 
        \begin{enumerate}
        \item  Suppose $L\neq R$ and $R_1$ is neither placed in $L$ nor in $R$. We denote by $c_i$ and $d_i$ the lower right and the upper right corners, respectively, of the cell of $H_2$ where $R_1$ is placed. Then $F_3=\big([a_{12}-(0,1),c_i]\setminus\{c_{k1}-(1,0):k\in I_t,k\leq i\}\big)\cup \big([d_i,d_{r2}+(1,0)]\setminus(\{a_{k2}-(1,0):k\in I_t,k\geq i\}\cup \{d_{r1}+(1,0)\})\big)\cup F_2$. (see Figure \ref{Figure Case 2-(2)-(B)-(i)}).
           
           \begin{figure}[h!]
	\centering	
    \includegraphics[scale=0.53]{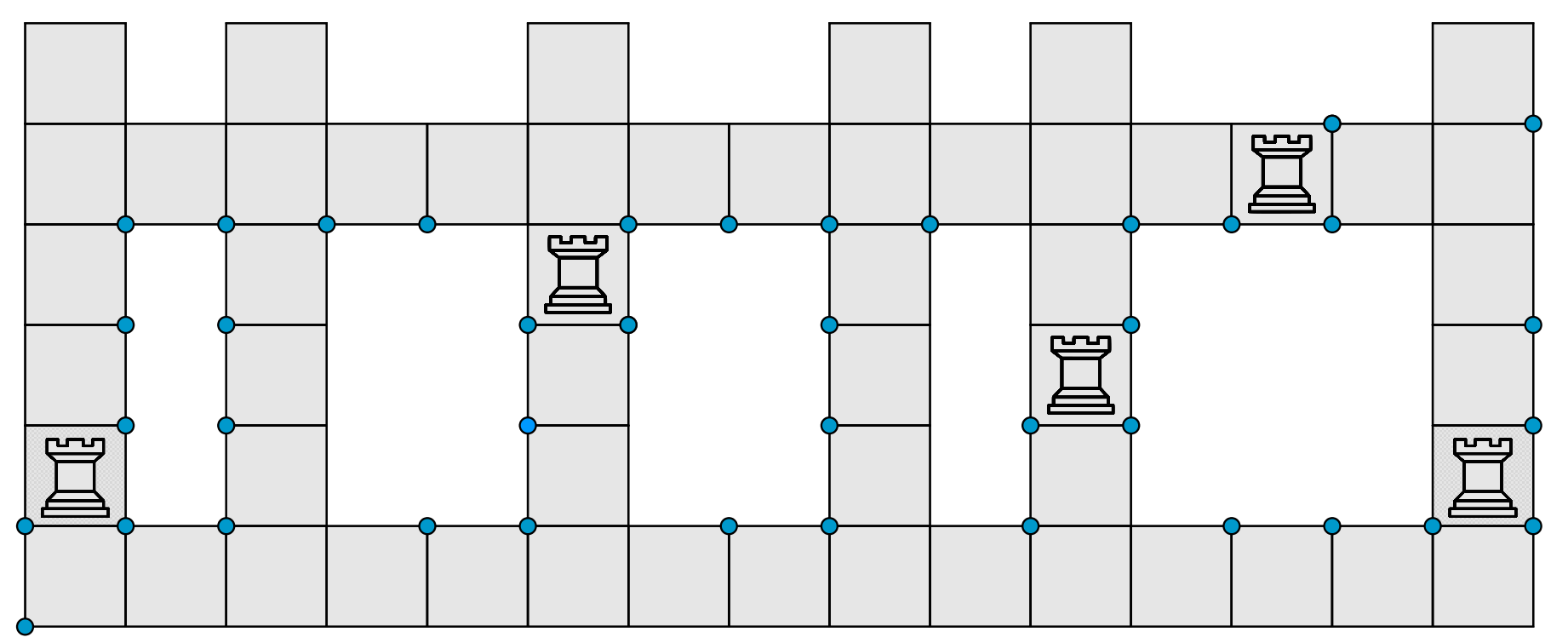}
	\caption{Example of Case 2-(2)-(B)-(i).}
	\label{Figure Case 2-(2)-(B)-(i)}
        \end{figure} 

        \item  Suppose $L\neq R$ and $R_1$ is placed in $R$. We set $F_3$ as done before, as in Figure \ref{Figure Case 2-(2)-(B)-(ii)}.
           
           \begin{figure}[h!]
	\centering	
    \includegraphics[scale=0.53]{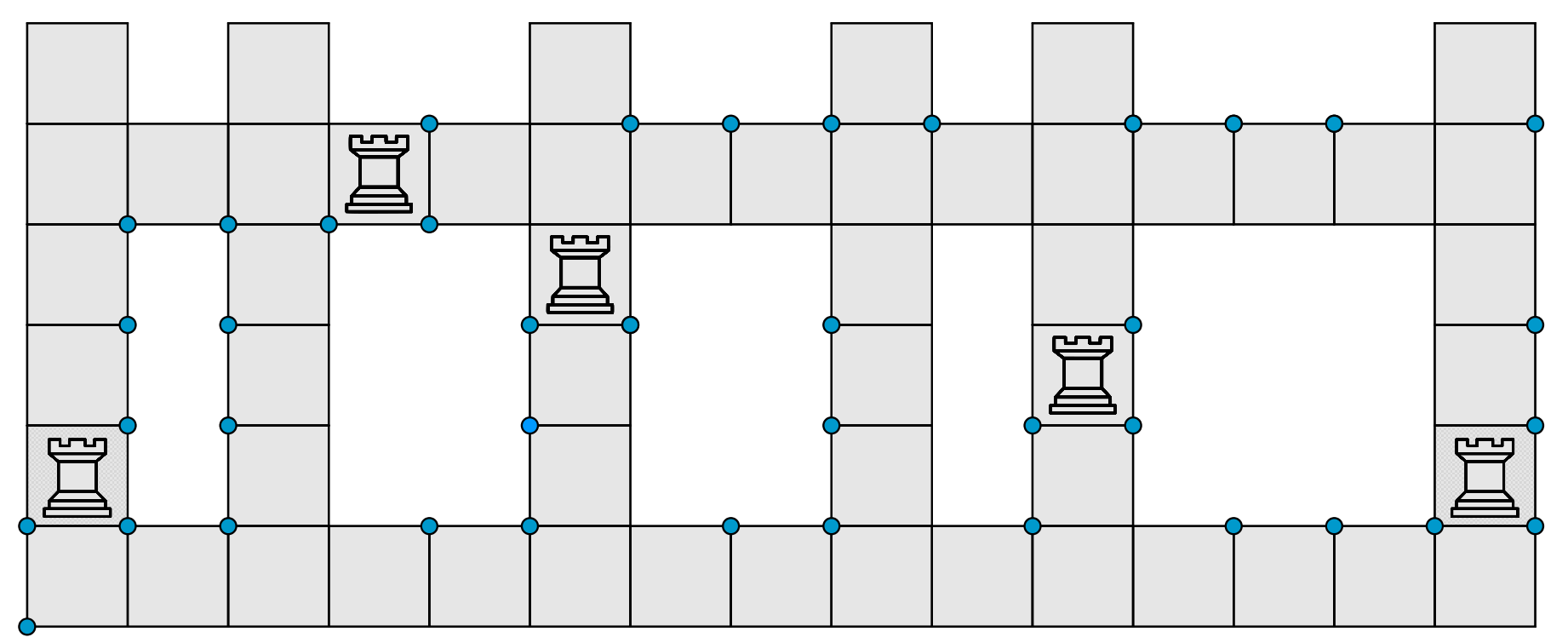}
	\caption{Example of Case 2-(2)-(B)-(ii).}
	\label{Figure Case 2-(2)-(B)-(ii)}
        \end{figure} 

         \item  Suppose $L\neq R$ and $R_1$ is placed in $L$.
         If $i+1\in I_t$, then $F_3=\big([a_{12}-(0,1),b_{i1}+(1,0)]\setminus\{c_{k1}-(1,0):k\in I_t,k\leq i\}\big)\cup \big([b_{i1}+(1,1),d_{r2}+(1,0)]\setminus(\{a_{k2}-(1,0):k\in I_t,k\geq i\}\cup \{d_{r1}+(1,0)\})\big)\cup F_2$ (see Figure \ref{Figure Case 2-(2)-(B)-(iii)}), otherwise $F_3=\big([a_{12}-(0,1),b_{i1}]\setminus\{c_{k1}-(1,0):k\in I_t,k\leq i\}\big)\cup \big([b_{i1}+(0,1),d_{r2}+(1,0)]\setminus(\{a_{k2}-(1,0):k\in I_t,k\geq i\}\cup \{d_{r1}+(1,0)\})\big)\cup F_2$.
           
           \begin{figure}[h!]
	\centering	
    \includegraphics[scale=0.53]{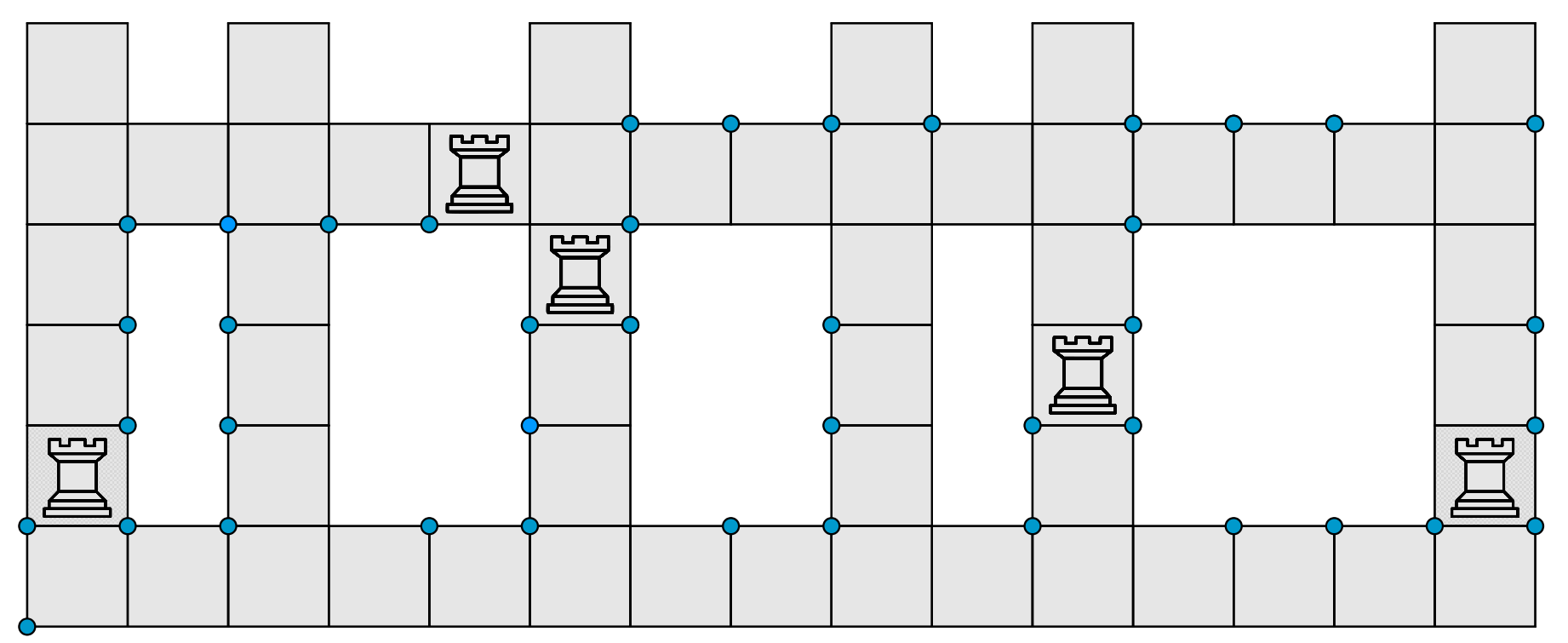}
	\caption{Example of Case 2-(2)-(B)-(iii).}
	\label{Figure Case 2-(2)-(B)-(iii)}
        \end{figure} 
         \item  Suppose $L=R$ and $R_1$ is placed in that cell. We set $F_3$ as done before, as shown in Figure \ref{Figure Case 2-(2)-(B)-(iv)}.
           
           \begin{figure}[h!]
	\centering	
    \includegraphics[scale=0.53]{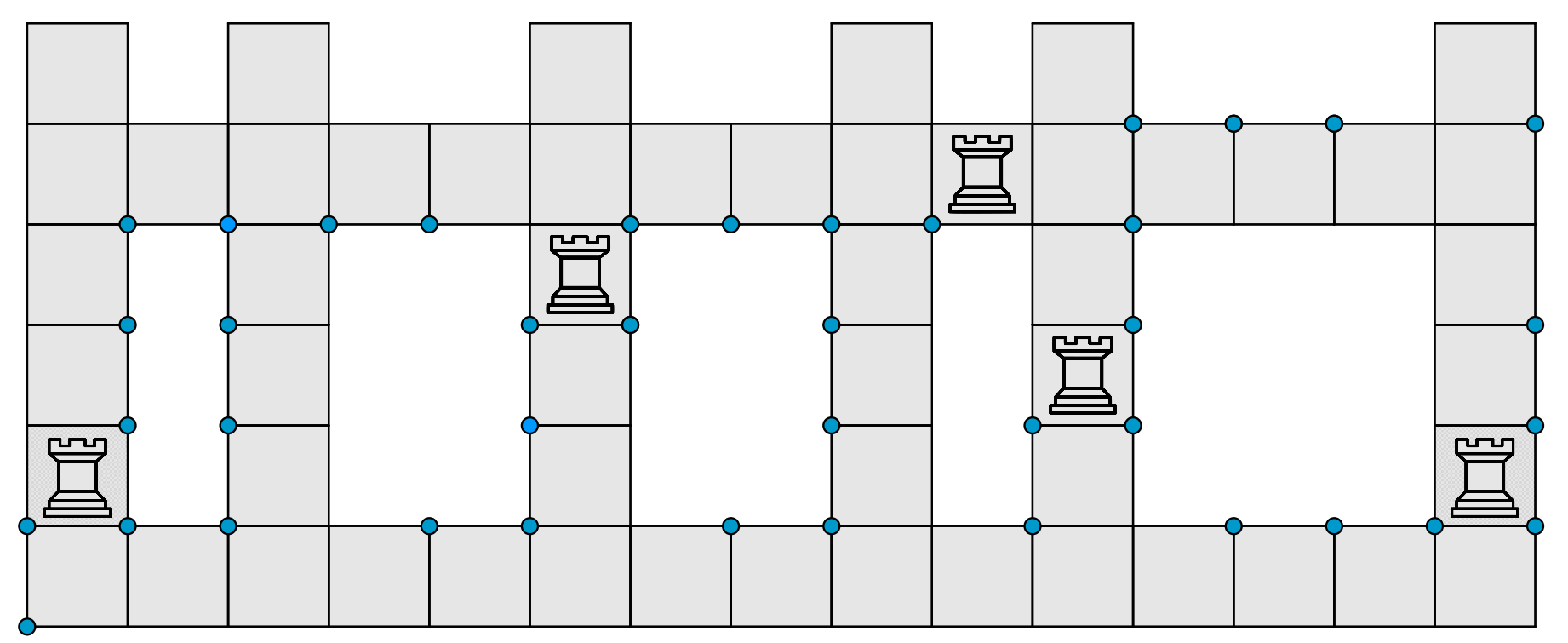}
	\caption{Example of Case 2-(2)-(B)-(iv).}
	\label{Figure Case 2-(2)-(B)-(iv)}
        \end{figure} 
        
        \end{enumerate}
        \end{enumerate}

        \end{enumerate}
    \begin{itemize}
         \item Both when no rook is in $]A_{i2},A_{i3}[$ and when a rook is in $]A_{i2},A_{i3}[$, $F_4$ can be defined following a strategy similar to that one provided for $F_2$ in (2)-(A) and (2)-(B). 
        \end{itemize}
     By continuing this procedure inductively, it becomes clear to the reader how to proceed further, and using this method, the desired facet $F$ can be easily constructed. 
    \end{proof}

       Figure \ref{Figure example rook configuration} presents two examples of facets of $\Delta_{\cP}$ corresponding to rook-configurations in $\cP$. Now, we are ready to prove the main result of this paper, which demonstrates that the $h$-polynomial of the coordinate ring of a grid polyomino $\cP$ equals the rook polynomial of $\cP$.

\begin{thm}\label{mainthm}
		Let $\cP$ be a grid polyomino. Then the rook polynomial of $\cP$ coincides with $h$-polynomial of $K[\cP]$.
	\end{thm}
	
	\begin{proof}
		From Propositions~\ref{Prop: For injective} and \ref{Prop: For surjective}, we deduce that there is a one-to-one correspondence between the facets of $\Delta_{\cP}$ with $k$ generalized steps and the $k$-rook configurations in $\cP$. Moreover, from Proposition \ref{Prop: shellability BH} and Theorem \ref{Thm: The lexicographic order gives a shelling order}, the $k$-th coefficient of the $h$-polynomial of $K[\cP]$ is equal to the number of the facets of  $\Delta_{\cP}$ having $k$ generalized steps. Hence, we get the desired result.
	\end{proof}

 	\begin{coro}\label{coro: mainthm}
		Let $\cP$ be a grid polyomino. Then the regularity of $K[\cP]$ equals the rook number of $\cP$.
	\end{coro}
    
        \begin{proof}
            Since $K[\cP]$ is Cohen-Macaulay by Theorem \ref{dim}, then the regularity of $K[\cP]$ is equal to the degree of the $h$-polynomial of $K[\cP]$, which is given by the rook number by Theorem \ref{mainthm}.
        \end{proof}

        \begin{exa}\rm\label{Exa: rook polynomial}
        By Theorem~\ref{mainthm} and Corollary~\ref{coro: mainthm}, and through computations carried out with \texttt{Macaulay2} (\cite{Package_M2, M2}), for the grid polyominoes illustrated in Figure \ref{Figure example rook configuration}, we obtain that $r(\cP_1)=8$ and $r(\cP_2)=6$ and that the rook polynomials are  
        \[r_{\cP_1}(t)=h_{K[\cP_1]}(t)=1+37\,t+508t^{2}+3324t^{3}+11280 t^{4}+20832t^{5}+20864t^{6}+10496t^{7}+2048t^{8},\]
        \[r_{\cP_2}(t)=h_{K[\cP_2]}(t)=1+27t+261t^{2}+1140t^{3}+2349t^{4}+2187t^{5}+729t^{6}.
        \]
        
         \begin{figure}[h!]
        	\centering	
        	\subfloat[$\cP_1$]{\includegraphics[scale=0.61]{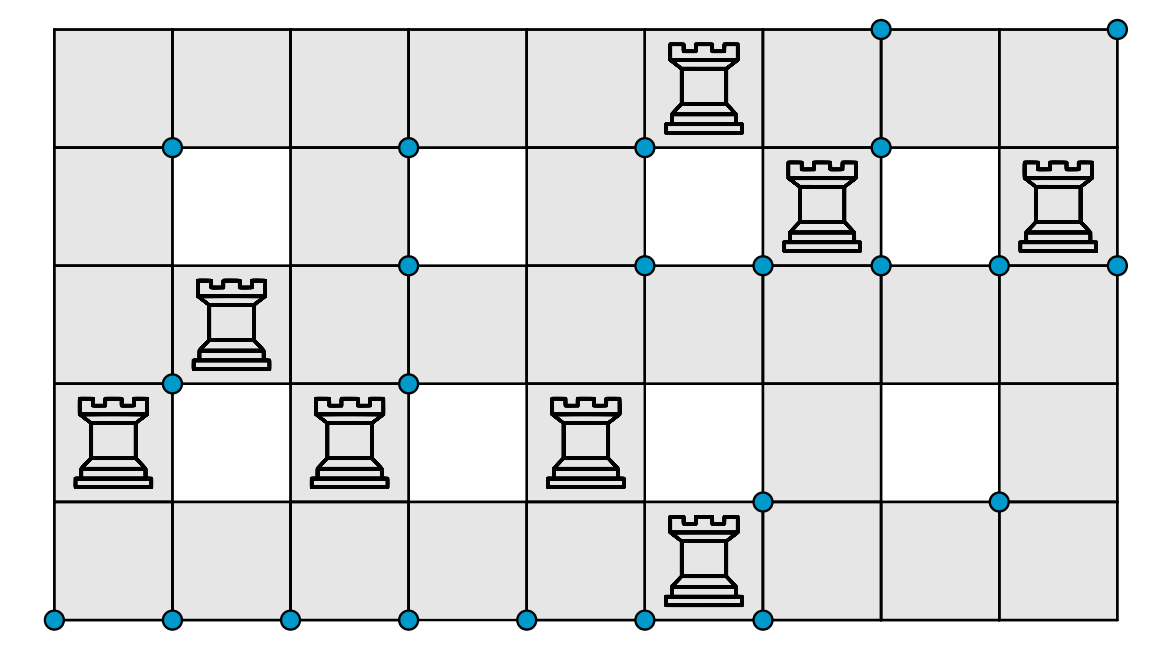}}\qquad
        	\subfloat[$\cP_2$]{\includegraphics[scale=0.61]{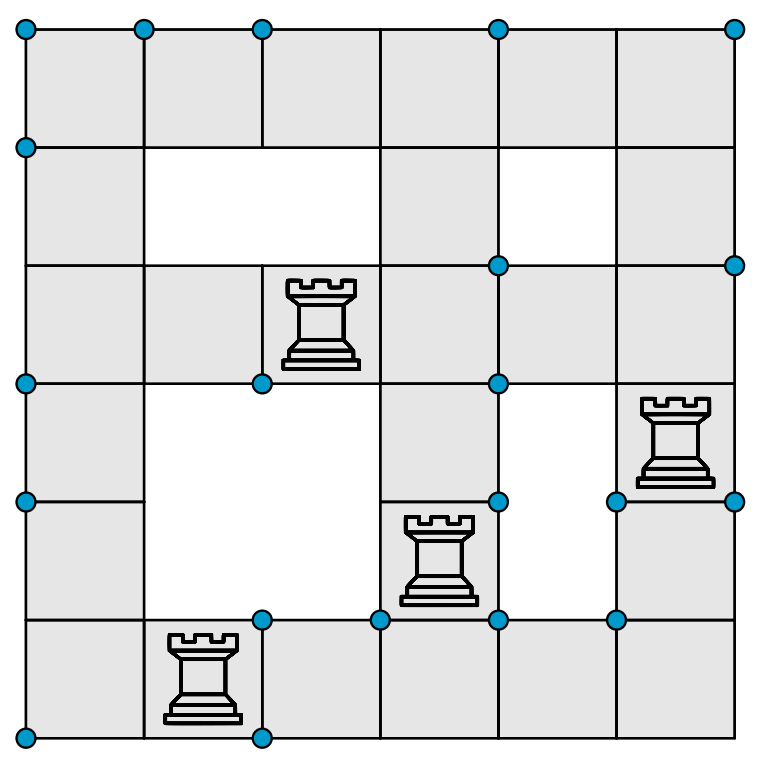}}
        	\caption{Two grid polyominoes.}
        	\label{Figure example rook configuration}
        \end{figure} 
        \end{exa}
    
As mentioned before, the coordinate ring of a grid polyomino is always Cohen-Macaulay. The next natural question is to determine which grid polyominoes have a Gorenstein coordinate ring. The following corollary answers this question.

        \begin{coro}\label{coro: Gor}
            Let $\cP$ be a grid polyomino. Then the following are equivalent:
            \begin{enumerate}
                \item $K[\cP]$ is Gorenstein;
                \item $\cP$ has one hole and consists of four maximal blocks of length three (see Figure \ref{Figure: grid polyominoes} (A)).
            \end{enumerate}
            
        \end{coro}
        \begin{proof}
            $(2)\implies (1).$ If $\cP$ has one hole and consists of four maximal blocks of length three, then $\cP$ is a closed path polyomino without zig-zag walks (see \cite[Section 3]{Cisto_Navarra_closed_path} for details on closed paths and \cite{Trento} for the definition of a zig-zag walk), hence $\cP$ satisfies condition (1) of \cite[Theorem 5.7]{Cisto_Navarra_Hilbert_series}, implying that $K[\cP]$ is Gorenstein.\\
            $(1)\implies (2).$ By contradiction, assume that $\cP$ has a hole but does not consist of four maximal blocks of length three. Then, $\cP$ is a specific type of closed path polyomino, with exactly four changes in direction forming an $L$-configuration and two horizontal or vertical maximal blocks with rank greater than three. According to \cite[Proposition 3.6]{Cisto_Navarra_closed_path}, $\cP$ does not have zig-zag walks, meaning that $\cP$ is a closed-path polyomino without zig-zag walks. Therefore, by \cite[Theorem 5.7]{Cisto_Navarra_Hilbert_series}, $K[\cP]$ should not be Gorenstein, which contradicts condition (1). Now, assume that $\cP$ has more than one hole. Recall that by Theorem \ref{dim}, $K[\cP]$ is a Cohen-Macaulay domain. Let $h(t) = \sum_{i=0}^s h_i t^i$ be the $h$-polynomial of $K[\cP]$, where $s$ is the rook number of $\cP$ by Corollary \ref{coro: mainthm}. If $K[\cP]$ is Gorenstein, a well-known result by Stanley (\cite{Stanley}) implies that $h_i = h_{s-i}$ for all $i \in \{0, \dots, s\}$, and in particular, we have $h_s = 1$. By Corollary \ref{coro: mainthm}, $h(t)$ is the rook polynomial of $\cP$, so $h_s = 1$ would imply that there is a unique configuration for placing the maximum number of rooks in a non-attacking position on $\cP$. We will now show that this is not the case, i.e., that $h_s > 1$. We may assume that the cells of $\cP$ are in the configuration shown in Figure~\ref{Figure: for Gorenstein}; otherwise, it is sufficient to apply a suitable rotation. Let $\cR$ be an $s$-rook configuration in $\cP$. Note that a rook $R$ in $\cR$ cannot be in the cell $D$ shown in Figure \ref{Figure: for Gorenstein}. If it did, we could remove $R$ from $\cR$ and add two rooks in the cells $A$ and $C$, obtaining an $(s+1)$-rook configuration in $\cP$, which contradicts the assumption that $s$ is the rook number. Furthermore, observe that a rook in $\cR$ must be placed in the maximal horizontal interval of $\cP$ containing $A$. Otherwise, as described earlier, we could add a rook in $A$ to $\cR$, obtaining an $(s+1)$-rook configuration, which again leads to a contradiction. Therefore, we can assume that there is a rook in $\cR$ placed in $A$. If we move this rook from $A$ to $B$, we obtain a new $s$-rook configuration in $\cP$, which implies that $h_s > 1$, as we have claimed.

              \begin{figure}[h!]
	\centering	
    \includegraphics[scale=0.56]{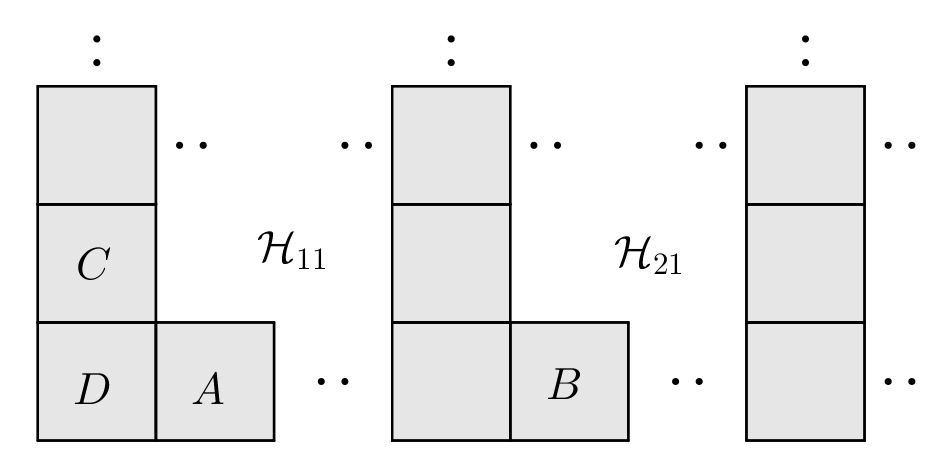}
	\caption{More than one hole in a grid polyomino.}
	\label{Figure: for Gorenstein}
        \end{figure} 
        \end{proof}

        \section*{Acknowledgement}  RD was supported by the Alexander von Humboldt Foundation and by a grant of the Ministry of Research, Innovation and Digitization, CNCS - UEFISCDI, project number PN-III-P1-1.1-TE-2021-1633, within PNCDI III. FN  is supported by The Scientific and Technological Research Council of Turkey - T\"UBITAK (Grant No: 122F128) and he is enrolled in the group GNSAGA of INDAM. This work was started at the Institute of Mathematics of the Romanian Academy in Bucharest when the second author visited the first one. He wants to express his gratitude for her hospitality and her nice support. RD and FN are sincerely grateful to the reviewers for carefully reading the article.\\

        \begin{small}
        \noindent {\bf Declaration of competing interest.} The authors declare that they have no known competing financial interests or personal relationships that could have appeared to influence the work reported in this paper.\\
        {\bf Data Availability Statement.} No data was used for the research described in the article.  
        \end{small}

\end{document}